\documentclass[reqno,11pt,a4paper]{amsart}
\usepackage{amssymb,amsmath,amsthm}

\usepackage[colorlinks=true, pdfstartview=FitV, linkcolor=blue, citecolor=blue, urlcolor=blue,pagebackref=false]{hyperref}
\usepackage{esint}
\usepackage{mathrsfs}
\usepackage{comment}
\usepackage{enumitem}

\usepackage{times,latexsym,color}

\newcommand{\BOX}{\ensuremath\Box}

\newtheorem{theorema}{Theorem}

\newtheorem{theorem}{Theorem}[section]
\newtheorem{proposition}[theorem]{Proposition}
\newtheorem{lemma}[theorem]{Lemma}
\newtheorem{corollary}[theorem]{Corollary}

\theoremstyle{remark}
\newtheorem{remark}[theorem]{Remark}

\theoremstyle{definition}
\newtheorem{definition}[theorem]{Definition}

\numberwithin{equation}{section}

\DeclareMathOperator{\dist}{dist}

\DeclareMathOperator*{\osc}{osc}

\newcommand{\N}{\mathbb{N}}

\newcommand{\R}{\mathbb{R}}
\newcommand{\Z}{\mathbb{Z}}
\newcommand{\dd}{\,{\rm d}}

\newcommand{\ep}{\varepsilon}

\newcommand{\B}{\mathbb{B}}

\def\Xint#1{\mathchoice
	{\XXint\displaystyle\textstyle{#1}}%
	{\XXint\textstyle\scriptstyle{#1}}%
	{\XXint\scriptstyle\scriptscriptstyle{#1}}%
	{\XXint\scriptscriptstyle\scriptscriptstyle{#1}}%
	\!\int}
\def\XXint#1#2#3{{\setbox0=\hbox{$#1{#2#3}{\int}$}
		\vcenter{\hbox{$#2#3$}}\kern-.5\wd0}}

\def\dashint{\Xint-}

\definecolor{darkgreen}{rgb}{0,0.5,0}
\definecolor{darkblue}{rgb}{0,0,0.7}
\definecolor{darkred}{rgb}{0.9,0.1,0.1}
\definecolor{lightblue}{rgb}{0,0.51,1}

\newenvironment{proofx}[1]%
{\vskip\baselineskip\noindent\textbf{Proof of {#1}:}}%
{\hspace*{.1pt}\hspace*{\fill}\BOX\vskip\baselineskip}
{\vskip\baselineskip\noindent\textbf{Proof of Theorem \protect\ref{#1}:}}%
{\hspace*{.1pt}\hspace*{\fill}\BOX\vskip\baselineskip}
{\vskip\baselineskip\noindent\textbf{Proof of Lemma~\protect\ref{#1}:}}%
{\hspace*{.1pt}\hspace*{\fill}\BOX\vskip\baselineskip}
{\vskip\baselineskip\noindent\textbf{Proof of Proposition~\protect\ref{#1}:}}%
{\hspace*{.1pt}\hspace*{\fill}\BOX\vskip\baselineskip}
{\vskip\baselineskip\noindent\textbf{Proof of Theorems \protect\ref{#1} --
		\protect\ref{#2}:}}%
{\hspace*{.1pt}\hspace*{\fill}\BOX\vskip\baselineskip}

\begin{document}
	
	\title[Large-scale regularity for fluids over non-Lipschitz boundaries]{Large-scale regularity for the stationary Navier-Stokes equations over non-Lipschitz boundaries}
	
	\author[M. Higaki]{Mitsuo Higaki}
	\address[M. Higaki]{Department of Mathematics, Kobe University, 1-1 Rokkodai, Nada-ku, Kobe 657-8501, Japan}
	\email{higaki@math.kobe-u.ac.jp}
	
	\author[C. Prange]{Christophe Prange}
	\address[C. Prange]{Laboratoire de Math\'ematiques AGM, UMR CNRS 8088, Cergy Paris Universit\'e, 
		2 avenue Adolphe Chauvin, 95302 Cergy-Pontoise Cedex, France}
	\email{christophe.prange@cyu.fr}
	
	\author[J. Zhuge]{Jinping Zhuge}
	\address[J. Zhuge]{Department of Mathematics, University of Chicago, Chicago, IL, 60637}
	\email{jpzhuge@math.uchicago.edu}
	
	
	\maketitle
	
	\begin{abstract}
		In this paper we address the large-scale regularity theory for the stationary Navier-Stokes equations in highly oscillating bumpy John domains. These domains are very rough, possibly with fractals or cusps, at the microscopic scale, but are amenable to the mathematical analysis of the Navier-Stokes equations. We prove: (i) a large-scale Calder\'on-Zygmund estimate, (ii) a large-scale Lipschitz estimate, (iii) large-scale higher-order regularity estimates, namely, $C^{1,\gamma}$ and $C^{2,\gamma}$ estimates. These nice regularity results are inherited only at mesoscopic scales, and clearly fail in general at the microscopic scales. We emphasize that the large-scale $C^{1,\gamma}$ regularity is obtained by using first-order boundary layers constructed via a new argument.
		The large-scale $C^{2,\gamma}$ regularity relies on the construction of second-order boundary layers, which allows for certain boundary data with linear growth at spatial infinity. To the best of our knowledge, our work is the first to carry out such an analysis. 
		In the wake of many works in quantitative homogenization, our results strongly advocate in favor of considering the boundary regularity of the solutions to fluid equations as a multiscale problem, with improved regularity at or above a certain scale.

	\end{abstract}

	\section{Introduction}

	In this work we consider the large-scale boundary regularity for the stationary Navier-Stokes equations
	\begin{equation}\tag{NS$^\ep$}\label{intro.NS.ep}
		\left\{
		\begin{array}{ll}
			-\Delta u^\ep+\nabla p^\ep=-u^\ep\cdot\nabla u^\ep&\mbox{in}\ B^\ep_{1,+}\\
			\nabla\cdot u^\ep=0&\mbox{in}\ B^\ep_{1,+}\\
			u^\ep=0&\mbox{on}\ \Gamma^\ep_1,
		\end{array}
		\right.
	\end{equation}
	in a domain with a rough bumpy boundary. The no-slip boundary condition is prescribed only on the lower part $\Gamma^\ep_1$ of $\partial B^\ep_{1,+}$. The boundary is rough in two aspects: (i) possible lack of regularity at the microscopic scale as the boundary may have fractals or inward cusps; (ii) bumpiness, i.e., the boundary is highly oscillating. 
	The functions $u^\ep = (u^\ep_1(x), u^\ep_2(x), u^\ep_3(x)) \in\R^3$ and $p^\ep=p^\ep(x)\in\R$ denote respectively the velocity field and the pressure field of the fluid. The definitions of $B^\ep_{r,+}$ and $\Gamma^\ep_r$ are given in Subsection \ref{sec.not}. 
	We will show large-scale regularity estimates, including Lipschitz estimate (see Theorem \ref{theo.lip.nonlinear} in Subsection \ref{sec.outline} below), $C^{1,\gamma}$ estimate (see Theorem \ref{thm.C1g}) and $C^{2,\gamma}$ estimate  (see Theorem \ref{thm.C2g}).
		These improved regularity results at the large scales are generally false at the small scales due to the roughness of the boundary. The tools developed in this paper enable us to decouple the large-scale regularity from the small-scale properties of the boundary. Therefore, our results (i.e., Theorems \ref{theo.lip.nonlinear}, \ref{thm.C1g} and \ref{thm.C2g}) show that stationary incompressible Newtonian fluids are regular above the microscopic scale, regardless of the irregularity of surfaces at the microscopic scale.

	Before going into the details of our results and of the mathematical analysis, let us give some more general perspectives. The study of fluids over rough boundaries plays a prominent role in the field of hydrodynamics, at least for three reasons.
	
	First, rough, bumpy or corrugated surfaces are ubiquitous in nature and engineering. They appear at any scales from geophysics (see for instance \cite{narteau2001} for the fractal-like core-mantle boundary in the Earth) to zoology \cite{pu2016} and microfluidics \cite{waheed2016}. At the microstructure, the geometry may be anything from fractal to periodic and crenellated. No surface is perfectly smooth, and the lack of smoothness may actually enable us to resolve certain oddities, such as the no-collision paradox for a sphere dropped in a viscous fluid under the action of gravity \cite{smart1989,joseph2001,davis2003,GVH12,izard2014}. Moreover, certain roughness patterns are either selected by biological processes and environmental pressure such as scales of sharks for their drag reduction properties, or designed for industrial applications especially in aeronautics, microfluidics and for the transport of fluids in pipes \cite{pu2016,dean2010,lee2005}.

	Second, the study of roughness is strongly tied to the derivation of boundary conditions in fluid mechanics. The question whether fluids slip or not over surfaces is still a matter of active debate. Experiments show that there is no universal answer and that the slip behavior depends a lot on the geometry and microstructure of the surface \cite{bocquet2007,lauga2007}. A widespread idea is that roughness favors slip. To give one specific example where finding the most accurate boundary condition is critical, let us cite the field of glaciology. The assessment of various friction laws for the flow of a glacier over a rough bedrock is crucial in order to understand the speed of glacier discharge and eventually estimate the sea level rise as a result of global warming \cite{joughin19,minchew20}.
	
	Third, the study of the impact of roughness on the behavior of fluids accompanied the development of  turbulence research, as underlined by Jim\'enez in \cite{jimenez2004}:
	\begin{quote}
		Turbulent flows over rough walls have been studied since the  early  works  of Hagen (1854) and  Darcy  (1857), who were  concerned  with  pressure  losses  in water conduits. They have been important in the history of turbulence. Had those conduits not been fully rough, turbulence theory would probably have developed more slowly. The pressure loss in pipes only becomes independent of viscosity in the fully rough limit, and this independence was the original indication that something was amiss with laminar theory. Flows over smooth walls never become fully turbulent, and their theory is correspondingly harder.
	\end{quote}
	Investigations of the effect of roughness on fluid flows span three distinct regimes. In the laminar regime, studies focus on the drag reducing properties of roughness elements \cite{bechert1989,GMJ11}. As for the onset of turbulence \cite{schultz2007,squire16}, there are some indications that roughness lowers the critical Reynolds number for the transition from the laminar to turbulent regime \cite{varnik2007}. In the fully turbulent regime, a similarity hypothesis for the flow over flat surfaces and for the flow over rough surfaces was put forward \cite{townsend1980}. The extent to which such a universal law holds is still being disputed \cite{jimenez2004,castro2007,flack2007,schlichting2016}.

	The three main directions raised above are reflected in the mathematical works. The literature is vast. Therefore we do not aim for exhaustivity. 
	
	First, there is an extensive body of works that deal with wall (or friction) laws, or in other words, effective or homogenized boundary conditions. One aims at replacing rough boundaries by fictitious, smooth or flat boundaries. In that line of research, it is well-known that Navier-slip boundary conditions provide refined approximations for fluids above bumpy boundaries. Under some quantitative ergodicity assumptions, one can get error estimates. Historically, periodic roughness profiles were first looked at \cite{AS97,APV98,JM01,JM03}. Analysis of almost-periodic \cite{GVM10} or random stationary ergodic \cite{GV09,BGV08} boundary oscillations was done more recently. Let us also mention a few works that address non-stationary fluids \cite{BFN10,Higaki16}, for which the analysis is less developed due to its inherent difficulties. 
	We also point out that some authors attempted to justify boundary conditions arising in fluid mechanics starting from boundary conditions at the microscopic scale; see for instance \cite{CDFS03,BFNW08,BB12} for the derivation of the no-slip boundary condition from a perfect slip boundary condition at the microscale, or \cite{DGV11} for the computation of the homogenized effect starting from Navier-slip boundary conditions at the microscale.
	
	A second topic is the study of the effect of roughness on singular limits. The topics of rotating fluids and of the homogenized effect of bumpiness on Ekman pumping has been studied in numerous papers \cite{GV03,GVD06,DP14,DGV17}. The paper \cite{GVLNR18} carries out an analysis of the vanishing viscosity limit in a specific scaling regime. There are also studies concerned with equations in singularly perturbed domains such as the Stokes equations in rough thin films \cite{CM12} or water waves above a rough topography in the shallow regime \cite{CLS12}.
	
	Third, rough domains pose considerable numerical challenge. This aspect has certainly driven the development of wall laws in a model reduction perspective; see for instance \cite{APV98,DDM15}. Other approaches are being elaborated, such as Direct Numerical Simulations \cite{cardillo13}, Lattice Boltzmann Methods that are adapted to intricate geometries \cite{varnik2007} and Large Eddy Simulations \cite{anderson2011,BSG18} that in this context cause important parametrization issues of the small scales.

	In this work, we tackle these questions from the angle of regularity theory. The following two general objectives in regularity theory motivate our results: (i) identify building blocks describing the local behavior of solutions, (ii) estimate the decay of certain excess quantities at various scales. We prove that fluids above bumpy boundaries, that are very rough at the microscopic scale, have improved regularity at large scales. 
	Our results are in the spirit of large-scale regularity estimates pioneered in \cite{AL87} for periodic homogenization, and later extended to stochastic homogenization; see for instance \cite{AS16,AM16,GNO14,GNO15} and \cite{AKM16} for the higher-order large-scale regularity theory. Our research program was started with the works \cite{KP15,KP18} concerned with uniform regularity estimates above highly oscillating boundaries for elliptic equations. In the paper \cite{HP}, the large-scale Lipschitz and $C^{1,\gamma}$ estimates for the stationary Navier-Stokes equations were established above Lipschitz boundaries. A local Navier wall law was also obtained. Finally, let us also mention the paper \cite{Z} that deals with the large-scale regularity of elliptic equations above arbitrarily rough microstructures.

	\subsection{Outline of the main results of the paper}
	\label{sec.outline}
	
	In this paper, we study the large-scale regularity for stationary incompressible viscous fluids modeled by the Stokes or Navier-Stokes equations, in domains that are very rough and bumpy at the microscale. Our results show that the large-scale regularity is completely independent of small-scale properties of the boundary. 
	
	Let us stress some novel aspects of our results. We refer to Subsection \ref{sec.comp} for a further comparison with a few related works, and to Subsection \ref{sec.strat} for an outline of the proofs.
	
	First we consider John domains, whose boundaries allow for fractals and inward cusps. Hence, the boundaries considered in this paper get closer to the modeling of real boundaries found in nature, that in particular do not need to be graphs.
	John domains have in a broad sense the minimal properties for the analysis of incompressible fluids. Indeed, we rely on a Bogovskii operator in John domains to estimate the pressure. For precise definitions and a more complete discussion, we refer to Subsection \ref{sec.not} below.
	
	Second, beyond the Lipschitz estimate, we prove higher-order $C^{1,\gamma}$ and $C^{2,\gamma}$ estimates for $\gamma\in[0,1)$, as stated in Theorems \ref{thm.C1g} and \ref{thm.C2g}. These require the construction of boundary layer correctors, which is at the heart of the paper in Section \ref{sec.bl}; see Subsection \ref{subsec.1st.BL} for the first-order boundary layers and Subsection \ref{subsec.2nd.BL} for the second-order boundary layers. As far as we know, the present work is the first to construct the second-order boundary layers with a linear growth in the direction tangential to the boundary. To make the analysis more tractable, we assume that the boundary is periodic for the structure result of second-order boundary layers; see Theorem \ref{prop.BL21} and Theorem \ref{prop.BL2356}. We are aware of the papers \cite{BLTV,BM10}, where a refined second-order approximation is constructed for the Stokes equations in a two-dimensional rough channel. However, the boundary layers considered in \cite{BLTV,BM10} only involve data spanned by linear and quadratic polynomials of the vertical variable, $x_2$ and $x_2^2$ in this two-dimensional case, which are bounded on the bumpy boundary. In our three-dimensional situation, the class of ``no-slip Stokes polynomials'' (see Subsection \ref{subsec.Taylor.polys}) is much richer and involves boundary data with linear growth at spatial infinity.

	Third, we provide explicit quantitative regularity estimates in the non-perturbative regime. 
	
	Fourth, in the vein of the seminal works \cite{AL87,AL91} and of \cite{KLS14,GZ19}, we provide pointwise estimates for the large-scale decay of the velocity and pressure parts of the Green function associated to the Stokes system in bumpy John half-spaces; see Subsection \ref{sec.strat} and Appendix \ref{app.green}. These estimates are pivotal to construct the first-order boundary layers in Subsection \ref{subsec.1st.BL}.
	
	We now state the three main theorems of the paper.
		
	%
	\begin{theorema}[Large-scale Lipschitz regularity]\label{theo.lip.nonlinear}
		For all $\ep\in(0,\frac12)$, 
		$L\in (0,\infty)$, $M\in(0,\infty)$ and $\delta\in(0,1)$, the following statement holds. Let $\Omega$ be a bumpy John domain with constant $L$ according to Definition \ref{def.John2} below. If $(u^\ep, p^\ep) \in H^1(B^\ep_{1,+})^3 \times L^2(B^\ep_{1,+})$ is a weak solution of \eqref{intro.NS.ep} satisfying
		\begin{align}\label{e.defMbis}
			\bigg(\dashint_{B^\ep_{1,+}} |\nabla u^\ep|^2 \bigg)^{1/2}
			\le
			M
		\end{align}
		(The precise definition of the bumpy cube $B^\ep_{r,+} = Q_r(0) \cap \Omega^\ep$ can be found in Subsection \ref{sec.not}.),
		then, for any $r\in(\ep, \frac12)$,
		\begin{align}\label{est2.theo.lip.nonlinear}
			\bigg(\dashint_{B^\ep_{r,+}} |\nabla u^\ep|^2 \bigg)^{1/2}
			+ \bigg(\dashint_{B^\ep_{r,+}} |p^\ep-\dashint_{B^\ep_{1/2,+}} p^{\ep}|^2\bigg)^{1/2}
			\le
			C(M+M^{4+\delta}),
		\end{align}
		where the constant $C$ is independent of $\ep$, $M$ and $r$, and depends on $L$ 
		and $\delta$.
	\end{theorema}
	%

	Notice that Theorem \ref{theo.lip.nonlinear}, as well as the subsequent results, holds in the non-perturbative regime for arbitrarily large $M$ in \eqref{e.defMbis}. This is due to the energy subcritical nature of the stationary Navier-Stokes equations, which makes it an easier problem than the non-stationary Navier-Stokes system. Remark also that the powers of $M$ in the right-hand side of \eqref{est2.theo.lip.nonlinear} are explicit.

	For higher-order $C^{1,\gamma}$ and $C^{2,\gamma}$ regularity results, we measure the oscillation of the solution with respect to modified polynomials that vanish on the bumpy boundary. These modified polynomials are polynomials of degree one and two that are corrected by the first-order and second-order boundary layers. 
	
	\begin{theorema}[Large-scale $C^{1,\gamma}$ regularity]\label{thm.C1g}
		For all $\gamma\in [0,1)$, 
		$\ep\in(0, 
		\frac12)$, $L\in (0,\infty)$, $M\in(0,\infty)$ and $\delta\in(0,1)$, the following statement holds. Let $\Omega$ be a bumpy John domain with constant $L$ according to Definition \ref{def.John2} below. If $(u^\ep, p^\ep) \in H^1(B^\ep_{1,+})^3 \times L^2(B^\ep_{1,+})$ is a weak solution of \eqref{intro.NS.ep} satisfying \eqref{e.defMbis},
		then, there exists a constant $\bar{P}_1$ (depending on $p^\ep$) such that, for any $r\in (\ep, \frac12)$, 
		\begin{equation}\label{est.thmC1a}
			\begin{aligned}
				&\inf_{(w,\pi)\in \mathscr{Q}_1(\Omega)}  \bigg\{ \frac{1}{r} \bigg( \dashint_{B_{r,+}^\ep} |u^\ep - \ep w(x/\ep)|^2 \dd x \bigg)^{1/2} + \bigg( \dashint_{B^\ep_{r,+}} | p^\ep - \pi(x/\ep) - \bar{P}_1 |^2\dd x \bigg)^{1/2} \bigg\} \\
				& \quad \le Cr^\gamma (M+M^{4+2\gamma+\delta}),
			\end{aligned}
		\end{equation}
		where $\mathscr{Q}_1(\Omega)$ is the class of all solutions to the Stokes equations in a bumpy John half-space $\Omega$ with linear growth at infinity that vanish on $\partial\Omega$; see \eqref{e.defQ1}. The constant $C$ is independent of $\ep$, $M$ and $r$, but depends on $L$, $\gamma$ and $\delta$.
	\end{theorema}


	The velocity estimate in (\ref{est.thmC1a}) will be derived via a large-scale estimate of $|\nabla u_\ep - \nabla w(x/\ep)|$ and the Poincar\'{e} inequality; see Subsection \ref{subsec.Large-scale.1}.

	While Theorem \ref{thm.C1g} holds for arbitrary bumpy John half-spaces, for the next result, we work in periodic John domains. As we outlined above, the extra periodicity assumption makes the analysis of the second-order boundary layers more manageable.

	\begin{theorema}[Large-scale $C^{2,\gamma}$ regularity]\label{thm.C2g}
		For all $\gamma\in [0,1)$, 
		$\ep\in(0,\frac12)$, $L\in(0,\infty)$, $M\in(0,\infty)$ and $\delta\in(0,1)$, the following statement holds. Let $\Omega$ be a periodic bumpy John domain with constant $L$ according to Definition \ref{def.PeriodicJohn} below. If $(u^\ep, p^\ep) \in H^1(B^\ep_{1,+})^3 \times L^2(B^\ep_{1,+})$ is a weak solution of \eqref{intro.NS.ep} satisfying \eqref{e.defMbis},
		then, there exists a constant $\bar{P}_{2}$ (depending on $p^\ep$) such that, for any $r\in (\ep
		, \frac12)$, 
		\begin{equation}\label{est.C2gamma}
			\begin{aligned}
				\inf\limits_{\substack{ (w_1,q_1) \in \mathscr{Q}_1(\Omega) \\ (w_2,q_2) \in \mathscr{Q}_2(\Omega)} } & \Bigg\{  \frac{1}{r} 
				\bigg(\dashint_{B^\ep_{r,+}} |u^\ep - \ep w_1(x/\ep) - \ep^2 w_2(x/\ep) |^2\dd x \bigg)^{1/2} \\
				&\qquad
				+ \bigg( \dashint_{B^\ep_{r,+}} | p^\ep - \pi_1(x/\ep) - \ep \pi_2(x/\ep) -  \bar{P}_2 |^2\dd x \bigg)^{1/2} \Bigg\} \\
				& \le C r^{1+\gamma} (M + M^{6+2\gamma +\delta}),
			\end{aligned}
		\end{equation}
		where $\mathscr{Q}_1(\Omega)$ is used in Theorem \ref{thm.C1g} and $\mathscr{Q}_2(\Omega)$ is the class of all solutions to the Stokes equations in a periodic bumpy John half-space $\Omega$, with quadratic growth at infinity, that vanish on $\partial\Omega$; see \eqref{e.defQ2}. The constant $C$ is independent of $\ep$, $M$ and $r$, but depends on $L$, $\gamma$ and $\delta$.
	\end{theorema}

	We point out that the building blocks in $\mathscr{Q}_1(\Omega)$ and $\mathscr{Q}_2(\Omega)$ are defined through the first-order and second-order boundary layers and play roles of correctors of Stokes system in the bumpy John domain $\Omega$. It turns out that the above
	three regularity results, Theorems \ref{theo.lip.nonlinear}, \ref{thm.C1g} and \ref{thm.C2g}, hold also for the linear Stokes equations, with a linear dependence on the size $M$ of the solutions in $\dot H^1(B^\ep_{1,+})$. Therefore, these statements immediately imply the Liouville theorems for Stokes equations in bumpy John half-spaces with sublinear (see Corollary \ref{cor.liouville}), subquadratic or subcubic growth (see Theorem \ref{thm.liouville}).

	\subsection{Comparison to two closely related works}
	\label{sec.comp}
	
	To further underline the novelty of our work, let us compare our results to the ones of two tightly linked papers.
	
	In the paper \cite{HP}, the first and second authors of the present work carried out the analysis of the large-scale Lipschitz and $C^{1,\gamma}$ regularity for the stationary Navier-Stokes system. The results there, similar to Theorem \ref{theo.lip.nonlinear} and Theorem \ref{thm.C1g} here, hold outside the perturbative regime, that is for arbitrarily large $M$ in \eqref{e.defMbis}. The main differences between \cite{HP} and the present work are: 
	\begin{enumerate}[label=(\roman*)]
		\item In \cite{HP} the bumpy boundary is given by a Lipschitz graph without structure, while here we work in bumpy John domains as defined in Definition \ref{def.John2} that are not necessarily graphs, without structure for the large-scale Lipschitz and $C^{1,\gamma}$ regularity.
		\item In \cite{HP} the analysis relies on a compactness method originating from \cite{AL87} and the first-order boundary layer correctors are needed to prove the large-scale Lipschitz estimate in Theorem \ref{theo.lip.nonlinear}, while here we resort to a quantitative method, which enables us to by-pass the use of the first-order boundary layers for the large-scale Lipschitz regularity; see Subsection \ref{sec.strat}.
		
		\item In \cite{HP} no analysis of the higher-order large-scale regularity is carried out, while here we build the second-order boundary layer correctors that make it possible to prove Theorem \ref{thm.C2g}.
		\item In \cite{HP}, no pressure estimate is established, while in the present paper, we establish the pressure estimates in all cases, following the strategy developed recently in \cite{GZ}.
		\item In \cite{HP}, the nonlinear estimates are not explicit, while here the dependence on $M$ in \eqref{est2.theo.lip.nonlinear}, \eqref{est.thmC1a} and \eqref{est.C2gamma} is given as an explicit polynomial in $M$. 
	\end{enumerate}

	In the paper \cite{Z}, the third author of the current paper carried out an analysis of the large-scale Lipschitz regularity for linear elliptic equations in domains with arbitrary roughness at small scales and quantitative Reifenberg flatness at large scales. Hence, those domains are much rougher than the bumpy John domains considered here. We underline that the discrepancy in these assumptions on the domains comes from the fact that for incompressible Navier-Stokes equations, as opposed to elliptic equations, we have to estimate the pressure in terms of the velocity, which relies on a Bogovskii-type operator as in \cite{HP}; see Subsection \ref{sec.strat} and Appendix \ref{appendix.Bogovskii}. To address this point we work in bumpy John domains defined by Definition \ref{def.John2}.

	\subsection{Outline of the strategy for the proofs}
	\label{sec.strat}
	
	We now point to some essential ingredients and ideas for the proofs. In particular, the fact that we cannot rely on any smoothness at the microscopic scale requires several technical innovations.
	
	\subsubsection*{\underline{Analysis in John domains}}
	
	We perform the analysis in bumpy John domains, as defined in Definition \ref{def.John2}. This type of domains is a good compromise between: 
	\begin{itemize}
		\item on the one hand a high level of arbitrariness of the boundary, which is not a graph, includes certain fractals or cusps, does not oscillate with any structure, and hence approaches better the properties genuinely rough physical surfaces found in real fluids, 
		\item and on the other hand the possibility of being amenable to mathematical analysis, considering the fact already underlined above that we work with incompressible fluid models that involve estimating the pressure, rather than elliptic equations which can be studied in even rougher domains.
	\end{itemize}
	
	In John domains, we can rely on the Bogovskii operator of \cite{ADM06}, whose properties are summarized in Theorem \ref{theo.acosta}. This operator is required from the beginning of our analysis in Subsection \ref{sec.lscz} in order to prove a weak Caccioppoli inequality for the Stokes system (the usual Caccioppoli inequality seems not available in John domains), which then implies the reverse H\"older inequality \eqref{est.Caccipoli+SP}, as a starting point of the large-scale regularity theory.
	
	All the boundary estimates of this work are mesoscopic estimates in the sense that they involve averaged quantities smoothing out the possibly rough microscales. Although it is a direct consequence of the weak Caccioppoli inequality, notice that the reverse H\"older inequality \eqref{est.Caccipoli+SP} is a large-scale estimate. Indeed, going from the weak Caccioppoli inequality \eqref{est.appendix.lem.S.Caccioppoli.ineq.} to \eqref{est.Caccipoli+SP} uses the Poincar\'e inequality that holds in balls large enough, typically at a scale greater than $\ep$. At scales smaller than $\ep$, inward cusps of highly oscillating bumpy John domains may be seen, preventing Poincar\'e's inequality to hold.

	In a nutshell: in the works \cite{KP18,HP}, tools were developed, particularly for the analysis of the first-order boundary layer correctors, to handle bumpy domains with a boundary given by the graph of a Lipschitz function without structure. Here, the analysis in bumpy John domains requires to push the techniques even further, to the limit, as it seems, of what is technically possible. There is one particular point, where we are completely unable to transfer the techniques used above Lipschitz graphs to the present context. Indeed, in \cite{KP18,HP} we used a domain decomposition method pioneered in \cite{GVM10} to study the well-posedness of the Stokes system for the first-order boundary layer correctors. We do not manage to adapt this strategy, in particular the technique of local energy estimates in the bumpy channel, to our current situation. In this paper, we develop a different argument to construct the first-order boundary layers, based on the large-scale Lipschitz estimate proved as an a priori estimate. We discuss this intricate point in more details at the beginning of Subsection \ref{subsec.1st.BL}.

	\subsubsection*{\underline{Quantitative method for the large-scale regularity}}
	
	We rely on a quantitative method for large-scale regularity, inspired by the Schauder's theory pioneered by \cite{AS16,AShen16,Shen17}, the Calder\'{o}n-Zygmund theory motivated by \cite{CP98} and \cite[Chapter 4]{Sh} and the pressure estimate developed in \cite{GX17,GZ19,GZ}.
	This method is based on a perturbation argument. The principle of this method is the following:
	\begin{enumerate}[label=(\arabic*)]
		\item\label{item.qm1} approximate the original rough problem, by a smooth problem at any mesoscopic scales and obtain suboptimal quantitative estimates;
		\item\label{item.qm2} use the improved regularity of the approximate problem, to get the scale-by-scale decay of excess quantities (measuring for instance, H\"older continuity, Lipschitz, $C^{1,\gamma}$, $C^{2,\gamma}$, or higher regularity) for the original rough problem, up to a small error;
		\item\label{item.qm3} conclude by a real variable argument such as Theorem \ref{thm.Shen} or an iteration lemma such as Lemma \ref{lem.iteration}, which are in some sense black boxes oblivious to the equations.
	\end{enumerate}
	In the context of homogenization, the homogenized limit problem with constant coefficients is the approximate problem. Here, the approximate problem is a Stokes problem in a domain with a flat boundary. Both problems have improved regularity, in the sense that the solutions are basically as smooth as one wishes.
	
	We remark that from a high-level point of view all the regularity estimates in this paper follow the above scheme. 
	For the large-scale $W^{1,p}$ regularity stated in Theorem \ref{thm.CZ.Stokes}, item \ref{item.qm1} above corresponds to Lemma \ref{lem.Duews}, item \ref{item.qm2} corresponds to the estimate \eqref{est.Dw.Linfty} and item \ref{item.qm3} corresponds to Theorem \ref{thm.Shen}. For the proof of Lipschitz estimate in Theorem \ref{theo.lip.nonlinear}, item \ref{item.qm1} corresponds to Lemma \ref{lem.est.approx.}, item \ref{item.qm2} corresponds to Lemma \ref{lem.est.approx.func.} and item \ref{item.qm3} corresponds to Lemma \ref{lem.iteration}. The proofs of higher-order regularity estimates in Theorems \ref{thm.C1g} and \ref{thm.C2g} follow a similar scheme.

	\subsubsection*{\underline{Relations between the results and structure of the proofs}}
	
	We conclude this part by dissecting the logical connections between the main results and the intermediate results stated in the paper. 
	
	For the Stokes system, the weak Caccioppoli inequality stated in Lemma \ref{appendix.lem.S.Caccioppoli.ineq.} and the Poincar\'e-Sobolev inequality imply a weak reverse H\"older inequality \eqref{est.Caccipoli+SP}. The Gehring lemma then implies an improvement of integrability stated in Lemma \ref{lem.selfimprove}. This large-scale Meyers-type estimate is crucial to get the quantitative approximation result, Lemma \ref{lem.Duews}, of the system in bumpy John domains by systems in flat domains, at any mesoscopic scales. The large-scale Calder\'on-Zygmund estimate for Stokes system stated in Theorem \ref{thm.CZ.Stokes} is then obtained as a combination of Lemma \ref{lem.Duews} and a real-variable argument in Theorem \ref{thm.Shen}.
	
	For the Navier-Stokes equations, we use the large-scale Calder\'on-Zygmund estimate of Theorem \ref{thm.CZ.Stokes} in combination with a large-scale Sobolev embedding stated in Theorem \ref{thm.Sobolev} to bootstrap the integrability of the nonlinear term. In the end, we are able to prove the decay of an excess quantity associated to the nonlinear term (see Theorem \ref{thm.MFL3}), which reads a large-scale H\"older estimate for the nonlinearity. The large-scale Calder\'on-Zygmund estimate also enables us to get an improved approximation result; see Lemma \ref{lem.est.approx.}. 
	These estimates, together with the smoothness of the approximate problems in flat domains (see Lemma \ref{lem.est.approx.func.}), imply an excess decay estimate in Lemma \ref{lem.excess.decay}, which leads to Theorem \ref{theo.lip.nonlinear} by an iteration lemma (i.e., Lemma \ref{lem.iteration}).
	
	The large-scale Lipschitz regularity 
	in Theorem \ref{theo.lip.nonlinear}
	then makes it possible to construct the velocity and pressure parts of the Green function in bumpy John domains, and to estimate its decay at large scales. This is the purpose of Appendix \ref{app.green}, where we prove estimates for the velocity part of the Green function (see Proposition \ref{prop.Green.G}), its derivatives (see Proposition \ref{prop.DG}), and the pressure part of the Green function (see Proposition \ref{prop.pressureest}). These estimates are the key for our new proof of the existence of the first-order boundary layer correctors; see Theorem \ref{prop.BL1j}. In this way we are able to by-pass the difficulties posed by the method used in \cite{GVM10,DP14,DGV17,KP18,HP}.
	
	To the best of our knowledge, the present work is the first to carry out a thorough analysis of the second-order boundary layer correctors, allowing for linear growth of the boundary data in the tangential direction.  Our idea is to use the first-order boundary layer correctors in an Ansatz for the second-order boundary layers. In Theorem \ref{prop.BL21} and Theorem \ref{prop.BL2356}, we give the detailed structure of the second-order boundary layers. For our analysis to go through, we need some good quantitative 
	convergence/decay
	of the first-order boundary layers away from the boundary. Hence we work in a periodic framework, according to Definition \ref{def.PeriodicJohn}; but this is by no means an optimal assumption. Other structures, such as almost-periodic structures with a non-resonance condition, or random ergodic with quantitative decorrelation properties at large scales, would certainly be manageable.
	
	The key outcome of Section \ref{sec.bl} handling the construction of boundary layers is summarized in Proposition \ref{lem.est.scaled.BL1j} and Proposition \ref{lem.est.scaled.BL2j}. They are then used in Section \ref{sec.higher} to run the excess decay method for the higher-order regularity. Theorem \ref{thm.C1g} and Theorem \ref{thm.C2g} are proved there.

	\subsection{Notations and definitions}
	\label{sec.not}
	
	\subsubsection*{\underline{John domains}}
	
	We first define John domains. These domains were introduced by John in \cite{John61} and named after John in \cite{MS79}.
	
	\begin{definition}\label{def.John}
		Let $\Omega \subset \R^d$ be an open bounded set and $\tilde{x} \in \Omega$. We say that $\Omega$ is a John domain 
		(or a bounded John domain)
		with respect to $\tilde{x}$ and with constant $L$ if for any $y\in \Omega$, there exists a Lipschitz mapping $\rho:[0,|y-\tilde{x}|] 
		\to \Omega$ with Lipschitz constant $L\in (0,\infty)$, such that $\rho(0) = y, \rho(|y-\tilde{x}|) = \tilde{x}$ and $\dist(\rho(t),\partial\Omega) \ge t/L$ for all $t\in [0, |y-\tilde{x}|]$. 
	\end{definition}
	
	Our analysis takes advantage of a key property of John domains, namely the existence of a right inverse of the divergence operator. Such an operator is usually called a Bogovskii operator; see Appendix \ref{appendix.Bogovskii} where we state the result of \cite{ADM06}.
	
	Examples of John domains are: Lipschitz domains, NTA domains, domains with inward cusps or certain fractals such as Koch's snowflake. Notice that domains with outward cusps are not John domains. For our work, 
	we generalize the above definition from bounded domains to a class of unbounded domains. 
	\begin{definition}\label{def.John2}
		Let $\Omega$ be a domain containing the upper half-space of $\R^3$ and assume that $\partial\Omega \subset\{ -1< x_3<0 \}$.
		We say that $\Omega$ is a \emph{bumpy John domain} (or a \emph{bumpy John half-space}) with constants $(L,K)$, if for any $x \in \{ x_3 = 0 \}$ and any $R\ge 1$, there exists a bounded John domain $\Omega_R(x)$ with respect to $x_R = x+R e_d$ and with constant $L\in(0,\infty)$
		according to Definition \ref{def.John}
		 such that
		\begin{equation}\label{e.exisinterdom}
			B_{R,+}(x) \subset \Omega_R(x) \subset B_{KR,+}(x),
		\end{equation}
		where $B_{R,+}(x) = Q_R(x) \cap \Omega$. Here $Q_R(x)$, defined later, is a cube centered at $x$ with side length $2R$.
	\end{definition}
	The above definition guarantees that the constants of John domains are rescaling- and translation-invariant. This is a natural requirement as we are considering unbounded domains.
	
	\begin{definition}\label{def.PeriodicJohn}
		We say that $\Omega$ is a \emph{periodic bumpy John domain} if the following holds:
		\begin{enumerate}[label=(\roman*)]
			\item $\Omega$ is a John domain with constant $(L,K)$,
			\item and $\Omega$ is $(2\pi\mathbb{Z})^2$-translation invariant, namely $2\pi z + \Omega = \Omega$ for any $z\in \Z^2\times \{ 0 \}$.
		\end{enumerate}
	\end{definition}
	
	For simplicity, we assume $K = 2$ in the whole paper. Otherwise, the constant in our main theorems will also depend on $K$.
	
	Throughout the paper, we assume that $\Omega$ is a bumpy John domain satisfying Definition \ref{def.John2}, or a periodic bumpy John domain satisfying Definition \ref{def.PeriodicJohn}. We will always specify in case periodicity is needed. In fact, periodicity is used to construct the second-order boundary layer correctors in Subsection \ref{subsec.2nd.BL} and 
	hence is also an assumption of 
	Theorem \ref{thm.C2g}, Proposition \ref{lem.est.scaled.BL2j}, Subsection \ref{sec.C2gamma} and Theorem \ref{thm.liouville} (ii).
	
	Let $\Omega^\ep := \ep \Omega=\{ x\in \R^3|\ \ep^{-1}x\in \Omega \}$. We refer to $\Omega^\ep$ as a \emph{highly oscillating bumpy John domain}. Note that
	\begin{equation}\label{cond.boundary}
		\partial \Omega^\ep \subset \{ x\in \R^3| -\ep < x_3 < 0 \}.
	\end{equation}
	A key fact about $\Omega^\ep$ is that $\Omega^\ep$ is still a John domain with the same constants as in Definition \ref{def.John2}, as these constants are scale-invariant.
	
	Throughout the paper, we use the following notations:
	\begin{align*}
		\begin{split}
			&B^\ep_{r,+} = \{x = (x',x_3) \in \R^3~|~x'\in(-r,r)^2\,, \ x_3< r \} \cap \Omega^\ep, \\
			&\Gamma^\ep_r = \{x  = (x',x_3)\in \R^3~|~x'\in(-r,r)^2\} \cap \partial \Omega^\ep.
		\end{split}
	\end{align*}
	Since the boundary could be very rough at small scales, $B_{r,+}^\ep$ and $\Gamma_r^\ep$ may have disconnected components. Fortunately, this will not cause any issue since the solutions will be extended naturally by zero across the boundary. We also define 
	\begin{align*}
		&Q_r=Q_{r}(0)= (-r,r)^3,\quad Q_{r}(y) = y + Q_{r}(0),\quad 
		{Q^\ep_r=Q_{r}\cap \{x\in \R^3| x_3 > - \ep \},}
		\\
		&Q_r^\ep(y) = Q_{r}(y) \cap \{x\in \R^3| x_3 > - \ep \}\quad\mbox{and}\quad Q_{r,+}(y) = Q_{r}(y) \cap \{x\in \R^3| x_3 >0\}.
	\end{align*}
	From the definition of $B^\ep_{r,+}$, one has $B_{r,+}^\ep \subset Q^\ep_{r}$ and $|Q^\ep_{r}\setminus B^\ep_r| \le 4\ep r^2$.

	\subsubsection*{\underline{Weak solutions}}
	
	We work in the framework of weak solutions of \eqref{intro.NS.ep}. A velocity/pressure pair $(u^\ep,p^\ep)\in H^1(B^\ep_{1,+})^3\times L^2(B^\ep_{1,+})$ is said to be a weak solution to \eqref{intro.NS.ep} if $u^\ep$ satisfies: (i) $\nabla\cdot u^\ep=0$ in the sense of distributions, (ii) $\psi u^\ep\in H^1_0(Q_1)^3$ for any cut-off function $\psi\in C_0^\infty(Q_1)$, and (iii) the weak formulation
	\begin{align}\label{intro.def2}
		\int_{B^\ep_{1,+}} \nabla u^\ep \cdot \nabla \varphi-\int_{B^\ep_{1,+}} p^\ep (\nabla\cdot\varphi)
		= -\int_{B^\ep_{1,+}} (u^\ep\cdot\nabla u^\ep) \cdot \varphi
	\end{align}
	for any $\varphi\in C^{\infty}_{0}(B^\ep_{1,+})^3$. 
	The Poincar\'e inequality is a fundamental tool in our paper. Since the weak solution vanishes on the lower boundary $\Gamma^\ep_1$, we extend it to $Q^\ep_1$ by zero across $\Gamma^\ep_1$. This enables us to use for instance \cite[Proposition 3.15]{GMbook}, to get that: for all fixed bumpy John domain $\Omega$ with constant $L\in(0,\infty)$ according to Definition \ref{def.John2}, for all fixed $r\geq\ep$, and for all $u\in H^1(B^\ep_{r,+})$ such that $u=0$ on $\Gamma^\ep_r$,
	\begin{equation}\label{e.poinca}
		\int_{B^\ep_{r,+}}|u|^2\leq Cr^2\int_{B^\ep_{r,+}}|\nabla u|^2,
	\end{equation}
	where $C$ is an absolute constant independent of $\ep$ and $r$. Notice that this estimate is only valid at scales $r\ge\ep$. Indeed, below that scale the constant in \eqref{e.poinca} may degenerate because in particular of inward cusps at small scales.

	\subsubsection*{\underline{Other frequently used notations}}
	
	The notation $C$ denotes a positive constant that varies from line to line, and may or may not be universal.  Whenever needed, we make precise what the constant depends on. The notation $x\cdot y$ stands for the inner product of vectors $x,\, y\in\R^N$. The notation $a\lesssim b$ (resp. $a\gtrsim b$) means that there exists a universal constant $C$ such that $a\leq Cb$ (resp. $Ca\geq b$). The notation $a\approx b$ stands for $a\lesssim b$ and $a\gtrsim b$. 
	
	\subsection{Outline of the paper}
	\label{sec.paper}
	
	Section \ref{sec.nlest} is devoted to the proof of the large-scale Calder\'on-Zygmund estimate stated in Theorem \ref{thm.CZ.Stokes}. We then use this result to bootstrap the regularity and obtain a large-scale H\"older estimate for the nonlinear term in the Navier-Stokes equations; see Theorem \ref{thm.MFL3}. In Section \ref{sec.large}, we prove Theorem \ref{theo.lip.nonlinear}. In Section \ref{sec.bl}
	we construct the first-order and second-order boundary layer correctors. Theorem \ref{thm.C1g} and Theorem \ref{thm.C2g} 
 are proved in Section \ref{sec.higher}. 
	There are three appendices. Appendix \ref{appendix.Bogovskii} is devoted to the results related to Bogovskii's operator in John domains. Appendix \ref{app.green} handles the construction and estimates for the Green function associated to the Stokes system in bumpy John domains. Appendix \ref{app.it} provides a proof for the iteration Lemma \ref{lem.iteration}.
	
	\subsection*{Acknowledgements} 
	The authors would like to thank Prof. Zhongwei Shen for pointing out a mistake in an early version of the paper.
	M.H. is partially supported by JSPS KAKENHI Grant Number JP 20K14345. C.P. is partially supported by the Agence Nationale de la Recherche, project BORDS, grant ANR-16-CE40-0027-01, project SINGFLOWS, grant ANR-18-CE40-0027-01 and project CRISIS, grant ANR-20-CE40-0020-01.
	\section{Estimates for the nonlinearity}
	\label{sec.nlest}
	
	The goal of this section is to obtain some regularity estimates for the nonlinearity $-u^\ep\otimes u^\ep$ for the Navier-Stokes equations. As usual, this follows from a bootstrap argument for the stationary Navier-Stokes equations. However, since there is no smoothness up to the boundary, we have to carry out a delicate large-scale bootstrap argument.

	\subsection{Large-scale Calder\'on-Zygmund estimate}
	\label{sec.lscz}
	Assume $r \ge \ep$. 
	Let $\Omega$ be a bumpy John domain with constant $L$ according to Definition \ref{def.John2}. Let $(u^\ep, p^\ep)\in H^1(B^\ep_{1,+})^3\times L^2(B^\ep_{1,+})$ be a weak solution of the linear Stokes system
	\begin{equation}\label{eq.Stokes}
		\left\{
		\begin{array}{ll}
			-\Delta u^\ep+\nabla p^\ep=\nabla\cdot F^\ep &\mbox{in}\ B^\ep_{1,+}\\
			\nabla\cdot u^\ep=0&\mbox{in}\ B^\ep_{1,+}\\
			u^\ep=0&\mbox{on}\ \Gamma^\ep_1.
		\end{array}
		\right.
	\end{equation}
	We extend $u^\ep$ and $F^\ep$ by zero to the whole of $Q_1=Q_1(0)$; they are denoted again by $u^\ep\in H^1(Q_1)^3$ and $F^\ep\in L^2(Q_1)^{3\times3}$ respectively. 
	Note that we also have $\nabla u^\ep = 0$ in $Q_1(0) \setminus B_{1,+}^\ep$. For any $r\ge \ep$ and $Q_{16r}(y) \subset Q_1(0)$, Lemma \ref{appendix.lem.S.Caccioppoli.ineq.} and the Sobolev-Poincar\'{e} inequality imply that for any $\theta \in (0,1)$
	\begin{equation}\label{est.Caccipoli+SP}
		\begin{aligned}
		\bigg( \dashint_{Q_{r}(y)} |\nabla u^\ep|^2 \bigg)^{1/2} &\le \theta \bigg( \dashint_{Q_{16r}(y)} |\nabla u^\ep|^2 \bigg)^{1/2} \\
		& \qquad + \frac{C}{\theta} \bigg( \dashint_{Q_{16r}(y)} |\nabla u^\ep|^{6/5} \bigg)^{5/6} + \bigg( \dashint_{Q_{16r}(y)} |F^\ep|^2 \bigg)^{1/2}.
	\end{aligned}
	\end{equation}
	Here the constant $C$ depends only on $L$.

	We refer to \cite[Lemma 2.2]{Z} for a similar proof of \eqref{est.Caccipoli+SP} in the case of elliptic equations. The John boundary condition for Stokes system results in additional difficulties as we only have a weak Caccioppoli inequality in Lemma \ref{appendix.lem.S.Caccioppoli.ineq.}. Notice that this estimate holds only at large-scales, namely, $r \ge \ep$, because Lemma \ref{appendix.lem.S.Caccioppoli.ineq.} as well as the Sobolev-Poincar\'{e} inequality 
	fail for $r\ll \ep$ (inward cusps are allowed in John domains and these cusps can be seen at a scale less then $\ep$). As a result, we are not able to derive the full-scale Gehring's inequality (e.g., \cite[Chapter V, Proposition 1.1]{G83} or \cite[Theorem 1.10]{BF02}).
	Instead, we can show a large-scale Gehring's inequality; see Lemma \ref{lem.selfimprove} below.

	For $p\in [1,\infty)$, define the averaging operator
	\begin{equation*}
		\mathcal{M}_t^p[g](x) = \bigg( \dashint_{Q_{t}(x)} |g
		|^{p} \bigg)^{1/p}.
	\end{equation*}
	The important exponents for us are $p = \frac{6}{5}$ and $p = 2$. For convenience, sometimes we write $\mathcal{M}_t^2$ as $\mathcal{M}_t$ in Subsection \ref{subsec.bootstrap}. The following lemma collects useful properties of $\mathcal{M}_t$.
	\begin{lemma}\label{lem.est.M.op.properties}
		For $p\in[1,\infty)$ and $g\in L^p(Q_1)$, we have the following properties. 
		\begin{enumerate}[label=(\roman*)]
			\item For $1\le p'\le p<\infty$ and $Q_{t}(x)\subset Q_1$, 
			\begin{equation}\label{est1.lem.est.M.op.properties}
				\mathcal{M}_t^{p'}[g](x) \le \mathcal{M}_t^{p}[g](x).
			\end{equation}
			\item For $0< t_1\le t_2<1$ and $Q_{t_2}(x)\subset Q_1$, 
			\begin{equation}\label{est2.lem.est.M.op.properties}
				\mathcal{M}_{t_1}^{p}[g](x) \le C\big(\frac{t_2}{t_1}\big)^{3/p} \mathcal{M}_{t_2}^{p}[g](x).
			\end{equation}
			\item For $0<t\le s$ with $Q_{s+t}(y)\subset Q_1$,
			\begin{equation}\label{est3.lem.est.M.op.properties}
				\int_{Q_s(y)} |g|^p 
				\le C \int_{Q_s(y)} \mathcal{M}_t^p[g]^p  
				\le C\int_{Q_{s+t}(y)} |g|^p.
			\end{equation}

			\item For $0<t_1\le t_2\le s$ with $Q_{s+t_1+t_2}(y)\subset Q_1$ and $q\in[p,\infty)$,
			\begin{equation}\label{est4.lem.est.M.op.properties}
				\begin{aligned}
					\dashint_{Q_{s}(y)} \mathcal{M}^{p}_{t_2} [g]^{q}
					\le C \dashint_{Q_{s+t_2}(y)} \mathcal{M}^{p}_{t_1} [g]^{q}.
				\end{aligned}
			\end{equation}
			\item  For $0<s\le t$ with $Q_{s+t}(y)\subset Q_1$,
			\begin{equation}\label{est5.lem.est.M.op.properties}
				\mathcal{M}_t^p[g](y)
				\le C\dashint_{Q_{s}(y)} \mathcal{M}_t^p[g].
			\end{equation}
	\end{enumerate}
		Here the constant $C$ depends on $p$ and $ p'$, but not on $s, t, t_1$ or $t_2$.
	\end{lemma}

	Using the averaging operator and Lemma \ref{lem.est.M.op.properties}, we can show a large-scale Gehring's inequality (also known as self-improving property or Meyers' estimate).
	\begin{lemma}\label{lem.selfimprove}
		Let $L\in(0,\infty)$ and $\Omega$ be a bumpy John domain with constant $L$ according to Definition \ref{def.John2}. There exists some $p_0\in(2,\infty)$ so that for any $0<r<1, \ep \le t \le 1$ with $Q_{3r+t}(y) \subset Q_{1}(0)$,
		\begin{equation}\label{est.Meyers}
			\bigg( \dashint_{Q_{r}(y)} | \mathcal{M}_t^2[\nabla u^\ep] |^{p_0} \bigg)^{1/p_0} \le C \bigg( \dashint_{Q_{3r}(y)} |\mathcal{M}_t^2 [\nabla u^\ep] |^{2} \bigg)^{1/2} + C\bigg( \dashint_{Q_{3r}(y)} | \mathcal{M}_t^2[F^\ep] |^{p_0} \bigg)^{1/p_0},
		\end{equation}
		where the constant $C$ and the Lebesgue exponent $p_0$ depend only on $L$.
	\end{lemma}
	\begin{proof}
		Assume first that $r\ge t$. Then by Lemma 
		\ref{lem.est.M.op.properties}, we may rewrite \eqref{est.Caccipoli+SP} as
		\begin{equation*}
			\begin{aligned}
				& \bigg( \dashint_{Q_{r}(y)} |\mathcal{M}_t^2 [\nabla u^\ep] |^{2} \bigg)^{1/2} \\
				&\le C \bigg( \dashint_{Q_{2r}(y)} | \nabla u^\ep|^2  \bigg)^{1/2} \\
				& \le C \theta \bigg( \dashint_{Q_{32r}(y)} |\nabla u^\ep|^2 \bigg)^{1/2} + \frac{C}{\theta}\bigg( \dashint_{Q_{32r}(y)} |\nabla u^\ep|^{6/5} \bigg)^{5/6} + C\bigg( \dashint_{Q_{32r}(y)} |F^\ep|^2 \bigg)^{1/2} \\
				& \le C \theta \bigg( \dashint_{Q_{32r}(y)} |\mathcal{M}_t^2[ \nabla u^\ep]|^2 \bigg)^{1/2}  \\
				& \qquad+ \frac{C}{\theta}\bigg( \dashint_{Q_{32r}(y)} | \mathcal{M}_t^{2} [\nabla u^\ep]|^{6/5}  \bigg)^{5/6} + C\bigg( \dashint_{Q_{32r}(y)} | \mathcal{M}_t^2[F^\ep] |^{2} \bigg)^{1/2}.
			\end{aligned}
		\end{equation*}
		For $0<r<t$, Lemma \ref{lem.est.M.op.properties} (v) implies
		\begin{equation*}
			\|  \mathcal{M}_t^2 [\nabla u^\ep] \|_{L^\infty(Q_r(y))} \le C\dashint_{Q_{4r}(y)} \mathcal{M}_t^2 [\nabla u^\ep].
		\end{equation*}
		These imply that a weaker reverse H\"{o}lder inequality holds for all scales $r\in (0,1)$ with $Q_{32r+t}(y) \subset Q_1(0)$. By a version of Gehring's inequality \cite[Chapter V, Proposition 1.1]{G83} or \cite[Theorem 1.10]{BF02}, and choosing $\theta$ sufficiently small, there exists some $p_0 > 2$ such that for all $r\in (0,1)$ with $Q_{32r+t}(y) \in Q_1(0)$, 
		\begin{equation}\label{e.lemMt65}
			\bigg( \dashint_{Q_{r}(y)} |\mathcal{M}_t^2 [\nabla u^\ep] |^{p_0} \bigg)^{1/p_0} \le C\bigg( \dashint_{Q_{32r}(y)} |\mathcal{M}_t^2[ \nabla u^\ep]|^2 \bigg)^{1/2} + C\bigg( \dashint_{Q_{32r}(y)} | \mathcal{M}_t^2[F^\ep] |^{2} \bigg)^{1/2}.
		\end{equation}
	
		To conclude the proof, we use a covering argument to adjust the size of cubes. By covering the cube $Q_{32r}(y)$ by a finite number of cubes $Q_{r}(y_i)$ and applying the last estimate in every $Q_{r}(y_i)$ , we get the estimate 
	\begin{equation*}
		\begin{aligned}
			\bigg( &\dashint_{Q_{32r}(y)} | \mathcal{M}_t^2 [\nabla u^\ep]|^{p_0}  \bigg)^{1/p_0} \\
			& \quad \leq C\bigg( \dashint_{Q_{96r}(y)} | \mathcal{M}_t^2 [\nabla u^\ep]|^2  \bigg)^{1/2} + C\bigg( \dashint_{Q_{96r}(y)} | \mathcal{M}_t^2[F^\ep] |^{2} \bigg)^{1/2}
		\end{aligned}
	\end{equation*}
	for $\ep \le t \le 1$ and $Q_{96r+t}(y)\subset Q_1(0)$, at the price of a larger constant $C$ than in \eqref{e.lemMt65}. Replacing $32r$ by $r$, we obtain the desired estimate.
	\end{proof}

		\begin{remark}[Covering argument]\label{rem.cov}
		The covering argument above to adjust the size of cubes should be a standard technique in analysis. Similar arguments may be used later in this paper.
	\end{remark}

	The following theorem is a large-scale boundary Calder\'{o}n-Zygmund estimate, or in other words, a large-scale boundary $W^{1,p}$ estimate, for the linear Stokes system. 
	\begin{theorem}\label{thm.CZ.Stokes}
		For all $\ep\in(0,\frac12)$, $L\in (0,\infty)$ and $p\in (2,\infty)$ 
		the following statement holds. Let $\Omega$ be a bumpy John domain with constant $L$ according to Definition \ref{def.John2}. 
		Suppose $\ep \le t \le r \le \frac12$, $Q_{5r}(x) \subset Q_1(0)$ and $\mathcal{M}_t^2[F^\ep] \in L^p(Q_{4r}(x))$.
		Then the weak solution $u^\ep$ to \eqref{eq.Stokes} satisfies
		\begin{equation}\label{est.CZ.Stokes}
			\bigg( \dashint_{Q_{r}(x)} | \mathcal{M}_t^2[\nabla u^\ep] |^{p} \bigg)^{1/p} \le C \bigg( \dashint_{Q_{4r}(x)} |\mathcal{M}_t^2 [\nabla u^\ep] |^{2} \bigg)^{1/2} + C\bigg( \dashint_{Q_{4r}(x)} | \mathcal{M}_t^2[F^\ep] |^{p} \bigg)^{1/p},
		\end{equation}
		where the constant $C$ depends only on $L$ and $p$.
	\end{theorem}
	
	The proof of Theorem \ref{thm.CZ.Stokes} relies on a combination of: (i) a real variable argument (see Theorem \ref{thm.Shen}), and (ii) the quantitative approximation at sufficiently large scales $s$ of the solution $u^\ep$ to the Stokes system in the bumpy domain by a solution to a Stokes problem in a flat domain (see Lemma \ref{lem.Duews}).

	We first state the real variable result. The following theorem is taken from \cite[Theorem 4.2.3]{Sh}, where it is stated for balls instead of cubes. Notice that we introduce some flexibility for the size of the cubes as in \cite[Theorem 2.6]{Z} and \cite[Theorem 4.1]{S20} to fit the cubes in Lemma \ref{lem.Duews}.
	%
	\begin{theorem}[{\cite[Theorem 4.2.3]{Sh}}]\label{thm.Shen}
		Let $N>1,\ 0<c_1<1$, $\kappa>0$ and $\lambda>2$. 
		Let $Q_0$ be a cube in $\R^3$ and $\mathcal{F}\in L^2(\lambda Q_0)$. 
		Let $q>2$ and $f\in L^p(\lambda Q_0)$ for some $2<p<q$. Suppose that for each cube $Q\subset 2Q_0$ with $|Q| \le c_1 |Q_0|$, there exist two measurable functions $F_Q$ and $R_Q$ on $2Q$ such that $|\mathcal{F}| \le |F_Q| + |R_Q|$ on $2Q$, and
		\begin{equation*}
			\begin{aligned}
				\bigg( \dashint_{2Q} |R_Q|^q \bigg)^{1/q} & \le N \bigg( \dashint_{\lambda Q} |\mathcal{F}|^2 \bigg)^{1/2},\\
				\bigg( \dashint_{2Q} |F_Q|^2 \bigg)^{1/2} &\le \kappa \bigg( \dashint_{\lambda Q} |\mathcal{F}|^2 \bigg)^{1/2} 
				+ \bigg( \dashint_{\lambda Q} |f|^2 \bigg)^{1/2}.
			\end{aligned}
		\end{equation*}
		There exists $\kappa_0>0$, depending on $\lambda,p,q,c_1$ and $N$, with the property that if $0<\kappa<\kappa_0$, then $\mathcal{F}\in L^p(Q_0)$ and
		\begin{equation*}
			\bigg( \dashint_{Q_0} |\mathcal{F}|^p \bigg)^{1/p} \le C\biggl\{ \bigg( \dashint_{\lambda Q_0} |\mathcal{F}|^2 \bigg)^{1/2} + \bigg( \dashint_{\lambda Q_0} |f|^p \bigg)^{1/p} \biggr\},
		\end{equation*}
		where $C$ depends on $\lambda$, $p$, $q$, $c_1$ and $N$.
	\end{theorem}
	
	We now turn to the approximation. Fix $t \ge \ep$. To apply Theorem \ref{thm.Shen}, we introduce an approximation of $u^\ep$ at all scales $s\ge  t$. Fix $y\in \{-1\le x_3\le 1\}$. 
	Let $Q_r^\ep(y) = Q_{r}(y) \cap \{{x_3} > - \ep \}$. Let $s\ge t$ be fixed. 
	By the co-area formula and the fact that $\nabla u^\ep \equiv 0$ below the bottom boundary,
	there exists some $t_0 \in [1,2]$ so that
	\begin{equation}\label{est.co-area}
		\bigg( \int_{\partial Q^\ep_{t_0 s}(y)} |\nabla u^\ep|^2 \bigg)^{1/2} \le \frac{C}{s^{1/2}} \bigg( \int_{ Q^\ep_{2s}(y)} |\nabla u^\ep|^2 \bigg)^{1/2}.
	\end{equation}
	Note that $t_0$ depends particularly on the specific solution $u^\ep$. But this is harmless as $t_0$ is bounded uniformly in $\ep$ in $[1,2]$ and $C$ is an absolute constant. 
Now, we construct an approximation of $u^\ep$ in $Q_{t_0 s}(y)$ by considering the following Stokes system
	\begin{equation}\label{eq.ws}
		\left\{
		\begin{array}{ll}
			-\Delta w_{s}+\nabla q_{s}=0 &\mbox{in}\ Q_{t_0 s}^\ep(y)\\
			\nabla\cdot w_{s}=0&\mbox{in}\ Q_{t_0 s}^\ep(y)\\
			w_{s}=u^\ep &\mbox{on}\ \partial Q_{t_0 s}^\ep(y).
		\end{array}
		\right.
	\end{equation}
	Since $w_{s} = 0$ on $\partial Q_s^\ep(y)\cap \{{x_3} = -\ep \}$, we may extend the solution $w_{s}$ naturally across this boundary. For our purpose, we need some regularity estimates for $w_s$. First of all, the energy estimate implies
	\begin{equation}\label{est1.lem,est.ws}
		\bigg(\dashint_{Q^\ep_{t_0 s}(y)} |\nabla w_s|^2 \bigg)^{1/2} 
		+ \bigg(\dashint_{Q^\ep_{t_0 s}(y)} |q_s-\dashint_{Q^\ep_{t_0 s}(y)} q_s|^2 \bigg)^{1/2}
		\le C\bigg(\dashint_{Q_{t_0 s}(y)} |\nabla u^\ep|^2 \bigg)^{1/2}.
	\end{equation}
	Second, by the classical regularity theory for the Stokes system over a flat boundary, we have
	\begin{multline}\label{est.Dw.Linfty}
		\|{\nabla w_{s}}\|_{L^\infty(Q_{\frac12 s}(y))} \le C\bigg( \dashint_{Q_{s}(y)} |\nabla w_{s}|^2 \bigg)^{1/2}\\
		\le \frac{C}{s^{3/2}}\bigg(\int_{Q_{t_0 s}(y)} |\nabla u^\ep|^2 \bigg)^{1/2}=Ct_0^{3/2}\bigg(\dashint_{Q_{t_0 s}(y)} |\nabla u^\ep|^2 \bigg)^{1/2}\leq C\bigg(\dashint_{Q_{t_0 s}(y)} |\nabla u^\ep|^2 \bigg)^{1/2}.
	\end{multline}
	Finally, since $Q^\ep_{t_0 s}(y)$ is a Lipschitz domain and because \eqref{est.co-area} implies $w_s |_{\partial Q^\ep_{t_0 s}(y)} \in H^1(\partial Q^\ep_{t_0 s}(y))^3$, it follows from \cite{FKV88} that $(\nabla w_s)^* \in L^2(\partial Q^\ep_{t_0 s}(y))$, where $(\nabla w_s)^*$ is the nontangential maximal function. More precisely, we have
	\begin{equation*}
		\bigg( \int_{\partial Q^\ep_{t_0 s}(y)} |(\nabla w_s)^*|^2 \bigg)^{1/2}\le C\bigg( \int_{\partial Q^\ep_{t_0 s}(y)} |\nabla u^\ep|^2 \bigg)^{1/2} \le \frac{C}{s^{1/2}} \bigg( \int_{ Q^\ep_{2s}(y)} |\nabla u^\ep|^2 \bigg)^{1/2}.
	\end{equation*}
	This yields,
	\begin{equation}\label{est.ws.L3}
		\bigg( \int_{Q^\ep_{t_0 s}(y)} |\nabla w_s|^3 \bigg)^{1/3} \le \frac{C}{s^{1/2}}\bigg( \int_{ Q^\ep_{2s}(y)} |\nabla u^\ep|^2 \bigg)^{1/2};
	\end{equation}
	see \cite[Lemma 3.3]{WZ14} and \cite[Remark 9.3]{KLS13}.
	The above higher integrability of $w_s$ plays an important role in the following lemma.

	\begin{lemma}\label{lem.Duews}
		Let $L\in (0,\infty)$ and $\Omega$ be a bumpy John domain with constant $L$ according to Definition \ref{def.John2}. Let $(w_{s}, q_{s})$ be given as above. Then there exists $\sigma\in(0,\frac{1}{12}]$ such that for any $\theta\in(0,1)$, $\ep\in(0,\theta]$, $s\in [\ep/\theta,1]$, $Q_{7s}(y)\subset Q_1(0)$,
		\begin{equation}\label{est.Duews}
			\begin{aligned}
				& \bigg( \dashint_{Q_{s}(y)} |\nabla u^\ep - \nabla w_{s}|^2 \bigg)^{1/2} \\
				& \qquad \le C\theta^{\sigma} \bigg( \dashint_{Q_{7s}(y)} |\nabla u^\ep|^2 \bigg)^{1/2} +  C_\theta \bigg( \dashint_{Q_{7s}(y)} |F^\ep|^2 \bigg)^{1/2},
			\end{aligned}
		\end{equation}
		where $C$ depends only on $L$, 
		and $C_\theta$ depends on $L$, $\sigma$ and $\theta$.
	\end{lemma}
	\begin{proof}
		We rely on the variational definition of the weak solutions of \eqref{eq.ws}. 
		First of all, by \eqref{eq.ws}, we see that $u^\ep - w_s \in H^1_0(Q^\ep_{t_0 s}(y))^3$ and $\nabla\cdot (u^\ep-w_s) = 0$, since $u^\ep$ has been extended by zero. Thus we can test \eqref{eq.ws} against $u^\ep - w_s$ to obtain
		\begin{equation}\label{eq.var.ws}
			\int_{Q^\ep_{t_0 s}(y)} \nabla w_s \cdot \nabla (u^\ep - w_s) = 0.
		\end{equation}
		Let $\eta_{\ep,+}$ be a smooth cut-off function so that $0\le\eta_{\ep,+}\le1$, $\eta_{\ep,+}({x}) = 1$ if ${x}_3>2\ep$, $\eta_{\ep,+} ({x}) = 0$ if ${x}_3<\ep$, and $|\nabla\eta_{\ep,+}| \le C\ep^{-1}$.  
		It is easy to verify that $\psi:=(u^\ep - w_{s})\eta_{\ep,+}^2 \in H_0^1(B^\ep_{t_0 s,+}({y}))^3$, where $B^\ep_{t_0 s,+}(y):=y+B^\ep_{t_0 s,+}$. Testing \eqref{eq.Stokes} 
		against $\psi$, we obtain
		\begin{equation}\label{eq.var.ue}
			\begin{aligned}
				&\int_{B^\ep_{t_0 s,+}(y)} \nabla u^\ep \cdot \nabla ((u^\ep - w_{s})\eta_{\ep,+}^2) 
				\\
				&\qquad = \int_{B^\ep_{t_0s,+}(y)} (p^\ep - \mathcal{P})  ((u^\ep - w_{s})\cdot 2\eta_{\ep,+} \nabla\eta_{\ep,+}) 
				-\int_{B^\ep_{t_0s,+}(y)} F^\ep \cdot \nabla ((u^\ep - w_{s})\eta_{\ep,+}^2)
			\end{aligned}
		\end{equation}
		for any $\mathcal{P} \in\R$ (to be determined later).
		Combining \eqref{eq.var.ws} and \eqref{eq.var.ue} and using the fact $\nabla u^\ep = 0$ in $Q_{t_0 s}^\ep(y) \setminus B_{t_0 s,+}^\ep(y)$, we arrive at
		\begin{equation}\label{eq.var.mix}
			\begin{aligned}
				\int_{Q^\ep_{t_0 s}(y)} \nabla (u^\ep -w_s) \cdot \nabla (u^\ep - w_s) & = \int_{B^\ep_{t_0 s,+}(y)} \nabla u^\ep \cdot \nabla ((u^\ep - w_{s})(1-\eta_{\ep,+}^2)) \\
				& \quad +  \int_{B^\ep_{t_0 s,+}(y)} (p^\ep - \mathcal{P})  ((u^\ep - w_{s})\cdot 2\eta_{\ep,+} \nabla\eta_{\ep,+}) 
				\\
				& \quad -\int_{B^\ep_{t_0 s,+}(y)} F^\ep \cdot \nabla ((u^\ep - w_{s})\eta_{\ep,+}^2).
			\end{aligned}
		\end{equation}
		Now, we are going to estimate the integrals on the right-hand side of the above equation. Note that $1-\eta_{\ep,+}^2$ and $\nabla \eta_{\ep,+}$ are both supported in $\{ -\ep < x_3 \le 2\ep \}$. Let $R^\ep_s := Q_{t_0 s}(y) \cap \{-\ep \le x_3\le 2\ep \}$ and $T^\ep_{s}=Q_{t_0 s}(y)\cap\{0 \le x_3\le2\ep\}$. Clearly, 
		$|T^\ep_{s}| \le |R^\ep_s|\leq C\ep {s}^2$. To estimate the first integral, we use the Poincar\'{e} inequality applied in $R^\ep_s$ to obtain
		\begin{equation}\label{est.Int-1}
			\begin{aligned}
				&\bigg| \int_{B^\ep_{t_0s,+}(y)} \nabla u^\ep \cdot \nabla ((u^\ep - w_{s})(1-\eta_{\ep,+}^2)) \bigg|\\
				& \le \bigg( \int_{R^\ep_s} |\nabla u^\ep|^2 \bigg)^{1/2} \bigg( \int_{R^\ep_s} |\nabla ((u^\ep - w_{s})(1-\eta_{\ep,+}^2))|^2 \bigg)^{1/2} \\
				& \le C\bigg( \int_{R^\ep_s} |\nabla u^\ep|^2 \bigg)^{1/2} \bigg\{\bigg( \int_{R^\ep_s} |\nabla (u^\ep - w_{s})|^2 \bigg)^{1/2}+\ep^{-1}\bigg( \int_{R^\ep_s} |u^\ep - w_{s}|^2 \bigg)^{1/2}\bigg\}\\
				& \le C\bigg( \int_{R^\ep_s} |\nabla u^\ep|^2 \bigg)^{1/2} \bigg\{\bigg( \int_{R^\ep_s} |\nabla (u^\ep - w_{s})|^2 \bigg)^{1/2}+\bigg( \int_{R^\ep_s} |\nabla(u^\ep - w_{s})|^2 \bigg)^{1/2}\bigg\}\\
				& \le C\bigg( \int_{R^\ep_s} |\nabla u^\ep|^2 \bigg)^{1/2} \bigg( \int_{Q_{t_0s}^\ep(y)} |\nabla (u^\ep - w_{s})|^2 \bigg)^{1/2}.
			\end{aligned}
		\end{equation}
		The last integral of $\nabla (u^\ep - w_{s})$ in the above estimate will 
		eventually be absorbed by the left-hand side of \eqref{eq.var.mix}. The main difficulty to proceed is to obtain a certain estimate of smallness for $\nabla u^\ep$ over the thin strip $R^\ep_s$. This can 
		be done by using Lemma \ref{lem.selfimprove}. In fact, if $\theta\in(0,1)$ and $s \ge \ep/\theta$, Lemma \ref{lem.selfimprove} yields
		\begin{equation*}
			\begin{aligned}
				&\bigg( \dashint_{Q_{2s}(y)} | \mathcal{M}_{\theta s}^2[\nabla u^\ep] |^{p_0} \bigg)^{1/p_0} \\
				&\qquad \le C \bigg( \dashint_{Q_{6s}(y)} |\mathcal{M}_{\theta s}^2 [\nabla u^\ep] |^{2} \bigg)^{1/2} + C\bigg( \dashint_{Q_{6s}(y)} | \mathcal{M}_{\theta s}^2[F^\ep] |^{p_0} \bigg)^{1/p_0}.
			\end{aligned}
		\end{equation*}
		Since for any $z\in Q_{6s}(y)$,
		\begin{equation*}
			\mathcal{M}_{\theta s}^2[F^\ep](z) \le C_\theta \bigg( \dashint_{Q_{7s}(y)} |F^\ep|^2 \bigg)^{1/2},
		\end{equation*}
		together with \eqref{est3.lem.est.M.op.properties}, we obtain,
		\begin{equation*}
			\bigg( \dashint_{Q_{2s}(y)} | \mathcal{M}_{\theta s}^2[\nabla u^\ep] |^{p_0} \bigg)^{1/p_0} \le C\bigg( \dashint_{Q_{7s}(y)} |\nabla u^\ep |^{2} \bigg)^{1/2} + C_\theta \bigg( \dashint_{Q_{7s}(y)} |F^\ep|^2 \bigg)^{1/2}.
		\end{equation*}
		It is important to notice that in the last inequality, $C$ is independent of $\theta$ and $C_\theta$ depends on $\theta$.
		By a similar argument as in \cite{Z}, we can now estimate the right-hand side of \eqref{est.Int-1} as
		\begin{equation}\label{est.ue.Rs}
			\begin{aligned}
				\bigg(\frac{1}{{|Q_s(y)|}}\int_{R^\ep_{s}} |\nabla u^\ep |^2\bigg)^{1/2} & \le C\big(\frac{\theta s}{s}\big)^{1/2-1/{p_0}}
				\bigg( \dashint_{Q_{2s}(y)} | \mathcal{M}_{\theta s}^2[\nabla u^\ep] |^{p_0} \bigg)^{1/p_0} \\
				& \le C\theta^{\sigma} \bigg( \dashint_{Q_{7s}(y)} |\nabla u^\ep |^{2} \bigg)^{1/2} + C_\theta \bigg( \dashint_{Q_{7s}(y)} |F^\ep|^2 \bigg)^{1/2}
			\end{aligned}
		\end{equation}
		with some $\sigma\in(0,\frac12)$. This is the desired estimate of $\nabla u^\ep$ in $R^\ep_s$. 
		Later on we will insert it into \eqref{est.Int-1} and then \eqref{eq.var.mix} to reach a conclusion. 

		Let us turn to the estimate of the second integral on the right-hand side of \eqref{eq.var.mix}. Using H\"{o}lder's inequality and the Poincar\'{e} inequality, we have
		\begin{equation}\label{est.Int-2}
			\begin{aligned}
				& \bigg| \int_{B^\ep_{t_0s,+}(y)} (p^\ep - \mathcal{P})  ((u^\ep - w_{s})\cdot 2\eta_{\ep,+} \nabla\eta_{\ep,+}) \bigg| \\
				&\quad \le C\ep^{-1} \bigg( \int_{T^\ep_{s}} |p^\ep-\mathcal{P}|^{2} \bigg)^{1/2} \bigg( \int_{R^\ep_{s}} |u^\ep - w_s|^2 \bigg)^{1/2} \\
				&\quad \le C \bigg( \int_{T^\ep_{s}} |p^\ep-\mathcal{P}|^{2} \bigg)^{1/2} \bigg( \int_{R^\ep_s} |\nabla (u^\ep - w_s)|^2 \bigg)^{1/2}.
			\end{aligned}
		\end{equation}
		Now, we pick
		\begin{equation*}
			\mathcal{P} := \dashint_{Q_{t_0 s,+}(y)} p^\ep,
		\end{equation*}
		where $Q_{t_0 s,+}(y) = Q_{t_0s}(y) \cap \{ x_3>0 \}$.
		Then the Bogovskii lemma applied in a Lipschitz domain $Q_{t_0 s,+}(y)$ implies
		\begin{equation}
			\begin{aligned}
				\bigg(\int_{T_{s}^\ep} |p^\ep-\mathcal{P} |^2\bigg)^{1/2} &\le \bigg(\int_{Q_{t_0 s,+}(y)} |p^\ep -\mathcal{P} |^2\bigg)^{1/2}\\
				&\le C \bigg(\int_{Q_{t_0s,+}(y)} |\nabla u^\ep|^2\bigg)^{1/2}
				+C \bigg(\int_{Q_{t_0s,+}(y)} |F^\ep|^2\bigg)^{1/2} \\
				& \le C\bigg(\int_{B^\ep_{{t_0s},+}} |\nabla u^\ep|^2\bigg)^{1/2}
				+C \bigg(\int_{B^\ep_{{t_0s},+}} |F^\ep|^2\bigg)^{1/2}.
			\end{aligned}
		\end{equation}
		Unlike the previous argument, we want to gain the smallness for \eqref{est.Int-2OK} below from
		\begin{equation*}
			\bigg( \int_{R^\ep_s} |\nabla (u^\ep - w_s)|^2 \bigg)^{1/2} \le \bigg( \int_{R^\ep_s} |\nabla u^\ep|^2\bigg)^{1/2} + \bigg( \int_{R^\ep_s} |\nabla w_s|^2  \bigg)^{1/2}.
		\end{equation*}
		The estimate for $\nabla u^\ep$ over $R^\ep_s$ is given \eqref{est.ue.Rs}. On the other hand, by \eqref{est.ws.L3} and the H\"{o}lder's inequality, we have
		\begin{equation*}
			\begin{aligned}
				\bigg( \int_{R^\ep_s} |\nabla w_s|^2  \bigg)^{1/2} & \le |R^\ep_s|^{1/6} \bigg( \int_{R^\ep_s} |\nabla w_s|^3  \bigg)^{1/3} \\
				& \le C \frac{|R^\ep_s|^{1/6}}{s^{1/2}}
				 \bigg( \int_{ Q^\ep_{2s}(y)} |\nabla u^\ep|^2 \bigg)^{1/2} \\
				& \le C\big(\frac{\ep}{s}\big)^{1/6} \bigg( \int_{ Q^\ep_{2s}(y)} |\nabla u^\ep|^2 \bigg)^{1/2}.
			\end{aligned}
		\end{equation*}
		Inserting this into \eqref{est.Int-2}, we have
		\begin{equation}\label{est.Int-2OK}
			\begin{aligned}
				& \bigg| \int_{B^\ep_{t_0s,+}(y)} (p^\ep - \mathcal{P})  ((u^\ep - w_{s})\cdot 2\eta_{\ep,+} \nabla\eta_{\ep,+}) \bigg| \\
				& 
				\le C
					\bigg\{ \bigg(\int_{B^\ep_{{t_0s},+}} |\nabla u^\ep|^2\bigg)^{1/2}
					+ \bigg(\int_{B^\ep_{{t_0 s},+}} |F^\ep|^2\bigg)^{1/2}\bigg\} \\
					& \qquad\times
					\bigg\{
					\bigg( \int_{R^\ep_s} |\nabla u^\ep|^2\bigg)^{1/2}
					+\big(\frac{\ep}{s}\big)^{1/6}\bigg(\int_{Q^\ep_{2s}(y)} |\nabla u^\ep|^2 \bigg)^{1/2}
					\bigg\}.
				\end{aligned}
		\end{equation}
		Finally, for the last integral of \eqref{eq.var.mix}, by the Poincar\'{e} inequality, we have
		\begin{equation}\label{est.Int-3OK}
			\bigg| \int_{B^\ep_{t_0s,+}(y)} F^\ep \cdot \nabla ((u^\ep - w_{s})\eta_{\ep,+}^2) \bigg| \le C\bigg( \int_{B^\ep_{t_0 s,+}(y)} |F^\ep|^2 \bigg)^{1/2} \bigg( \int_{Q^\ep_{t_0 s}(y)} |\nabla (u^\ep - w_s)|^2 \bigg)^{1/2}.
		\end{equation}
		Now, \eqref{eq.var.mix} together with \eqref{est.Int-1}, \eqref{est.ue.Rs}, \eqref{est.Int-2OK} and \eqref{est.Int-3OK} gives
		\begin{equation}\label{est.ue-ws.inQs}
			\begin{aligned}
				&\bigg( \int_{Q_{t_0 s}^\ep(y)} |\nabla (u^\ep - w_s)|^2 \bigg)^{1/2} \\
				& \quad \le C\Big( \theta^{\sigma} + \big(\frac{\ep}{s}\big)^{1/12} \Big) \bigg( \int_{Q_{7s}(y)} |\nabla u^\ep|^2 \bigg)^{1/2} + C_\theta \bigg( \int_{Q_{7s}(y)} |F^\ep|^2 \bigg)^{1/2}.
			\end{aligned}
		\end{equation}
		Since we have assume $s \ge \ep/\theta$, then $\ep/s \le \theta$. In view of $t_0\in [1,2]$, \eqref{est.ue-ws.inQs} divided by $|Q_s(y)|^{1/2}$ leads to \eqref{est.Duews}. The proof is complete.
	\end{proof}

	\begin{proof}[Proof of Theorem \ref{thm.CZ.Stokes}]
		We will first prove a slightly weaker version of \eqref{est.CZ.Stokes} when $Q_{{57r}}(x)\subset Q_1(0)$. Then, \eqref{est.CZ.Stokes} can be recovered thanks to a covering argument at the price of enlarging the constant by a numerical factor; see Remark \ref{rem.cov} for more details. When $Q_{57r}(x)$ is far away from the boundary $\Gamma^\ep_1$, the estimate \eqref{est.CZ.Stokes} is a consequence of interior regularity. Hence it suffices to prove \eqref{est.CZ.Stokes} when $Q_{57r}(x)\cap \Gamma^\ep_1\neq\emptyset$. Note that this case can be reduced to the case when $x\in \{ {z}_3 = 0 \}$ by a covering argument as well as interior regularity. 
		To apply Theorem \ref{thm.Shen} to 
		$Q_0:=Q_r(x)$, $\lambda:=56$ and
		$\mathcal F:=\mathcal{M}_t^2[\nabla u^\ep]$ in $Q_{56r}(x)$ with $x\in \{ {z}_3 = 0 \}$, we 
		approximate $u^\ep$ in any cube $Q_s(y)$ contained in $Q_{2r}(x)$ for any scales for $s \ge \ep/\theta$ where $\theta$ is as in Lemma \ref{lem.Duews}.  
		If $Q_s(y)$ is entirely contained in $\{z_3>0\}$, then the well-known interior estimate for the Stokes system applies. If $Q_s(y)$ is contained entirely in $\{ {z}_3<-\ep \}$, then trivially $u^\ep \equiv 0$ in $Q_s(y)$. Hence, it suffices to focus on the typical boundary case $Q_s(y)$ with $y \in \{ {z}_3 = 0 \}$.
		Moreover, we assume $s<r/2$ so that $Q_{57s}(y) \subset Q_{57r}(x) \subset Q_1(0)$ whenever $Q_s(y) \subset Q_{2r}(x)$.

		Now, for each $Q_s(y)$ with $y\in \{{z}_3 =0 \}$, we will discuss two cases. 
		
		{\bf Case 1:} $s{\ge}4t$. 
		By \eqref{est.Dw.Linfty} and \eqref{est.Duews}, there exists $w_s$ solving \eqref{eq.ws} and satisfying
		\begin{equation}\label{est.Dws.Qy}
			\| \nabla w_s \|_{L^\infty(Q_{\frac12 s}(y))} \le C\bigg( \dashint_{Q_{2s}(y)} |\nabla u^\ep|^2 \bigg)^{1/2}
		\end{equation}
		and 
		\begin{equation}\label{est.Duews.Qy}
			\begin{aligned}
				& \bigg( \dashint_{Q_{s}(y)} |\nabla u^\ep - \nabla w_{s}|^2 \bigg)^{1/2} \\
				& \qquad \le C\theta^{\sigma} \bigg( \dashint_{Q_{7s}(y)} |\nabla u^\ep|^2 \bigg)^{1/2} +   \bigg( \dashint_{Q_{7s}(y)} |C_\theta F^\ep|^2 \bigg)^{1/2}.
			\end{aligned}
		\end{equation}
		Note that the above estimate only holds for $s\ge \ep/\theta$. 
		Therefore, we will use Lemma \ref{lem.est.M.op.properties} and replace $\nabla u^\ep$ and $\nabla w_s$ by $\mathcal{M}_t^2[\nabla u^\ep]$ and $\mathcal{M}_t^2[\nabla w_s]$, respectively. Precisely, the above two inequalities imply for $s\ge 4 t \ge \ep/\theta$, 
		\begin{equation*}
			\| \mathcal{M}_t^2[ \nabla w_s] \|_{L^\infty(Q_{\frac14s}(y))} \le C\bigg( \dashint_{Q_{2s}(y)} |\mathcal{M}_t^2[ \nabla u^\ep]|^2 \bigg)^{1/2}\le C\bigg( \dashint_{Q_{7s}(y)} |\mathcal{M}_t^2[ \nabla u^\ep]|^2 \bigg)^{1/2}
		\end{equation*}
		and
		\begin{equation*}
			\begin{aligned}
				& \bigg( \dashint_{Q_{\frac14  s}(y)} |\mathcal{M}_t^2[ \nabla u^\ep] -  \mathcal{M}_t^2[ \nabla w_s]|^2 \bigg)^{1/2} \\
				& \qquad \le C\theta^{\sigma} \bigg( \dashint_{Q_{7s}(y)} |\mathcal{M}_t^2[ \nabla u^\ep]|^2 \bigg)^{1/2} +   \bigg( \dashint_{Q_{7s}(y)} |\mathcal{M}_t^2[ C_\theta F^\ep] |^2 \bigg)^{1/2}.
			\end{aligned}
		\end{equation*}
		
		{\bf Case 2:} $0<s<4t$. In this case $\mathcal{M}_t^2[\nabla u^\ep]$ itself satisfies some trivial estimate. Note that for any ${z}\in Q_{\frac12 s}(y)$, 
		as $Q_{\frac12 s}(z)\subset Q_{s}(y)$, by Lemma \ref{lem.est.M.op.properties} (v),
		
		\begin{equation*}
			\begin{aligned}
				\mathcal{M}_t^2[\nabla u^\ep](z) 
				\le C \bigg( \dashint_{Q_{\frac12 s}(z)} |\mathcal{M}_{t}^2[\nabla u^\ep]|^2\bigg)^{1/2}
				& \le C \bigg( \dashint_{Q_{s}(y)} |\mathcal{M}_t^2[\nabla u^\ep] |^2 \bigg)^{1/2},
			\end{aligned}
		\end{equation*}
		which yields
		\begin{align*}
			\| \mathcal{M}_t^2[\nabla u^\ep]\|_{L^\infty(Q_{\frac14 s}(y))}\le\ &\| \mathcal{M}_t^2[\nabla u^\ep]\|_{L^\infty(Q_{\frac12 s}(y))}\\
			\le\ & C \bigg( \dashint_{Q_{s}(y)} |\mathcal{M}_t^2[\nabla u^\ep] |^2 \bigg)^{1/2}\le C \bigg( \dashint_{Q_{7s}(y)} |\mathcal{M}_t^2[\nabla u^\ep] |^2 \bigg)^{1/2}.
		\end{align*}
		This ends the study of the two cases. We now apply Theorem \ref{thm.Shen} with $\lambda:=56$,  $Q_0 := Q_r(x),\ q := \infty,\ \mathcal F := \mathcal{M}_t^2[\nabla u^\ep]$ 
		and $f := \mathcal{M}_t^2[ C_\theta F^\ep] $. 
		Moreover,
		\begin{equation*}
			F_{Q} = \left\{ 
			\begin{aligned}
				&\mathcal{M}_t^2[ \nabla u^\ep] -  \mathcal{M}_t^2[ \nabla w_s], \quad &s \ge 4t,& \\
				&0, \quad &0<s<4t,&
			\end{aligned}
			\right.
		\end{equation*}
		and
		\begin{equation*}
			R_{Q} = \left\{ 
			\begin{aligned}
				& \mathcal{M}_t^2[ \nabla w_s], \quad &s \ge 4t,& \\
				&\mathcal{M}_t^2[ \nabla u^\ep], \quad &0<s<4t.&
			\end{aligned}
			\right.
		\end{equation*}	
		For any given $p>2$, we may choose $\theta$ sufficiently small with $C\theta^\sigma < \kappa_0$ 
		so that the requirement of Theorem \ref{thm.Shen} is satisfied. Consequently, we arrive at
		\begin{equation}\label{e.est56}
			\bigg( \dashint_{Q_{r}(x)} | \mathcal{M}_t^2[\nabla u^\ep] |^{p} \bigg)^{1/p} \le C \bigg( \dashint_{Q_{56r}(x)} |\mathcal{M}_t^2 [\nabla u^\ep] |^{2} \bigg)^{1/2} + C\bigg( \dashint_{Q_{56r}(x)} | \mathcal{M}_t^2[C_\theta F^\ep] |^{p} \bigg)^{1/p}
		\end{equation}
		for all $\ep\in(0,\theta(p))$ and $\ep/(4\theta) \le t\le r$ and $Q_{57r}(x)\subset Q_1(0)$. 
		Estimate \eqref{est.CZ.Stokes} now follows by a covering argument (see the proof of Lemma \ref{lem.selfimprove}) and Lemma \ref{lem.est.M.op.properties} (in order to adjust the size of balls and relax the condition $t \ge (4\theta)$ to $t \ge \ep$). To remove the smallness condition $\ep\in(0,\theta(p))$, we observe that the case $\theta(p) \le \ep \le t \le r \le \frac12$ is trivial as the constant $C$ is allowed to depend on $p$.
	\end{proof}

	\subsection{Bootstrap argument}\label{subsec.bootstrap}
	In this subsection, we apply the large-scale Calder\'{o}n-Zyg\-mund estimate proved previously to study the regularity of the stationary Navier-Stokes equations \eqref{intro.NS.ep}. Note that in Theorem \ref{thm.CZ.Stokes}, $F^\ep$ is a general function. We will take advantage of the nonlinearity $F^\ep = -u^\ep \otimes u^\ep$. As usual, the proof relies on a bootstrap argument.

	Throughout this subsection, we set $F^\ep = -u^\ep \otimes u^\ep$. To begin with, note that the Sobolev embedding theorem implies $F^\ep \in L^3$, which yields $\mathcal M_t[F^{\ep}] \in L^3$. Hence, \eqref{est.CZ.Stokes} holds with $p = 3$. To further improve the large-scale regularity, we need to lift the regularity of $F^\ep$ from that of $\nabla u^\ep$.
	
	For any $0\le a < b\le \infty$, define a new maximal function
	\begin{equation*}
		\quad \mathcal{M}_{(a,b)}^1[g](x) = \sup_{a<t<b}  \dashint_{Q_{t}(x)} |g|.
	\end{equation*}
	Note that $\mathcal{M}_{(0,\infty)}^1$ is the usual Hardy-Littlewood maximal function. Clearly, by the $L^p$ boundedness of the Hardy-Littlewood maximal function, $\mathcal{M}_{(a,b)}^1$ is uniformly bounded in $L^p$ space for $p\in (1,\infty)$.
	
	Fix $t>0$. Define
	\begin{equation*}
		K_{q}(r) = K_{q,t}(r) = \bigg( \dashint_{Q_{r}(0)} \mathcal{M}_t^2[\nabla u^\ep]^{q} \bigg)^{1/q}.
	\end{equation*}
	The following estimate is a sort of the large-scale Sobolev embedding theorem. 
	
	\begin{theorem}\label{thm.Sobolev}
		Let $L\in(0,\infty)$ and $\Omega$ be a bumpy John domain with constant $L$ according to Definition \ref{def.John2}. Let $\ep\le t\le r \le \frac17$ and $F^\ep = -u^\ep \otimes u^\ep$. Then for any $p>3$ and any $q$ satisfying
		\begin{equation}\label{e.pqcond}
			\frac{1}{q} < \frac{1}{2p} + \frac{1}{3},
		\end{equation}
		we have
		\begin{equation}\label{est.MtF}
			\bigg( \dashint_{Q_r(0)} |\mathcal{M}_t [F^\ep]|^p \bigg)^{1/p} \le Cr^2 \big( K_q(5r) \big)^2,
		\end{equation}
		where the constant $C$ depends only $L$, $p$ and $q$.
		\end{theorem}
	
	\begin{proof}
		Let $p>3$ and $q$ satisfy \eqref{e.pqcond}. Without loss of generality, we assume in addition $\frac1{2p}<\frac1q$. 
		Let $x\in B_{r,+}^\ep(0)$. 
		We first estimate 
		\begin{equation*}
			\mathcal{M}_t[F^\ep](x) = \bigg( \dashint_{Q_t(x)} |F^\ep|^2 \bigg)^{1/2} \le C\bigg( \dashint_{Q_t(x)} |u^\ep|^4 \bigg)^{1/2}.
		\end{equation*}
		Let $x = (x',x_3)$. We consider the cases $x_3 \ge t$ and $x_3 < t$ separately. Assume first $x_3 \ge t$ and let $N$ be the natural number so that $2^{N-1} t < x_3 \le 2^{N} t $.
		Note that $u^\ep$ vanishes in a large portion of $Q_{2^{N+1} t}(x)$. By the triangle inequality and the Poincar\'{e} inequality, we have
		\begin{equation*}
			\begin{aligned}
				\bigg( \dashint_{Q_t(x)} |u^\ep|^4 \bigg)^{1/4} & \le \bigg( \dashint_{Q_t(x)} |u^\ep - \dashint_{Q_{2t}(x)} u^\ep |^4 \bigg)^{1/4} \\
				&\quad +\sum_{j=1}^N \bigg| \dashint_{Q_{2^j t}(x)} u^\ep - \dashint_{Q_{2^{j+1} t}(x)} u^\ep \bigg| + \bigg| \dashint_{Q_{2^{N+1} t}(x)} u^\ep \bigg| \\
				& \le C \sum_{j = 0}^{N} 2^{j+1} t \bigg( \dashint_{Q_{2^{j+1}t}(x)} |\nabla u^\ep|^2 \bigg)^{1/2}.
			\end{aligned}
		\end{equation*}
		Now, let $\alpha \in (0,\min\{\frac{1}{3},\frac1q\})$ 
		and write
		\begin{equation}\label{est.Due.jt}
			\begin{aligned}
				& 2^{j+1} t \bigg( \dashint_{Q_{2^{j+1}t}(x)} |\nabla u^\ep|^2 \bigg)^{1/2} \\
				& \le C 2^{j+1} t \bigg( \dashint_{Q_{2^{j+1}t}(x)} \mathcal{M}_t[\nabla u^\ep]^2 \bigg)^{1/2} \\
				& \le C 2^{j+1} t \bigg( \dashint_{Q_{2^{j+1}t}(x)} \mathcal{M}_t[\nabla u^\ep]^q \bigg)^{1/q} \\
				& \le C 2^{j+1} t \bigg( \dashint_{Q_{2^{j+1}t}(x)} \mathcal{M}_t[\nabla u^\ep]^q \bigg)^{\alpha} \bigg( \dashint_{Q_{2^{j+1}t}(x)} \mathcal{M}_t[\nabla u^\ep]^q \bigg)^{1/q - \alpha} \\
				& \le C (2^{j+1} t)^{1-3\alpha} \bigg( \int_{Q_{2^{j+1}t }(x) } \mathcal{M}_t[\nabla u^\ep]^q \bigg)^\alpha \bigg( \dashint_{Q_{2^{j+1}t}(x)} \mathcal{M}_t[\nabla u^\ep]^q \bigg)^{1/q - \alpha} \\
				& \le C (2^{j+1} t)^{1-3\alpha} r^{3\alpha} \bigg( \dashint_{Q_{{5r}}(0) } \mathcal{M}_t[\nabla u^\ep]^q \bigg)^\alpha \bigg( \dashint_{Q_{2^{j+1}t}(x)} \mathcal{M}_t[\nabla u^\ep]^q \bigg)^{1/q - \alpha}.
			\end{aligned}
		\end{equation}
		Using the definition of $K_q$ and $\mathcal{M}^1_{(2t,{5r})}$, we obtain
		\begin{equation*}
			\begin{aligned}
				& 2^{j+1} t \bigg( \dashint_{Q_{2^{j+1}t}(x)} |\nabla u^\ep|^2 \bigg)^{1/2} \\
				& \quad \le C (2^{j+1} t)^{1-3\alpha} r^{3\alpha} (K_q({5r}))^{\alpha q} \Big( \mathcal{M}_{(2t,{5r})}^1 \big[ \mathcal{M}_t[\nabla u^\ep]^q \big](x) \Big)^{1/q - \alpha}.
			\end{aligned}
		\end{equation*}
		It follows that
		\begin{equation*}
			\begin{aligned}
				\bigg( \dashint_{Q_t(x)} |u^\ep|^4 \bigg)^{1/4} & \le C\sum_{j = 0}^{N} (2^{j+1} t)^{1-3\alpha} r^{3\alpha} (K_q({5r}))^{\alpha q} \Big( \mathcal{M}_{(2t,{5r})}^1 \big[ \mathcal{M}_t[\nabla u^\ep]^q \big](x) \Big)^{1/q - \alpha} \\
				& \le C r (K_q({5r}))^{\alpha q} \Big( \mathcal{M}_{(2t,{5r})}^1 \big[ \mathcal{M}_t[\nabla u^\ep]^q \big](x) \Big)^{1/q - \alpha},
			\end{aligned}
		\end{equation*}
		which yields
		\begin{equation}\label{est.MtFex}
			\mathcal{M}_t[F^\ep](x) \le Cr^2 (K_q({5r}))^{2\alpha q} \Big( \mathcal{M}_{(2t, {5r})}^1 [\mathcal{M}_t[\nabla u^\ep]^q ](x) \Big)^{2/q - 2\alpha}.
		\end{equation}
		On the other hand, if $x_3 < t$, then $B_{2t}(x)$ has a relatively large portion not contained in $\Omega^\ep$. Thus, the Sobolev-Poincar\'{e} inequality implies
		\begin{equation*}
			\bigg( \dashint_{Q_t(x)} |u^\ep|^4 \bigg)^{1/4} \le C\bigg( \dashint_{Q_{2t}(x)} |u^\ep|^4 \bigg)^{1/4} \le Ct \bigg( \dashint_{Q_{2t}(x)} |\nabla u^\ep|^2 \bigg)^{1/2}.
		\end{equation*}
		Using the same argument as \eqref{est.Due.jt}, we see that $\mathcal{M}_t[F^\ep](x)$ has the same bound as \eqref{est.MtFex} for $x_3 < t$.

		Since by assumption, $\frac1{2p}<\frac{1}{q} < \frac{1}{2p} + \frac13$ and $0<\alpha<\min\{\frac13,\frac1q\}$ is arbitrary, we may choose $\alpha$ so that $\frac{1}{q} > \frac{1}{2p} + \alpha$. 
		This implies $ p(\frac2q-2\alpha) > 1$. Thus, using the $L^{p(\frac2q-2\alpha)}$ boundedness of the Hardy-Littlewood maximal function, we obtain
		\begin{equation*}
			\begin{aligned}
				& \int_{Q_r(0)} |\mathcal{M}_t[F^\ep](x)|^p \dd x \\
				& \qquad \le C r^{2p} (K_q({5r}))^{2\alpha p q} \int_{Q_r(0)} \Big( \mathcal{M}_{(2t, {5r})}^1 [\mathcal{M}_t[\nabla u^\ep]^q ](x) \Big)^{p(2/q-2\alpha)} \dd x \\
				& \qquad \le C r^{2p} (K_q({5r}))^{2\alpha p q} \int_{Q_{{5r}}(0)} \Big(  \mathcal{M}_t[\nabla u^\ep] (x) \Big)^{2p(1-\alpha q)} \dd x.
			\end{aligned}
		\end{equation*}
		Consequently,
		\begin{equation*}
			\bigg( \dashint_{Q_r(0)} |\mathcal{M}_t[F^\ep]|^p \bigg)^{1/p} \le  Cr^2 (K_q({5r}))^{{2\alpha q}} (K_{2p(1-\alpha q)}({5r}))^{2(1-\alpha q)}.
		\end{equation*}
		Now, observe that we may choose $\alpha < \frac{1}{q} - \frac{1}{2p}$ but sufficiently close to $\frac{1}{q} - \frac{1}{2p}$. Then $q< 2p(1-\alpha q) \to q$ as $\alpha$ approaches $\frac{1}{q} - \frac{1}{2p}$. This implies
		\begin{equation}\label{est.Kq'}
			\bigg( \dashint_{Q_r(0)} |\mathcal{M}_t[F^\ep]|^p \bigg)^{1/p} \le Cr^2 \big( K_{2p(1-2\alpha)}({5r}) \big)^2 \le Cr^2 \big( K_{\hat q}({5r}) \big)^2.
		\end{equation}
		for any $\hat q> q$,
		where we also used the fact that $K_m(r) \le K_n(r)$ for any $1\le m\le n$. Finally, to recover the case with the exact exponent $q$, we may start with a $\tilde{q} < q$ still satisfying $\frac{1}{\tilde{q}} < \frac{1}{2p} + \frac{1}{3}$.
		Then \eqref{est.Kq'} holds for any $\hat q> \tilde{q}$, which includes the case $\hat q = q$. This proves the desired estimate.
	\end{proof}

	Now, a bootstrap argument between \eqref{est.CZ.Stokes} and \eqref{est.MtF} 
	shows that both $\mathcal{M}_t^2[\nabla u^\ep]$ and $\mathcal{M}_t^2[F^\ep]$ are in $L^p$ for any $p\geq 3$. 
	In the following, we use this to prove a large-scale H\"{o}lder's estimate for $F^\ep$, which plays an important role in the Lipschitz estimate in the next section.
	
	\begin{theorem}\label{thm.MFL3}
		Let $L\in(0,\infty)$ and $\Omega$ be a bumpy John domain with constant $L$ according to Definition \ref{def.John2}. Let $\ep\le t \le r\le \frac12$. 
		Let $M\ge 0$ be such that
		\begin{equation*}
			\bigg( \dashint_{B_{1,+}^\ep} |\nabla u^\ep|^2 \bigg)^{1/2}\le M.
		\end{equation*}
		For every $l>3$ and $\delta>0$ satisfying $l\delta<6$, we have
		\begin{equation}\label{est.MFL3}
			\bigg( \dashint_{Q_r} |\mathcal{M}_t [F^\ep]|^3 \bigg)^{1/3} \le Cr^{2-6/l} \big( M + M^{2(4- 6/l+\delta)} \big),
		\end{equation}
		where the constant $C$ depends only on 
		$L$, $l$ and $\delta$.
	\end{theorem}
	\begin{proof}
		Note that, by a similar argument as at the end of the proof of Theorem \ref{thm.CZ.Stokes}, we only have to prove \eqref{est.MFL3} when $\ep/N_0 \le t\le r \le 1/N_1$ for some $N_0, N_1 \ge 2$.
		Let $l>3$ and $\delta>0$ with $l\delta<6$ be given and fixed for the proof.   
		First of all, by the Sobolev embedding theorem, $\| F^\ep \|_{L^3(Q_1)} \le CM^2$. This implies
		$\| \mathcal{M}_t[F^\ep] \|_{L^3(Q_{1/2})} \le CM^2$. By Theorem \ref{thm.CZ.Stokes},
		\begin{equation*}
			\bigg( \dashint_{Q_{{1/8}}} |\mathcal{M}_t[\nabla u^\ep]|^3 \bigg)^{1/3} \le C(M + M^2).
		\end{equation*}
		Then, applying Theorem \ref{thm.Sobolev}, we obtain that for any $3\leq p<\infty$,
		\begin{equation*}
			\bigg( \dashint_{Q_{{1/40}}} |\mathcal{M}_t [F^\ep]|^{p} \bigg)^{1/{p}} \le C \bigg( \dashint_{Q_{{1/8}}} \mathcal{M}_t[\nabla u^\ep]^3 \bigg)^{2/3} \le C(M + M^4).
		\end{equation*}
		Now, using Theorem \ref{thm.CZ.Stokes} again combined with a covering argument, we derive from the last inequality that
		\begin{equation*}
			\bigg( \dashint_{Q_{{1/80}}} |\mathcal{M}_t [\nabla u^\ep]|^{p} \bigg)^{1/{p}} \le C(M+M^4).
		\end{equation*}
		Now, let $p>l$. By the interpolation, we have
		\begin{equation}\label{est.MtMM}
			\begin{aligned}
				\bigg( \dashint_{Q_{{1/80}}} |\mathcal{M}_t [\nabla u^\ep]|^{l} \bigg)^{1/l} & \le \bigg( \dashint_{Q_{{1/80}}} |\mathcal{M}_t [\nabla u^\ep]|^3 \bigg)^{\theta/3} \bigg( \dashint_{Q_{{1/80}}} |\mathcal{M}_t [\nabla u^\ep]|^{p} \bigg)^{(1-\theta)/p} \\
				& \le C(M+M^2)^{\theta} (M+M^4)^{1-\theta} \\
				& \le C(M + M^{4-2\theta}),
			\end{aligned}
		\end{equation}
		where
		\begin{equation*}
			\frac{1}{l} = \frac{\theta}{3} + \frac{1-\theta}{p}.
		\end{equation*}
		For the given $\delta \in (0,1)$, we want $4-2\theta = 4- 6/l+\delta$. This implies $\theta = 3/l - \delta/2$ and thus we may choose
		\begin{equation*}
			p = \frac{6(1-\frac{3}{l} + \frac{\delta}{2})}{\delta}.
		\end{equation*}
		One can easily verify that $\theta\in(0,1)$ by the assumption on $l$ and $\delta$.
		Consequently, we derive from \eqref{est.MtMM} that
		\begin{equation*}
			\bigg( \dashint_{Q_{{1/80}}} |\mathcal{M}_t [\nabla u^\ep]|^{l} \bigg)^{1/l} \le C\big(M + M^{4- 6/{l}+\delta} \big).
		\end{equation*}
		Finally, we apply Theorem \ref{thm.Sobolev} to obtain for $r\le 1/400$,
		\begin{equation*}
			\begin{aligned}
				\bigg( \dashint_{Q_r} |\mathcal{M}_t [F^\ep]|^3 \bigg)^{1/3} & \le Cr^2 \bigg( \dashint_{Q_{{5r}}} |\mathcal{M}_t [\nabla u^\ep]|^{l} \bigg)^{2/l} \\
				& \le Cr^{2-6/l} \bigg( \int_{Q_{{5r}}} |\mathcal{M}_t [\nabla u^\ep]|^l \bigg)^{2/l} \\
				& \le Cr^{2-6/l} \big( M + M^{2(4- 6/l+\delta)}\big).
			\end{aligned}
		\end{equation*}
		This ends the proof.
	\end{proof}

	Note that if $\nabla u^\ep$ itself is in $L^p$ for $p>3$, then Morrey's inequality implies that $u^\ep$ is $C^{0,1-3/p}$, which implies, since $u^\ep$ vanishes on the boundary,
	\begin{equation*}
		\bigg( \dashint_{Q_r} |F^\ep|^3 \bigg)^{1/3} \leq C r^{2-6/p},
	\end{equation*}
	where $C$ depends on $M$, $L$ and $p$. Hence, \eqref{est.MFL3} is consistent with the usual Morrey estimate.

	\section{Large-scale Lipschitz estimate}
	\label{sec.large}
	In this section, we will establish the large-scale Lipschitz estimate of $u^\ep$ and the oscillation estimate of $p^\ep$. We remark, for later use in Subsection \ref{subsec.1st.BL}, that Theorem \ref{theo.lip.nonlinear} implies the following Liouville theorem for the Stokes system
	\begin{equation}\label{eq.forLiouvillebis}
		\left\{
		\begin{array}{ll}
			-\Delta u+\nabla p=0 &\text{in }\Omega \\
			\nabla\cdot u=0 &\text{in }\Omega \\
			u=0 &\text{on }\partial\Omega, 
		\end{array}
		\right.
	\end{equation}
	where $\Omega$ is a John domain in Definition \ref{def.John2}. The proof of the following statement is standard.
	\begin{corollary}\label{cor.liouville}
		Let $\Omega$ be a bumpy John domain according to Definition \ref{def.John2}. Let $(u,p)$ be a weak solution of \eqref{eq.forLiouvillebis}. 
		If
		\begin{equation*}
			\lim_{R\to \infty} \frac{1}{R} \bigg( \dashint_{B_R(0)\cap \Omega} |u|^2 \bigg)^{1/2} = 0,
		\end{equation*}
		then $u \equiv 0$ (hence $p$ is constant).
	\end{corollary}

	\subsection{Sep-up and approximation}\label{subsec.set-up and approx.}
	First of all, we may write \eqref{intro.NS.ep} as a linear Stokes system
	\begin{equation}\tag{S$^\ep$}\label{S.ep}
		\left\{
		\begin{array}{ll}
			-\Delta u^\ep+\nabla p^\ep=
			{\nabla\cdot F^\ep} &\mbox{in}\ B^\ep_{1,+}\\
			\nabla\cdot u^\ep=0&\mbox{in}\ B^\ep_{1,+}\\
			u^\ep=0&\mbox{on}\ \Gamma^\ep_1,
		\end{array}
		\right.
	\end{equation}
	where $F^\ep = -u^\ep \otimes u^\ep$. As in the classical regularity theory for Stokes system, we will use the large-scale $C^{0,\alpha}$ estimate of $F^\ep$ in Theorem \ref{thm.MFL3} to prove the large-scale Lipschitz estimate. The proof is based on the excess decay method.

	Similarly to the large-scale Calder\'{o}n-Zygmund estimate of Theorem \ref{thm.CZ.Stokes}, we also need to approximate the Stokes system \eqref{S.ep} at all scales greater than $\ep$. 
	Fix $r\in[\ep,{\frac12}]$ and let $(v_r,q_r)$ be the weak solution of the following Stokes system
	\begin{equation}\tag{S$_{r}$}\label{approx.Stokes}
		\left\{
		\begin{array}{ll}
			-\Delta v_r+\nabla q_r 
			= 0 &\mbox{in}\ Q^\ep_{t_0 r} \\
			\nabla\cdot v_r=0&\mbox{in}\ Q^\ep_{t_0 r}\\
			v_r=u^\ep &\mbox{on}\ \partial Q^\ep_{t_0 r},
		\end{array}
		\right.
	\end{equation}
	where we have automatically extended $u^\ep$ across the bottom boundary by zero-extension and $t_0$ is a constant in the interval $[1,2]$ chosen analogously as in \eqref{est.co-area} and \eqref{eq.ws}. Note that \eqref{approx.Stokes} is a special case of \eqref{eq.ws} with $s = r$ and $y = 0$, which means the estimates \eqref{est1.lem,est.ws}-\eqref{est.ws.L3} hold also for $(v_r,q_r)$, in place of $(w_s,q_s)$.
	The following lemma is an analog of Lemma \ref{lem.Duews}.

	\begin{lemma}\label{lem.est.approx.}
		Let $L\in (0,\infty)$ and $\Omega$ be a bumpy John domain with constant $L$ according to Definition \ref{def.John2}. 
		Let $(u^\ep,p^\ep)$ and $(v_r,q_r)$ be weak solutions of \eqref{S.ep} and \eqref{approx.Stokes}, respectively. If $\ep\in(0,\frac{1}{10}]$ and $r\in[2\ep,\frac15]$, then
		\begin{equation}\label{est1.lem.est.approx.}
			\begin{aligned}
				&\bigg(\dashint_{B^\ep_{r,+}} |\nabla u^\ep - \nabla v_r|^2 \bigg)^{1/2} 
				+ \bigg(\dashint_{B^\ep_{r/2,+}} |p^\ep - q_r - \dashint_{B^\ep_{r/2,+}} (p^\ep-q_r)|^2 \bigg)^{1/2} \\
				&\le 
				C\big(\frac{\ep}{r}\big)^{1/12}
				\bigg(\dashint_{B^\ep_{5r,+}} |\nabla u^\ep|^2\bigg)^{1/2}
				+ C\bigg(\dashint_{{Q_{4r}}} |\mathcal{M}^2_{\ep} [F^\ep]|^3\bigg)^{1/3},
			\end{aligned}
		\end{equation}
		where $C$ depends only on $L$.
	\end{lemma}

	\begin{proof}
		Let us set $R^\ep_r=Q^\ep_{t_0 r}\cap\{-\ep< x_3\le2\ep\}$. 
		By examining the proof of Lemma \ref{lem.Duews}, we obtain
		\begin{equation}\label{est1.proof.lem.est.approx.}
			\begin{aligned}
				\int_{B^\ep_{t_0 r,+}} |\nabla (u^\ep - v_r)|^2 
				\le C\int_{R^\ep_r} |\nabla u^\ep |^2 + C\big(\frac{\ep}{r}\big)^{1/6} \int_{B^\ep_{2r,+}} |\nabla u^\ep |^2
				+ C\int_{B^\ep_{2r,+}} |F^\ep|^2.
			\end{aligned}
		\end{equation}
		From Lemma \ref{lem.est.M.op.properties} (iii) and Theorem \ref{thm.CZ.Stokes} with $p=3$,
		\begin{equation*}
			\begin{aligned}
				\bigg(\frac{1}{|B^\ep_{r,+}|}\int_{R^\ep_r} |\nabla u^\ep|^2\bigg)^{1/2} &\le C \bigg(\frac{1}{|B^\ep_{r,+}|}\int_{R^\ep_r} |\mathcal{M}^2_{\ep}[ \nabla u^\ep]|^2\bigg)^{1/2}\\
				&\le C\big(\frac{|R^\ep_r|}{|B^\ep_{r,+}|}\big)^{1/6} 
				\bigg(\dashint_{B^\ep_{r,+}} |\mathcal{M}^2_{\ep}[ \nabla u^\ep]|^3\bigg)^{1/3} \\
				&\le C\big(\frac{\ep}{r}\big)^{1/6}
				\bigg\{ \bigg(\dashint_{B^\ep_{5r,+}} |\nabla u^\ep|^2\bigg)^{1/2} + \bigg(\dashint_{{Q_{4r}}} |\mathcal{M}^2_{\ep}[ F^\ep ]|^3\bigg)^{1/3} \bigg\}.
			\end{aligned}
		\end{equation*}
		Inserting this into \eqref{est1.proof.lem.est.approx.}, we have
		\begin{equation}\label{est.Due-Dvr}
			\begin{aligned}
				&	\bigg(\dashint_{B^\ep_{r,+}} |\nabla u^\ep - \nabla v_r|^2 \bigg)^{1/2}\\
				& \le C\big(\frac{\ep}{r}\big)^{1/12} \bigg(\dashint_{B^\ep_{5r,+}} |\nabla u^\ep|^2\bigg)^{1/2} 
				+ C\bigg(\dashint_{{Q_{4r}}} |\mathcal{M}^2_{\ep}[ F^\ep ]|^3\bigg)^{1/3}.
			\end{aligned}
		\end{equation}
		
		Next, we estimate the pressure by using the Bogovskii lemma. The issue is that, in general, $B^\ep_{r,+}$ is not a John domain. By Definition \ref{def.John2}, for $r\ge 2\ep$, there exists a John domain $\Omega_r^\ep$ with constant $L$ 
		satisfying
		\begin{equation*}
			B_{r/2,+}^\ep \subset \Omega_{r}^\ep \subset B_{r,+}^\ep.
		\end{equation*}
		Note that $(u^\ep -v_r, p^\ep-q_r)$ satisfies
		\begin{equation*}
			-\Delta (u^\ep - v_r)+\nabla (p^\ep - q_r) 
			= \nabla\cdot F^\ep \quad \text{in } B_{r,+}^\ep.
		\end{equation*}
		Thus, we may use the Bogovskii lemma in $\Omega_r^\ep$ and \eqref{est.Due-Dvr} to obtain
		\begin{equation}\label{est.pe-qr.Omega}
			\begin{aligned}
				& \bigg( \dashint_{\Omega_{r}^\ep} | p^\ep- q_r -\dashint_{\Omega_{r}^\ep} (p^\ep - q_r) |^2  \bigg)^{1/2} \\
				& \le C\bigg(\dashint_{\Omega_{r}^\ep} |\nabla u^\ep - \nabla v_r|^2 \bigg)^{1/2} + C\bigg(\dashint_{\Omega_{r}^\ep} |F^\ep|^2\bigg)^{1/2} \\
				& \le C\big(\frac{\ep}{r}\big)^{1/12}
				\bigg(\dashint_{B^\ep_{5r,+}} |\nabla u^\ep|^2\bigg)^{1/2} + C \bigg(\dashint_{{Q_{4r}}} |\mathcal{M}^2_{\ep}[ F^\ep ]|^3\bigg)^{1/3},
			\end{aligned}
		\end{equation}
		where we also used the fact $|\Omega_{r}^\ep| \approx |B_{r,+}^\ep|$. Using a well-known fact
		\begin{equation*}
			\int_{E} |f - \dashint_E f|^2 = \inf_{a\in \R} \int_{E} |f - a|^2, \quad \text{for any open set } E,
		\end{equation*}
		we derive
		\begin{equation}\label{est.pe-qr}
			\begin{aligned}
				&	\bigg( \dashint_{B_{r/2,+}^\ep } | p^\ep- q_r -\dashint_{B_{r/2,+}^\ep} (p^\ep - q_r) |^2  \bigg)^{1/2} \\
				& \le \bigg( \dashint_{B_{r/2,+}^\ep } | p^\ep- q_r -\dashint_{\Omega_{r}^\ep} (p^\ep - q_r) |^2  \bigg)^{1/2} \\
				& \le C \bigg( \dashint_{\Omega_{r}^\ep } | p^\ep- q_r -\dashint_{\Omega_{r}^\ep } (p^\ep - q_r) |^2  \bigg)^{1/2}.
			\end{aligned}
		\end{equation}
		Combining \eqref{est.Due-Dvr}, \eqref{est.pe-qr.Omega} and \eqref{est.pe-qr}, we obtain the desired estimate.
	\end{proof}

	\begin{remark}\label{rmk.Standard.P.est}
		The pressure estimate in John domains in the proof of Lemma \ref{lem.est.approx.} is a standard technique that we will frequently use throughout this paper. It allows us to transfer the pressure estimate to the estimates of $\nabla u^\ep$ and $F^\ep$.
	\end{remark}

	\subsection{Excess decay}
	Let $\mathscr{P}_1 = \{ (ax_3, bx_3,0)~|~a,b\in \R \}$. Note that $\mathscr{P}_1$ consists of all the linear solutions (velocity component) of the Stokes equations in the whole space with the no-slip condition on $\{x_3 = 0\}$. These linear solutions are dubbed as {\it no-slip Stokes polynomials} of degree $1$.

	For a pair of function $(w^\ep,\pi^\ep)\in H^1(B^\ep_{r,+})^3\times L^2(B^\ep_{r,+})$ with $r \in(0,1]$, we set
	\begin{equation}\label{def.H}
		\begin{aligned}
			H(w^\ep,\pi^\ep;\rho)
			&=
			\inf_{P\in \mathscr{P}_1}
			\bigg(\dashint_{B^\ep_{\rho,+}} |\nabla w^\ep - \nabla P|^2 \bigg)^{1/2}
			 \\
			&\quad
			+ \sup_{s,t\in [1/16,1/4]}
			\bigg| \dashint_{B^\ep_{s\rho,+}} \pi^\ep - \dashint_{B^\ep_{t\rho,+}} \pi^{\ep} \bigg|, \quad \rho\in(0,r],
		\end{aligned}
	\end{equation}
	and
	\begin{equation}\label{def.Phi}
		\begin{aligned}
			\Phi(w^\ep,\pi^\ep;\rho)
			&= 
			\bigg(\dashint_{B^\ep_{\rho,+}} |\nabla w^\ep|^2 \bigg)^{1/2}
			 \\
			&\quad
			+ \sup_{s,t\in [1/16,1/4]}
			\bigg| \dashint_{B^\ep_{s\rho,+}} \pi^\ep - \dashint_{B^\ep_{t\rho,+}} \pi^{\ep} \bigg|, \quad \rho\in(0,r].
		\end{aligned}
	\end{equation}
	The quantity $H$ can be dubbed as a zeroth-order excess quantity. In Section \ref{sec.higher} we will consider higher-order excess quantities $H_{{\rm 1st}}$ and $H_{{\rm 2nd}}$ to address the large-scale $C^{1,\gamma}$ and $C^{2,\gamma}$ regularity.
	
	Moreover, for a pair of function $(w_r,\pi_r)\in H^1(Q^\ep_{r})^3\times L^2(Q^\ep_{r})$ with $r\in(0,1]$, we set
	\begin{equation}\label{def.H.tilde}
		\begin{aligned}
			\widetilde{H}(w_r,\pi_r;\rho)
			&= 
			\inf_{P\in \mathscr{P}_1} 
			\bigg(\dashint_{Q^\ep_{\rho}} |\nabla w_r - \nabla P|^2 \bigg)^{1/2}
			 \\
			&\quad
			+ \sup_{s,t\in [1/16,1/4]}\bigg| \dashint_{Q^\ep_{s\rho}} \pi_r - \dashint_{Q^\ep_{t\rho}} \pi_r \bigg|, \quad \rho\in(0,r].
		\end{aligned}
	\end{equation}

	The following lemma states the comparability between $H(v_r,q_r;\theta r)$ and $\widetilde{H}(v_r,q_r;\theta r)$.
	%
	\begin{lemma}\label{lem.est.comparability}
		Let $L\in(0,\infty)$ and $\Omega$ be a bumpy John domain with constant $L$ according to Definition \ref{def.John2}. Fix $\ep\in(0,\frac14]$, $r\in[\ep,\frac14]$ and let $(v_r,q_r)$ satisfy \eqref{approx.Stokes}. Then we have the following statements. 
		\begin{enumerate}[label=(\roman*)]
			\item For all $\theta\in(0,1]$,
			\begin{equation}\label{est1.lem.est.comparability}
				\begin{aligned}
					H(v_r,q_r;\theta r)
					&\le 
					C\widetilde{H}(v_r,q_r;\theta r) 
					+ C\theta^{-1}\big(\frac{\ep}{r}\big)^{1/2}
					\bigg(\dashint_{Q^\ep_{\theta r}} |\nabla v_r|^2\bigg)^{1/2}.
				\end{aligned}
			\end{equation}
			\item For all $\theta\in(0,1]$,
			\begin{equation}\label{est2.lem.est.comparability}
				\begin{aligned}
					\widetilde{H}(v_r,q_r;\theta r)
					&\le
					CH(v_r,q_r;2\theta r) 
					+ 
					C\theta^{-5/2}\big(\frac{\ep}{r}\big)
					\bigg(\dashint_{Q^\ep_{2\theta r}} |\nabla v_r|^2\bigg)^{1/2}.
				\end{aligned}
			\end{equation}
		\end{enumerate}
		Here $C$ depends only on $L$.
	\end{lemma}
	%
	\begin{proof}
		(i) We first deal with $v_r$. Since $B_{\theta r,+}^\ep \subset Q_{\theta r}^\ep$ and $|B_{\theta r,+}^\ep| \approx |Q_{\theta r}^\ep|$, we have
		\begin{align*}
			\inf_{P\in \mathscr{P}_1} 
			\bigg(\dashint_{B^\ep_{\theta r,+}} |\nabla v_r-\nabla P|^2\bigg)^{1/2} 
			\le C\inf_{P\in \mathscr{P}_1}
			\bigg(\dashint_{Q^\ep_{\theta r}} |\nabla v_r- \nabla P|^2\bigg)^{1/2}.
		\end{align*}
		On the other hand, the triangle inequality implies
		\begin{equation}\label{est1.proof.lem.est.comparability}
			\begin{aligned}
				\sup_{s,t\in [1/16,1/4]}\bigg|\dashint_{B^\ep_{\theta sr,+}} q_r - \dashint_{B^\ep_{\theta tr,+}} q_r \bigg|
				&\le
				\sup_{s,t\in [1/16,1/4]}\bigg|\dashint_{Q^\ep_{\theta sr}} q_r - \dashint_{Q^\ep_{\theta tr}} q_r \bigg| \\
				&\quad
				+2\sup_{\rho\in [1/16,1/4]}\bigg|\dashint_{Q^\ep_{\theta \rho r}}q_r - \dashint_{B^\ep_{\theta \rho r,+}}q_r \bigg|.
			\end{aligned}
		\end{equation}
		Combining the above two inequalities, we obtain
		\begin{align}\label{est2.proof.lem.est.comparability}
			H(v_r,q_r;\theta r)
			\le
			C\widetilde{H}(v_r,q_r;\theta r)
			+2\sup_{\rho\in [1/16,1/4]}\bigg|\dashint_{Q^\ep_{\theta \rho r}}q_r - \dashint_{B^\ep_{\theta \rho r,+}}q_r \bigg|.
		\end{align}
		Since $|Q^\ep_{\theta \rho r}\setminus B^\ep_{\theta \rho r,+}|$ is less than $C\ep (\theta \rho r)^2$, a direct computation yields
		\begin{equation}\label{est3.proof.lem.est.comparability}
			\begin{aligned}
				\bigg|\dashint_{Q^\ep_{\theta \rho r}} q_r - \dashint_{B^\ep_{\theta \rho r,+}} q_r\bigg| 
				&\le 
				\Big(\frac{1}{|B^\ep_{\theta \rho r,+}|} - \frac{1}{|Q^\ep_{\theta \rho r}|}\Big) 
				\bigg|\int_{Q^\ep_{\theta \rho r}} (q_r-\dashint_{Q^\ep_{\theta r}} q_r)\bigg| \\
				&\quad
				+ \frac{1}{|B^\ep_{\theta \rho r,+}|}
				\bigg|\int_{Q^\ep_{\theta \rho r}\setminus B^\ep_{\theta \rho r,+}} (q_r-\dashint_{Q^\ep_{\theta r}} q_r)\bigg|\\
				&\le 
				\Big(\frac{|Q^\ep_{\theta \rho r}\setminus B^\ep_{\theta \rho r,+}| |Q^\ep_{\theta r}|^{1/2}}{|B^\ep_{\theta \rho r,+}| |Q^\ep_{\theta \rho r}|^{1/2}}
				+ \frac{|Q^\ep_{\theta \rho r}\setminus B^\ep_{\theta \rho r,+}|^{1/2} |Q^\ep_{\theta r}|^{1/2}}{|B^\ep_{\theta \rho r,+}|}\Big)\\
				&\quad
				\times\bigg(\dashint_{Q^\ep_{\theta r}} |q_r-\dashint_{Q^\ep_{\theta r}} q_r|^2 \bigg)^{1/2} \\
				&\le 
				C\Big(\theta^{-1}\rho^{-5/2}\big(\frac{\ep}{r}\big) 
				+ \theta^{-1/2}\rho^{-2}\big(\frac{\ep}{r}\big)^{1/2}\Big)
				\bigg(\dashint_{Q^\ep_{\theta r}} |\nabla v_r|^2\bigg)^{1/2},
			\end{aligned}
		\end{equation}
		where we have applied the H\"older inequality in the second inequality and the Bogovskii lemma in $Q_{\theta r}^\ep$ in the third inequality. Noting $\rho\ge \frac{1}{16}$ and using \eqref{est2.proof.lem.est.comparability} and \eqref{est3.proof.lem.est.comparability}, we obtain the first 
		inequality \eqref{est1.lem.est.comparability}.
		
		(ii) 
		Let $P_*\in \mathscr{P}_1$ be such that 
		\begin{align*}
			\bigg(\dashint_{B^\ep_{\theta r,+}} |\nabla v_r - \nabla P_*|^2 \bigg)^{1/2}
			=\inf_{P\in \mathscr{P}_1} \bigg(\dashint_{B^\ep_{\theta r,+}} |\nabla v_r - \nabla P|^2 \bigg)^{1/2}.
		\end{align*}
		Since $v_r(x)-P_*(x+\ep{\bf e}_3)$ is a weak solution to \eqref{approx.Stokes} with the same pressure $q_r$ and $P_*(x+\ep {\bf e}_3) = P_*(x) + \ep(\nabla P_*) {\bf e}_3$, by the Bogovskii lemma in $Q^\ep_{\theta r}$, we see that
		\begin{equation}\label{est4'.proof.lem.est.comparability}
			\begin{aligned}
				\sup_{s,t\in [1/16,1/4]}\bigg|\dashint_{Q^\ep_{\theta sr,+}} q_r - \dashint_{Q^\ep_{\theta tr,+}} q_r \bigg|
				&\le 2\sup_{\rho\in [1/16,1/4]}
				\bigg( 
				\dashint_{Q^\ep_{\theta \rho r}}
				|q_r - \dashint_{Q^\ep_{\theta r}} q_r|^2  \bigg)^{1/2} \\
				&\le 
				C
				\bigg( 
				\dashint_{Q^\ep_{\theta r}}
				|\nabla v_r - \nabla P_*|^2  \bigg)^{1/2}.
			\end{aligned}
		\end{equation}
		Moreover, by the Caccioppoli inequality (see Lemma \ref{lem.StrongCaccioppoli}) in rectangular region $Q_{2\theta r}^\ep$, we have
		\begin{equation}\label{est4''.proof.lem.est.comparability}
			\begin{aligned}
			&\bigg(\dashint_{Q^\ep_{\theta r}} |\nabla v_r - \nabla P_*|^2 \bigg)^{1/2} \\		
			&\le \frac{C}{\theta r} 
			\bigg(\dashint_{Q^\ep_{2\theta r}} |v_r - P_*(x+\ep{\bf e}_3)|^2 \dd x \bigg)^{1/2} \\
			&\le \frac{C}{\theta r} 
			\bigg(\dashint_{B^\ep_{2\theta r,+}} |v_r - P_*|^2 \bigg)^{1/2}
			+  C\theta^{-1} \big(\frac{\ep}{r}\big)|\nabla P_*| \\
			&\quad 
			+ \frac{C}{(\theta r)^{5/2}} 
			\bigg\{\bigg(\int_{Q^\ep_{2\theta r}\setminus B^\ep_{2\theta r,+}} |v_r|^2 \bigg)^{1/2}
			+ \bigg(\int_{Q^\ep_{2\theta r}\setminus B^\ep_{2\theta r,+}} |P_*(x+\ep{\bf e}_3)|^2 \bigg)^{1/2}
			\bigg\}.
			\end{aligned}
\end{equation}
	Now \eqref{est4'.proof.lem.est.comparability} and \eqref{est4''.proof.lem.est.comparability} combined with the Poincar\'e inequality imply
	\begin{equation}\label{est4.proof.lem.est.comparability}
			\begin{aligned}
				&\widetilde{H}(v_r,q_r;\theta r) \\
				&\le
				CH(v_r,q_r;2\theta r) 
				+  C\theta^{-1} \big(\frac{\ep}{r}\big)|\nabla P_*| \\
				&\quad
			+ \frac{C}{(\theta r)^{5/2}} 
\bigg\{\bigg(\int_{Q^\ep_{2\theta r}\setminus B^\ep_{2\theta r,+}} |v_r|^2 \bigg)^{1/2}
+ \bigg(\int_{Q^\ep_{2\theta r}\setminus B^\ep_{2\theta r,+}} |P_*(x+\ep{\bf e}_3)|^2 \bigg)^{1/2}
\bigg\}.
			\end{aligned}
		\end{equation}
		By the definition of $P_*$, we have
		\begin{align}\label{est5'.proof.lem.est.comparability}
			|\nabla P_*| \le C\bigg(\dashint_{B^\ep_{2\theta r,+}} |\nabla v_r|^2 \bigg)^{1/2} \le C
			\bigg(\dashint_{Q^\ep_{2\theta r}} |\nabla v_r|^2 \bigg)^{1/2}.
		\end{align}
		Consequently,
		\begin{equation}\label{est5.proof.lem.est.comparability}
			\begin{aligned}
				&\frac{C}{(\theta r)^{5/2}} 
				\bigg\{\bigg(\int_{Q^\ep_{2\theta r}\setminus B^\ep_{2\theta r,+}} |v_r|^2 \bigg)^{1/2}
				+ \bigg(\int_{Q^\ep_{2\theta r}\setminus B^\ep_{2\theta r,+}} |P_*(x+\ep{\bf e}_3)|^2 \bigg)^{1/2}
				\bigg\} \\
				&\le
				\frac{C}{(\theta r)^{5/2}} 
				\bigg\{\ep\bigg(\int_{Q^\ep_{2\theta r}\setminus B^\ep_{2\theta r,+}} |\nabla v_r|^2 \bigg)^{1/2}
				+ \ep (\ep \theta^2 r^2)^{1/2}|\nabla P_*|\bigg\} \\
				&\le
				C\Big(\theta^{-5/2}\big(\frac{\ep}{r}\big) 
				+ \theta^{-3/2}\big(\frac{\ep}{r}\big)^{3/2}\Big)
				\bigg(\dashint_{Q^\ep_{2\theta r}} |\nabla v_r|^2 \bigg)^{1/2}.
			\end{aligned}
		\end{equation}
		Hence, we obtain \eqref{est2.lem.est.comparability} from \eqref{est4.proof.lem.est.comparability} combined with \eqref{est5'.proof.lem.est.comparability} and \eqref{est5.proof.lem.est.comparability}. This completes the proof of Lemma \ref{lem.est.comparability}.
	\end{proof}
	%

		%
\begin{lemma}\label{lem.est.approx.func.}
	Let $L\in(0,\infty)$ and $\Omega$ be a bumpy John domain with constant $L$ according to Definition \ref{def.John2}. Fix $\alpha\in(0,1]$ arbitrarily.
	For all $\ep\in(0,\frac12]$, $r\in[\ep,\frac12]$, 
	$\theta\in(0,\frac18]$, 
	and $(v_r,q_r)$ satisfying \eqref{approx.Stokes}, 
		\begin{equation}\label{est1.lem.est.approx.func.}
			\begin{aligned}
				\widetilde{H}(v_r,q_r;\theta r)
				&\le 
				C\Big(\theta^\alpha + \theta^{-1}\big(\frac{\ep}{r}\big)\Big) 
				\widetilde{H}(v_r,q_r; r),
			\end{aligned}
		\end{equation}
	where $C$ depends only on $L$ and $\alpha$.
\end{lemma}
%
\begin{proof}
	By the regularity of the Stokes equations in flat domains, 
	\begin{align*}
		v_r\in C^{1,\alpha}(Q^\ep_{r/2})\,, \qquad 
		q_r\in C^{0,\alpha}(Q^\ep_{r/2}).
	\end{align*}
	%
	Let ${\bf e}_3=(0,0,1)$. The boundary $C^{1,\alpha}$ estimate of $v_r$ on $\{ x_3 = -\ep\}$ implies
	\begin{align*}
		|v_r(x)-v_r(-\ep{\bf e}_3)-(x_3+\ep)\partial_3v_r(-\ep{\bf e}_3)|
		\le C\frac{|x+\ep{\bf e}_3|^{1+\alpha}}{r^{\alpha}}
		\bigg(\dashint_{Q^\ep_{r/2}} |\nabla v_r|^2\bigg)^{1/2}.
	\end{align*}
	Note that $\partial_3v_{r,3}(-\ep{\bf e}_3)=0$ by the condition $\nabla\cdot v_{r}=0$. Thus, from $v_r(-\ep{\bf e}_3)=0$, there exists $\widetilde{P}(x)=(\partial_3v_{r,1}(-\ep{\bf e}_3),\partial_3v_{r,2}(-\ep{\bf e}_3),0) x_3 \in \mathscr{P}_1$ and
		\begin{align*}
			|v_r(x)-\widetilde{P}(x+\ep {\bf e}_3)|
			\le
			C\frac{|x_3+\ep{\bf e}_3|^{1+\alpha}}{r^{\alpha}} 
			\bigg(\dashint_{Q^\ep_{r/2}} |\nabla v_r|^2\bigg)^{1/2},
		\end{align*}
		for all $x\in Q_{\theta r}^\ep$. 
		Since $v_r(x)-\widetilde{P}(x+\ep {\bf e}_3)$ is a weak solution to \eqref{approx.Stokes} with the same pressure $q_r$, by the Caccioppoli inequality (see Lemma \ref{lem.StrongCaccioppoli}) in $Q_{2\theta r}^\ep$, we have
		\begin{align*}
			\begin{aligned}
				\bigg(\dashint_{Q_{\theta r}^\ep} |\nabla v_r-\nabla \widetilde{P}(x+\ep {\bf e}_3)|^2\dd x\bigg)^{1/2} 
				&\le \frac{C}{\theta r} \bigg(\dashint_{Q_{2\theta r}^\ep} |v_r-\widetilde{P}(x+\ep {\bf e}_3)|^2\dd x\bigg)^{1/2} \\		
				&\le
				C\Big(\theta^\alpha + \theta^{-1}\big(\frac{\ep}{r}\big)\Big)
				\bigg(\dashint_{Q^\ep_{r/2}} |\nabla v_r|^2\bigg)^{1/2}.
			\end{aligned}
		\end{align*}
		Then the observation $\widetilde{P}(x+\ep {\bf e}_3) = \widetilde{P}(x) + \ep(\nabla \widetilde{P}) {\bf e}_3$ yields
		\begin{equation}\label{est1.proof.lem.est.approx.func.}
			\bigg(\dashint_{Q_{\theta r}^\ep} |\nabla v_r-\nabla \widetilde{P}|^2\bigg)^{1/2} 
			\le
			C\Big(\theta^\alpha + \theta^{-1}\big(\frac{\ep}{r}\big)\Big)
			\bigg(\dashint_{Q^\ep_{r/2}} |\nabla v_r|^2\bigg)^{1/2}.
		\end{equation}		
	The $C^{0,\alpha}$ estimate of $q_r$ implies
	\begin{align*}
		|q_r(x)-q_r(0)|
		&\le C\frac{|x|^\alpha}{r^\alpha} 
		\bigg(\dashint_{Q^\ep_{r/2}} |q_r-\dashint_{Q^\ep_{r/2}}q_r|^2\bigg)^{1/2}.
	\end{align*}
	Then by the Bogovskii lemma 
	in a Lipschitz domain $Q^\ep_{r/2}$,
	we have
	\begin{align*}
		|q_r(x)-q_r(0)|
		&\le C\frac{|x|^\alpha}{r^\alpha}
		\bigg(\dashint_{Q^\ep_{r/2}} |\nabla v_r|^2\bigg)^{1/2},
	\end{align*}
	which results in
	\begin{equation}\label{est2.proof.lem.est.approx.func.}
		\begin{aligned}
			\sup_{s,t\in [1/16,1/4]}\bigg|\dashint_{Q_{\theta sr}^\ep} q_r - \dashint_{Q_{\theta tr}^\ep} q_r \bigg| 
			&\le C\theta^\alpha 
			\bigg(\dashint_{Q_{r/2}^\ep} |\nabla v_r|^2\bigg)^{1/2}.
		\end{aligned}
	\end{equation}
		Hence on the one hand, by \eqref{est1.proof.lem.est.approx.func.} and \eqref{est2.proof.lem.est.approx.func.} we see that
		\begin{align}\label{est3.proof.lem.est.approx.func.}
			\begin{aligned}
				\widetilde{H}(v_r,q_r;\theta r)
				& \le 
				C\Big(\theta^\alpha + \theta^{-1}\big(\frac{\ep}{r}\big)\Big)
				\bigg(\dashint_{Q^\ep_{r}} |\nabla v_r|^2\bigg)^{1/2}.
			\end{aligned}
		\end{align}
		On the other hand, since $v_r(x)-P(x+\ep{\bf e}_3)$, for any $P\in \mathscr{P}_1$, is a weak solution to \eqref{approx.Stokes} with the same pressure $q_r$ and $P(x+\ep {\bf e}_3) = P(x) + \ep(\nabla P) {\bf e}_3$, we may apply \eqref{est3.proof.lem.est.approx.func.} to 
		$v_r(x)-P(x+\ep{\bf e}_3)$ and obtain
		\begin{equation}\label{est4.proof.lem.est.approx.func.}
			\widetilde{H}(v_r -P(x+\ep {\bf e}_3) ,q_r;\theta r) 
			\le 
			C\Big(\theta^\alpha + \theta^{-1}\big(\frac{\ep}{r}\big)\Big)
			\bigg(\dashint_{Q^\ep_{r}} 
			|\nabla v_r- \nabla P|^2\dd x\bigg)^{1/2}.
		\end{equation}
		In particular, we may choose $P = P_*$ that minimizes
		\begin{equation*}
			\bigg(\dashint_{Q^\ep_{r}} |\nabla v_r- \nabla P |^2\bigg)^{1/2}.
		\end{equation*}
		Then, it is clear that
		\begin{equation}\label{est5.proof.lem.est.approx.func.}
			\bigg(\dashint_{Q^\ep_{r}} |\nabla v_r- \nabla P_* |^2\bigg)^{1/2} \le \widetilde{H}(v_r,q_r; r).		
		\end{equation}
		Then the estimate \eqref{est1.lem.est.approx.func.} follows from \eqref{est4.proof.lem.est.approx.func.} with $P = P_*$ and \eqref{est5.proof.lem.est.approx.func.}. The proof is complete.
\end{proof}

	\begin{lemma}\label{lem.est.approx.func.adjusted}
		Let $L\in(0,\infty)$ and $\Omega$ be a bumpy John domain with constant $L$ according to Definition \ref{def.John2}. Let $\alpha\in(0,1]$ be the number in Lemma \ref{lem.est.approx.func.}. For all $\ep\in(0,\frac14]$, $r\in[\ep,\frac14]$, 
		$\theta\in(0,\frac18]$, 
		and $(v_r,q_r)$ satisfying \eqref{approx.Stokes},
		\begin{equation}\label{est1.lem.est.approx.func.adjusted}
			\begin{aligned}
				H(v_r,q_r;\theta r) 
				&\le C\theta^\alpha H(v_r,q_r;2r) 
				+ C \theta^{-5/2}
				\big(\frac{\ep}{r}\big)^{1/2}
				\bigg(\dashint_{Q_{2r}} |\nabla v_r|^2\bigg)^{1/2},
			\end{aligned}
		\end{equation}
		where $C$ depends only on $L$ and $\alpha$.
	\end{lemma}
	%
	\begin{proof}
		The estimate \eqref{est1.lem.est.approx.func.adjusted} follows readily from Lemma \ref{lem.est.comparability} and Lemma \ref{lem.est.approx.func.}.
	\end{proof}
	\begin{lemma}\label{lem.excess.decay}
		Let $L\in(0,\infty)$ and $\Omega$ be a bumpy John domain with constant $L$ according to Definition \ref{def.John2}. Let $(u^\ep,p^\ep)$ be a weak solution of \eqref{S.ep}
		and let $\alpha\in(0,1]$ be the number in Lemma \ref{lem.est.approx.func.}. For all $\ep\in(0,\frac1{32}]$,  $r\in[ 2\ep,\frac1{16}]$ and $\theta\in(0,\frac18]$,
		\begin{equation}\label{est1.lem.excess.decay}
			\begin{aligned}
				H(u^\ep, p^\ep;\theta r)
				&\le 
				C\theta^\alpha H(u^\ep, p^\ep; 2r) 
				+ C\theta^{-3}\big(\frac{\ep}{r}\big)^{1/12} \Phi(u^\ep,p^\ep;{16r}) \\
				&\quad
				+ C\theta^{-3}\bigg(\dashint_{{Q_{10r}}} |\mathcal{M}^2_{\ep} [F^\ep]|^3\bigg)^{1/3},
			\end{aligned}
		\end{equation}
		where $C$ depends only on $L$ and $\alpha$.
	\end{lemma}
	%
	\begin{proof}
		The triangle inequality and Lemma \ref{lem.est.approx.func.adjusted} imply
		\begin{equation}\label{est1.proof.lem.excess.decay}
			\begin{aligned}
				H(u^\ep, p^\ep;\theta r)
				&\le H(v_r, q_r;\theta r) + H(u^\ep-v_r, p^\ep-q_r;\theta r) \\
				&\le C\theta^\alpha H(v_r, q_r;2r) + H(u^\ep-v_r, p^\ep-q_r;\theta r) \\
				&\quad
				+ C\theta^{-5/2} \big(\frac{\ep}{r}\big)^{1/2}
				\bigg(\dashint_{Q^\ep_{2r}} |\nabla v_r|^2\bigg)^{1/2} \\
				&\le 
				C\theta^\alpha H(u^\ep, p^\ep;2r) \\
				&\quad
				+ C\theta^\alpha H(u^\ep-v_r, p^\ep-q_r;r) + H(u^\ep-v_r, p^\ep-q_r;\theta r) \\
				&\quad
				+ C {\theta^{-5/2}} \big(\frac{\ep}{r}\big)^{1/2}
				\bigg(\dashint_{B^\ep_{{4r},+}} |\nabla u^\ep|^2\bigg)^{1/2},
			\end{aligned}
		\end{equation}
		where in the last line the energy estimate of \eqref{approx.Stokes} is applied. By the definition of $H$, we find
		\begin{equation}\label{est2.proof.lem.excess.decay}
			\begin{aligned}
				&\theta^\alpha H(u^\ep-v_r, p^\ep-q_r;r) + H(u^\ep-v_r, p^\ep-q_r;\theta r) \\
				&\le
				C(\theta^\alpha+\theta^{-3})\\
				&\quad\times
				\bigg\{
				\bigg(\dashint_{B^\ep_{r,+}} |\nabla u^\ep- \nabla v_r|^2\bigg)^{1/2} 
				+ \sup_{\rho\in [1/16,1/4]} \dashint_{B^\ep_{\rho r,+}} |p^\ep-q_r- \dashint_{B^\ep_{r/2,+}} (p^\ep-q_r)|
				\bigg\}.
			\end{aligned}
		\end{equation}
		The Poincar\'{e} inequality and Lemma \ref{lem.est.approx.} imply
		\begin{equation}\label{est3.proof.lem.excess.decay}
			\begin{aligned}
				&
					\bigg(\dashint_{B^\ep_{r,+}} |\nabla u^\ep- \nabla v_r|^2\bigg)^{1/2} 
					+ \sup_{\rho\in [1/16,1/4]} \dashint_{B^\ep_{\rho r,+}} |p^\ep-q_r- \dashint_{B^\ep_{r/2,+}} (p^\ep-q_r)| \\
				&\le 
				C
				\bigg\{
				\bigg(\dashint_{B^\ep_{r,+}} |\nabla u^\ep-\nabla v_r|^2 \bigg)^{1/2}
				+ \bigg(\dashint_{B^\ep_{r/2,+}} |p^\ep-q_r- \dashint_{B^\ep_{r/2,+}} (p^\ep-q_r)|^2\bigg)^{1/2}
				\bigg\} \\
				&\le 
				C\big(\frac{\ep}{r}\big)^{1/12}
				\bigg(\dashint_{B^\ep_{5r,+}} |\nabla u^\ep|^2\bigg)^{1/2}
				+ C\bigg(\dashint_{{Q_{4r}}} |\mathcal{M}^2_{\ep}[F^\ep]|^3\bigg)^{1/3}.
			\end{aligned}
		\end{equation}
		Now from \eqref{est1.proof.lem.excess.decay} to \eqref{est3.proof.lem.excess.decay}, we obtain the desired estimate \eqref{est1.lem.excess.decay} 
		by the definition of $\Phi$ in \eqref{def.Phi}. This completes the proof.
	\end{proof}
	%
	
	\subsection{Iteration}
	In the following two lemmas, we prove some properties of $H$ and $\Phi$ needed when iterating \eqref{est1.lem.excess.decay}.
	\begin{lemma}\label{lem.h}
		Let $L\in(0,\infty)$ and $\Omega$ be a bumpy John domain with constant $L$ according to Definition \ref{def.John2}. Let $(u^\ep,p^\ep)$ be a weak solution of \eqref{S.ep}. There exists a function $h(r)$ defined on $[\ep,\frac12]$ such that
		\begin{align}
			h(r)&\le C\big(H(u^\ep, p^\ep;r)+\Phi(u^\ep, p^\ep;r)\big), \label{est1.lem.h} \\
			\Phi(u^\ep, p^\ep;r)&\le C\big(H(u^\ep, p^\ep;r)+h(r)\big), \label{est2.lem.h} \\
			\sup_{r_1,r_2\in[r,2r]}|h(r_1)-h(r_2)|&\le CH(u^\ep, p^\ep;2r). \label{est3.lem.h}
		\end{align}
		Here $C$ depends only on $L$. Notice that the function $h$ depends on $u^\ep$.
	\end{lemma}
	%
	\begin{proof}
		The proof is similar to \cite[Lemma 6.1]{GZ} and hence we provide the outline of the proof. Let $P_r\in \mathscr{P}_1$ be such that
		\begin{align*}
			\bigg(\dashint_{B^\ep_{r,+}} |\nabla u^\ep - \nabla P_r|^2 \bigg)^{1/2}
			= \inf_{P\in \mathscr{P}_1}  \bigg(\dashint_{B^\ep_{r,+}} |\nabla u^\ep - \nabla P|^2 \bigg)^{1/2}.
		\end{align*}
		We define 
		\begin{align*}
			h(r)=|\nabla P_r|, \quad r\in\big[\ep,\frac12\big].
		\end{align*}
		Then the inequality \eqref{est1.lem.h} follows from
		\begin{align*}
			h(r)
			&\le 
			C \bigg(\dashint_{B^\ep_{r,+}} |\nabla P_r |^2 \bigg)^{1/2} 
			\\
			&\le C\big(H(u^\ep, p^\ep;r)+\Phi(u^\ep, p^\ep;r)\big)
		\end{align*}
		and \eqref{est2.lem.h} is trivial by definition. For \eqref{est3.lem.h}, we observe that for any $r_1,r_2\in[r,2r]$
		\begin{align*}
			|h(r_1)-h(r_2)|
			&\le 
			C\bigg(\dashint_{B^\ep_{r,+}} |\nabla P_{r_1}-\nabla P_{r_2}|^2 \bigg)^{1/2} 
			 \\
			&\le CH(u^\ep, p^\ep;2r).
		\end{align*}
		This completes the proof of Lemma \ref{lem.h}.
	\end{proof}
	%

\begin{lemma}\label{lem.Phi}
	Let $L\in(0,\infty)$ and $\Omega$ be a bumpy John domain with constant $L$ according to Definition \ref{def.John2}. Let $(u^\ep,p^\ep)$ be a weak solution of \eqref{S.ep}. Then for $\ep\in(0,{\frac{1}{8}}]$, $r\in[{2\ep},\frac14]$,
	\begin{align}\label{est1.lem.Phi}
		\sup_{\tau\in[r,2r]} \Phi(u^\ep, p^\ep;\tau) \le C\Phi(u^\ep, p^\ep;2r)
		+ \bigg(\dashint_{{Q_{2r}}} |\mathcal{M}^2_{\ep} [F^\ep]|^3\bigg)^{1/3},
	\end{align}
	where $C$ depends only on $L$.
\end{lemma}
%
\begin{proof}
	Let $\tau\in[r,2r]$. A simple computation implies
	\begin{equation}\label{est1.proof.lem.Phi}
		\begin{aligned}
			\bigg(\dashint_{B^\ep_{\tau,+}} |\nabla u^\ep|^2 \bigg)^{1/2}
			\le C
			\bigg(\dashint_{B^\ep_{2r,+}} |\nabla u^\ep|^2 \bigg)^{1/2}.
		\end{aligned}
	\end{equation}
	For the pressure estimate, by a similar argument as in \eqref{est.pe-qr}, 
	\begin{equation}\label{est2.proof.lem.Phi}
		\begin{aligned}
			&\sup_{s,t\in [1/16,1/4]}
			\bigg|\dashint_{B^\ep_{s\tau,+}} p^\ep - \dashint_{B^\ep_{t\tau,+}} p^\ep\bigg| \\
			&\le
			C\dashint_{B^\ep_{r/2,+}}|p^\ep-\dashint_{B^\ep_{r/2,+}} p^\ep| \\
			&\le
			C\bigg\{
			\bigg(\dashint_{B^\ep_{r,+}} |\nabla u^\ep|^2\bigg)^{1/2}
			+ \bigg(\dashint_{B^\ep_{r,+}} |F^\ep|^2\bigg)^{1/2}
			\bigg\},
		\end{aligned}
	\end{equation}
	where we need to assume $r\ge 2\ep$.
	Then the assertion \eqref{est1.lem.Phi} follows from \eqref{est1.proof.lem.Phi}, \eqref{est2.proof.lem.Phi} and Lemma \ref{lem.est.M.op.properties} (iii).
\end{proof}

We now state the iteration lemma. Its 
proof is given in Appendix \ref{app.it}.
\begin{lemma}\label{lem.iteration}
	Let $H, \Phi, h:(0,1] \to [0,\infty)$ be nonnegative functions. 
	Let $\ep \in (0,\frac{1}{48}]$. Suppose that there exist positive constants $C_0$, $B_0$, $\alpha, \beta$ and $\theta\in(0,{\frac18}]$ so that
	\begin{subequations}
		\begin{align}
			H(\theta r)
			&\le 
			\frac12 H({2r}) + C_0 \Big(\big(\frac{\ep}{r}\big)^{\alpha} \Phi(16r) + B_0 r^\beta\Big), 
			\quad r\in [\ep,\frac{1}{16}],
			\label{lem.iteration.assump.a} \\
			H(r) &\le C_0\Phi(r), 
			\quad r\in [\ep,\frac12],
			\label{lem.iteration.assump.b} \\
			\sup_{\tau\in[r,2r]} \Phi(\tau) &\le C_0\big(\Phi(2r) + B_0 r^\beta\big),
			\quad r\in [\ep,\frac14],
			\label{lem.iteration.assump.c} \\
			h(r) &\le C_0\big(H(r) + \Phi(r)\big), 
			\quad r\in [\ep,\frac12],
			\label{lem.iteration.assump.d} \\
			\Phi(r) &\le C_0\big(H(r) + h(r)\big), 
			\quad r\in [\ep,\frac12],
			\label{lem.iteration.assump.e} \\
			\sup_{r_1,r_2\in[r,2r]}|h(r_1)-h(r_2)| &\le C_0H(2r), 
			\quad r\in [\ep,\frac14].
			\label{lem.iteration.assump.f}
		\end{align}
	\end{subequations}
	Then, 
	\begin{equation}\label{est1.lem.iteration}
		\int_{\ep}^{1/2} \frac{H(t)}{t} \dd t + \sup_{r\in[\ep,1/2]}\Phi(r) 
		\le C(\Phi(\frac12) + B_0 ),
	\end{equation}
	where the constant $C$ depends only on $C_0, \alpha, \beta$ and $\theta$. 
\end{lemma}

%
%
%
\begin{proofx}{Theorem \ref{theo.lip.nonlinear}}
	In the following proof, we actually only need to show (\ref{est2.theo.lip.nonlinear}) for the case $N_0 \ep \le r \le 1/N_1$ for some $N_0, N_1 \ge 2$. The case $1/2 \ge r\ge 1/N_1$ follows trivially by enlarging the size of the cube and a standard pressure estimate (see Remark \ref{rmk.Standard.P.est}); the case $\ep \le r\le N_0 \ep$ follows from the case $r = N_0 \ep$. From the previous lemmas, we can choose $N_0 = 4$ and $N_1 = 16$. Hence, we may assume without loss of generality that $r\in [4\ep, \frac{1}{16}]$.

	We apply Lemma \ref{lem.iteration} to $H(r) = H(u^\ep,p^\ep;r)$ and $\Phi(r) = \Phi(u^\ep,p^\ep; r)$. 
	Choose $\theta$ sufficiently small so that we have $C\theta^\alpha\le 1/2$ in \eqref{est1.lem.excess.decay} in Lemma \ref{lem.excess.decay}.
	We need to verify the conditions in Lemma \ref{lem.iteration}. Note that \eqref{lem.iteration.assump.b} is obvious and \eqref{lem.iteration.assump.d}-\eqref{lem.iteration.assump.f} follow from Lemma \ref{lem.h}. To verify \eqref{lem.iteration.assump.a} from Lemma \ref{lem.excess.decay} and verify \eqref{lem.iteration.assump.c} from Lemma \ref{lem.Phi} (with $\ep$ replaced by $4\ep$), it suffices to note that Theorem \ref{thm.MFL3} implies
	\begin{equation}\label{est.MFep.large-scale}
		\bigg(\dashint_{{Q_{r}}} |\mathcal{M}^2_{\ep} [F^\ep]|^3\bigg)^{1/3} \le C (M+M^{4+2\beta+2\delta}) r^\beta,
	\end{equation}
	for any $\beta\in (0,2)$, $\delta\in(0,1)$ with $\beta+\delta<2$ and $r\in [\ep.\frac12]$. Hence, we may apply Lemma \ref{lem.iteration} with $B_0 = C(M + M^{4+2\beta+2\delta})$ to obtain
	\begin{equation}\label{est.HPHi}
		\begin{aligned}
			&\int_{4\ep}^{1/2} \frac{H(u^\ep, p^\ep; t)}{t} d t + \sup_{{{r\in[4\ep,1/2]}}}\Phi(u^\ep, p^\ep; r) \\
			&\le C\big(\Phi\big(u^\ep, p^\ep;\frac12) + (M + M^{4+2\beta+2\delta})\big) \\
			& \le C(M + M^{4+2\beta+2\delta}),
		\end{aligned}
	\end{equation}
	where in the last inequality, we have used a standard pressure estimate (see Remark \ref{rmk.Standard.P.est}) to bound $\Phi\big(u^\ep, p^\ep;\frac12)$ by $C(M+M^2)$. Hence, for $r\in [4\ep,\frac{1}{16}]$,
	\begin{align*}
		\bigg(\dashint_{B^\ep_{r,+}} |\nabla u^\ep|^2\bigg)^{1/2}
		&\le 
		C\big(\Phi\big(u^\ep, p^\ep;\frac12) + (M + M^{4+2\beta+2\delta})\big) \\
		& \le C(M + M^{4+2\beta+2\delta}),
	\end{align*}
	which proves the desired estimate of the velocity $u^\ep$.

	Next, we give an estimate for the pressure. For $r\in[\ep,\frac14]$, we observe that
	\begin{align*}
		\bigg(\dashint_{B^\ep_{r,+}} |p^\ep-\dashint_{B^\ep_{1/2,+}} p^{\ep}|^2\bigg)^{1/2}
		\le
		\bigg(\dashint_{B^\ep_{r,+}} |p^\ep-\dashint_{B^\ep_{r,+}} p^{\ep}|^2\bigg)^{1/2} 
		+ \bigg|\dashint_{B^\ep_{r,+}} p^{\ep} - \dashint_{B^\ep_{1/2,+}} p^{\ep}\bigg|.
	\end{align*}
	Using the technique as in \eqref{est.pe-qr} and by the Bogovskii lemma, the desired estimate of $\nabla u^\ep$ just proved and (\ref{est.MFep.large-scale}), we have
	\begin{align*}
		\bigg(\dashint_{B^\ep_{r,+}} |p^\ep-\dashint_{B^\ep_{r,+}} p^{\ep}|^2\bigg)^{1/2} 
		&\le
		C \bigg\{\bigg(\dashint_{B^\ep_{2r,+}} |\nabla u^\ep|^2\bigg)^{1/2}
		+ \bigg(\dashint_{B^\ep_{2r,+}} |F^\ep|^2\bigg)^{1/2} \bigg\} \\
		& \le C(M + M^{4+2\beta+2\delta}).
	\end{align*}
	On the other hand, let $N\in\N$ 
	be such that $2^{N}r\in[{\frac1{32}, \frac1{16}}]$. Then
	\begin{align*}
		\bigg|\dashint_{B^\ep_{r,+}} p^{\ep} - \dashint_{B^\ep_{1/2,+}} p^{\ep}\bigg|
		\le 
		\sum_{j=0}^{N-1} \bigg|\dashint_{B^\ep_{2^jr,+}} p^{\ep} - \dashint_{B^\ep_{2^{j+1}r,+}} p^{\ep}\bigg| 
		+ \bigg|\dashint_{B^\ep_{2^Nr,+}} p^{\ep} - \dashint_{B^\ep_{1/2,+}} p^{\ep}\bigg|.
	\end{align*}
	Now, observe that for each $j = {0,1,\cdots,N-1}$,
	\begin{align*}
		\bigg|\dashint_{B^\ep_{2^jr,+}} p^{\ep} - \dashint_{B^\ep_{2^{j+1}r,+}} p^{\ep}\bigg|
		&\le
		4\int_{2^{j+3}r}^{2^{j+4}r}
		\frac{1}{\tilde{r}} \sup_{s,t\in[1/16,1/4]}
		\bigg|\dashint_{B^\ep_{s\tilde{r},+}} p^{\ep} - \dashint_{B^\ep_{t\tilde{r},+}} p^{\ep}\bigg| \dd \tilde{r}.
	\end{align*}
	Thus, \eqref{est.HPHi} leads to
	\begin{align*}
		\sum_{j=0}^{{N-1}} \bigg|\dashint_{B^\ep_{2^jr,+}} p^{\ep} - \dashint_{B^\ep_{2^{j+1}r,+}} p^{\ep}\bigg| 
		&\le
		4\int_{\ep}^{1/2}
		\frac{1}{\tilde{r}} \sup_{s,t\in[1/16,1/4]}
		\bigg|\dashint_{B^\ep_{s\tilde{r},+}} p^{\ep} - \dashint_{B^\ep_{t\tilde{r},+}} p^{\ep}\bigg| \dd \tilde{r} \\
		&\le
		C(M + M^{4+2\beta+2\delta}).
	\end{align*}
	Finally, by the same trick as in \eqref{est.pe-qr}, we obtain
	\begin{align*}
		\begin{aligned}
			\bigg|\dashint_{B^\ep_{2^Nr,+}} p^{\ep} - \dashint_{B^\ep_{1/2,+}} p^{\ep}\bigg|
			&\le
			C\bigg\{\bigg(\dashint_{B^\ep_{1,+}} |\nabla u^\ep|^2\bigg)^{1/2}
			+ \bigg(\dashint_{B^\ep_{1,+}} |F^\ep|^2\bigg)^{1/2} \bigg\} \\
			& \le C(M+M^2).
		\end{aligned}
	\end{align*}
	Summarizing up the above estimates, we obtain the desired estimate for the pressure $p^\ep$. 
	This completes the proof of Theorem \ref{theo.lip.nonlinear}.
\end{proofx}
%

	%
	\section{Boundary layers in bumpy John domains}
	\label{sec.bl}
	%
	As seen in the previous section, the no-slip Stokes polynomials of degree $1$ (i.e., the basis of $\mathscr{P}_1$)
	\begin{equation}\label{Taylor.poly.deg1}
		P^{(11)} = (x_3, 0 ,0),\quad P^{(12)} = (0, x_3 ,0)
	\end{equation}
	are the key ingredients for the large-scale Lipschitz estimate. 
	Their trace on non-flat bumpy boundaries can be corrected by adding boundary layer correctors. Consequently, one obtains polynomial solutions of the Stokes equations in the bumpy John domains considered in this paper. 
	
	In Subsection \ref{subsec.Taylor.polys}, we determine the no-slip Stokes polynomials of $2$ by explicit computation. The boundary layer equations are introduced as well. Subsections \ref{subsec.1st.BL} and \ref{subsec.2nd.BL} are respectively devoted to the analysis of the first-order and the second-order boundary layer equations. The estimates for the Green function, obtained in Appendix \ref{app.green} using the large-scale Lipschitz estimate of Theorem \ref{theo.lip.nonlinear}, play a fundamental role. We summarize the estimates for the boundary layers in Subsection \ref{subsec.Ests.BL}. These estimates are key to the theory of higher-order regularity in Section \ref{sec.higher}.

	%
	\subsection{No-slip Stokes polynomials}\label{subsec.Taylor.polys}
	%
	Let $u$ be a solution of $-\Delta u + \nabla p = 0 $ and $ \nabla\cdot u = 0$ in $Q_{1,+}(0)$ 
	and $u = 0$ on $
	\partial\R^3_+\cap B_1(0)$. The real analyticity of $u$ in $Q_{1/2,+}(0)$ 
	is classical and well-known; see \cite{Mas67,Giga83}. Here we want to identify the form of the no-slip Stokes polynomials 
	of degree $2$ of $u$ at $0$.

	Let $P(x) = (P_1(x), P_2(x), P_3(x))$ be the no-slip Stokes polynomials 
	of degree $2$ of $u$ at $0$. First of all, since $u = 0$ on $\partial\R^3_+$, 
	then we must have
	\begin{equation}\label{eq.P}
		\begin{aligned}
			P_1(x) &= a_1 x_3 + b_{11}x_1 x_3 + b_{12}x_2 x_3 + b_{13} x_3^2,\\
			P_2(x) &= a_2 x_3 + b_{21}x_1 x_3 + b_{22}x_2 x_3 + b_{23} x_3^2,\\
			P_3(x) &= b_{31}x_1 x_3 + b_{32}x_2 x_3 + b_{33} x_3^2.
		\end{aligned}
	\end{equation}
	The linear part is familiar. So let us concentrate on the quadratic part. Note that there are no terms $x_1^2, x_2^2$, or $x_1 x_2$, because $u = 0$ on the boundary. If there is no further restriction on $u$, then there are $9$ free variables $b_{ij}, 1\le i,j\le 3$, as shown in \eqref{eq.P}.
	If $\nabla\cdot u =0$ in $Q_{1/2,+}(0)$, 
	then we claim that $P$ is also divergence-free. If this claim is true, then we must have
	\begin{equation*}
		b_{11} + b_{22} + 2b_{33} = 0, \quad b_{31} = b_{32} = 0.
	\end{equation*}
	Because of this restriction on the coefficients, the dimension for the homogeneous no-slip Stokes polynomials of degree $2$ becomes $6$. We can find basis polynomials as follows:
	\begin{equation}\label{Taylor.poly.deg2}
		\begin{aligned}
			&P^{(21)} = (x_2x_3, 0 ,0),\quad P^{(22)} = (x_3^2, 0 ,0),\\
			&P^{(23)} = (0, x_1x_3, 0), \quad P^{(24)} = (0, x_3^2, 0),\\
			&P^{(25)} = (-2x_1 x_3, 0 , x_3^2), \quad P^{(26)} = (0, -2x_2 x_3, x_3^2).
		\end{aligned}
	\end{equation}
	Note that these polynomials are solutions to the stationary Stokes system with associated pressure $L^{(2j)}$ 
	given by
	\begin{equation}\label{Taylor.poly.deg2.pressure}
		\begin{aligned}
			&L^{(2j)}(x) = 0, \quad \text{for } j=1,3.\\
			&L^{(22)}(x) = 2x_1,\ L^{(24)}(x) = 2x_2,\\
			\text{and } &L^{2j}(x) = 2x_3, \quad \text{for } j = 5,6.
		\end{aligned}
	\end{equation}

	Now, let us show the claim that $P$ is divergence-free. Since $u = P + O(|x|^3)$, we have that $\nabla\cdot u = \nabla\cdot P + O(|x|^2) = 0$ in $\{x_3\ge 0\}\cap B_{1/2}(0)$. Because of $\nabla\cdot P = C_0 + C_1\cdot x$ for some $C_0\in \R$ and $C_1\in \R^3$, we see that $C_0 + C_1\cdot x = O(|x|^2)$. Hence we must have $C_0 = 0$ and $C_1 = 0$; otherwise, it is easy to find a contradiction by taking $x = \delta C_1$ or $-\delta C_1$ for sufficiently small $\delta$. 
	
	Similarly to the linear solution pairs $(P^{(1j)}, 0)$, the fundamental fact about the polynomial pairs constructed above is that $(P^{(2j)}, L^{(2j)})$ are quadratic solutions of Stokes equations in the upper half-space $\R^3_{+}$, namely
	\begin{equation*}
		\left\{
		\begin{array}{ll}
			-\Delta P^{(2j)}+\nabla L^{(2j)}=0 &\text{in } \R^3_+ \\
			\nabla\cdot P^{(2j)}=0 &\text{in } \R^3_+ \\
			P^{(2j)}=0 &\text{on } \partial\R^3_+.
		\end{array}
		\right.
	\end{equation*}
	To study the $C^{1,\gamma}$ and $C^{2,\gamma}$ regularity of \eqref{intro.NS.ep}, the linear and quadratic solutions of Stokes equations in $\R^3_+$ are not enough. We need to construct linear and quadratic solutions in $\Omega$ which vanish on $\partial \Omega$, where $\Omega$ is a bumpy John half-space in the sense of Definition \ref{def.John2}. These solutions will be constructed based on $(P^{(1j)}, 0)$ and $(P^{(2j)}, L^{(2j)})$. Observe that $P^{(ij)}$ does not vanish on $\partial \Omega$. Therefore we have to introduce new correctors, called boundary layers, in order to correct the boundary discrepancy on $\partial\Omega$. Precisely, we will 
	show the existence of weak solutions (with corresponding sublinear or subquadratic growth) of the following boundary layer equations
	\begin{equation}
		\tag{BL$_{i {\rm th}}^{(j)}$}\label{BLij}
		\left\{
		\begin{array}{ll}
			-\Delta v+\nabla q=0 &\text{in } \Omega \\
			\nabla\cdot v=0 &\text{in } \Omega \\
			v+P^{(ij)}=0 &\text{on } \partial\Omega,
		\end{array}
		\right.
	\end{equation}
	where $i\in\{1,2\}$. Here a couple $(v,q)\in H^1_{{\rm loc}}(\overline{\Omega})^3\times L^2_{{\rm loc}}(\overline{\Omega})$ is said to be a weak solution of \eqref{BLij} if it satisfies: (i) $\nabla\cdot v=0$ in the sense of distributions, (ii) $\chi(v+P^{(ij)})\in H^1_0(\Omega)^3$ for any $\chi\in C^\infty_0(\R^3)$, and (iii) the weak formulation:
	\begin{align}\label{BLij.weak}
		\int_\Omega \nabla v \cdot\nabla \phi
		- \int_\Omega q (\nabla\cdot \phi)
		=0, \quad \text{for any } \phi\in C^\infty_0(\Omega)^3.
	\end{align}
	%

	\subsection{First-order boundary layers}\label{subsec.1st.BL}
	We consider the first-order boundary layer equations
	\begin{equation}\tag{BL$_{{\rm 1st}}^{(j)}$}\label{BL1j}
		\left\{
		\begin{array}{ll}
			-\Delta v+\nabla q=0 &\text{in }\Omega \\
			\nabla\cdot v=0 &\text{in }\Omega \\
			v+P^{(1j)}=0 &\text{on }\partial\Omega,
		\end{array}
		\right.
	\end{equation}
	for $j\in\{1,2\}$. The solvability of \eqref{BL1j} follows from the next statement.
	\begin{theorem}\label{prop.BL1j}
		Let $L\in(0,\infty)$ and $\Omega$ be a bumpy John domain with constant $L$ according to Definition \ref{def.John2}. For $j\in\{1,2\}$, there exists a unique weak solution $(v^{(1j)},q^{(1j)})\in H^1_{{\rm loc}}(\overline{\Omega})^3\times L^2_{{\rm loc}}(\overline{\Omega})$ of \eqref{BL1j}
		satisfying
		\begin{align}\label{est1.prop.BL1j}
			&\sup_{\xi\in\Z^2}
			\int_{\Omega\cap(\xi+(0,1)^2)\times\R}
			\big(|\nabla v^{(1j)}|^2+| q^{(1j)}|^2\big)
			\le 
			C,
		\end{align}
		where the constant $C$ depends only on $L$.
	\end{theorem}
	
	In the work \cite{HP} the well-posedness of the system \eqref{BL1j} was proved over Lipschitz graphs by a domain decomposition method: coupling of the Stokes problem in a bumpy channel $\Omega\cap\{x_3<0\}$ with the Stokes problem in the flat half-space $\{x_3>0\}$ via a nonlocal Dirichlet to Neumann boundary condition at the interface $\{x_3=0\}$. We face considerable technical difficulties when trying to adapt this strategy to the case of bumpy John domains. Indeed, the local energy estimates in the bumpy channel require to estimate the pressure, or to work with divergence-free test functions. In either case, we need to construct a Bogovskii operator for a sequence of exhausting domains containing $\Omega\cap\{|x'|\leq k,\, x_3<0\}$ with a constant uniform in $k$. The construction of the Bogovskii operator of Theorem \ref{theo.acosta} by \cite{ADM06} relies on connecting any point in the bumpy John domain to a fixed neighborhood of a reference point $\tilde x$. Such a procedure gives, for a slim domain such as  $\Omega\cap\{|x'|\leq k,\, x_3<0\}$, a constant in the estimate \eqref{e.estbogtheom} that scales proportionately to the horizontal size $k$ of the domain. We are unable to take advantage of the small vertical extent of the domain to provide a modified construction of the Bogovskii operator. This would be needed to carry out the downward iteration on the local energy estimates, also called Saint-Venant estimates, in \cite{HP}.
	
	Here we take advantage of the fact that we already proved large-scale Lipschitz estimates by the quantitative method, without relying on boundary layers as in \cite{HP}. Therefore, we develop a new strategy using the large-scale Lipschitz estimate to prove the existence of solutions to \eqref{BL1j}. We rely on the Green kernel estimates proved in Appendix \ref{app.green}. For $N\in\R$, let us set 
	\begin{equation}\label{e.defOmegaN}
		\Omega_{\leq N}:=\Omega\cap\{ z_3 \leq N \},
		\qquad \Omega_{\geq N}:=\Omega\cap\{ z_3 \geq N \}.
	\end{equation}
	We also define $\Omega_{< N}$ and $\Omega_{> N}$ in a similar manner.

	\begin{proof}[Proof of Theorem \ref{prop.BL1j}]\ We denote $(v^{(1j)},q^{(1j)})$ by $(v,q)$, not to burden the notation. 
		
		\noindent (\textbf{Uniqueness}) Let $P^{(1j)}=0$ in \eqref{BL1j}. Then the Liouville-type result, Corollary \ref{cor.liouville}, implies $v=0$ in the class
		\begin{align*}
			&\sup_{\xi\in\Z^2}
			\int_{\Omega\cap(\xi+(0,1)^2)\times\R}
			|\nabla v|^2 < \infty.
		\end{align*}
		This implies $q=0$ in the class \eqref{est1.prop.BL1j} 
		as well from the equations.

		\noindent (\textbf{Existence})\\
		\noindent{\bf{Step 1: lifting the boundary data.}} Let ${\eta_-(x_3)}$ be a smooth cut-off function such that 
		\begin{equation}\label{def.eta1}
			\begin{aligned}
				&\text{$\eta_-(t)$ is smooth and non-negative,} \\
				&\text{$\eta_-(t)={1}$ if $t<3$ and $\eta_-(t)=0$ if $t>4$.}
			\end{aligned}
		\end{equation}
		By writing $w = v + {\eta_-} P^{(1j)}$, we see that $w$ satisfies
		\begin{equation}\label{eq1.proof.prop.BL1j}
			\left\{
			\begin{array}{ll}
				-\Delta w+\nabla q= F:= -\Delta({\eta_-} P^{(1j)}) &
				\text{in }\Omega \\
				\nabla\cdot w=0 &
				\text{in }\Omega \\
				w = 0 &
				\text{on }\partial\Omega.
			\end{array}
			\right.
		\end{equation}
		Notice that $F$ is a bounded function supported in a slim channel $S:= \{x\in \R^3|\ 3\le x_3 \le 4 \}$. Thus, the problem is reduced to finding a weak solution of \eqref{eq1.proof.prop.BL1j} satisfying
		\begin{align}\label{est1.proof.prop.BL1j}
			&\sup_{\xi\in\Z^2}
			\int_{\Omega\cap(\xi+(0,1)^2)\times\R}
			\big(|\nabla w|^2+|q|^2\big)
			\le 
			C\|F\|_{L^\infty}^2.
		\end{align}
		We rely on the representation of $w$ and $q$ by the Green kernel
		\begin{equation*}
			w(x) = \int_{\Omega} G(x,y) F(y) \dd y,\qquad
			q(x) = \int_{\Omega} \Pi(x,y)\cdot F(y) \dd y.
		\end{equation*}
		Thanks to the properties of the Green function $(G,\Pi)$, it suffices to prove that $\nabla w$ and $q$ are well-defined and satisfy the estimate \eqref{est1.proof.prop.BL1j}. 
		
		In the following proof, we take the zero-extension of $(G,\Pi)$ as is done in Appendix \ref{app.green}.
		
		\smallskip
		
		\noindent{\bf{Step 2: estimate on $\Omega_{\ge 8}$.}} 
		For any $y\in S$ and $x\in \Omega_{\ge 8}$, by Proposition \ref{prop.DG} (i)
		\begin{equation*}
			|\nabla_x G(x,y)|
			\le \frac{C}{|x-y|^3}.
		\end{equation*}
		Then it follows from the H\"{o}lder's inequality that
		\begin{equation*}
			\begin{aligned}
				|\nabla w(x)| 
				& \le  
				\int_{S} |\nabla_x G(x,y)| |F(y)| \dd y \\
				& \le  \|F\|_{L^\infty} 
				\int_{S} \frac{C}{|x-y|^3} \dd y \\
				& \le \frac{C \|F\|_{L^\infty}}{x_3}.
			\end{aligned}
		\end{equation*}
		A similar computation using Proposition \ref{prop.pressureest} (i) gives the same bound for the pressure $q(x)$ with $x_3 \ge 8$. Consequently,
		\begin{equation}\label{est2.proof.prop.BL1j}
			\begin{aligned}
				\sup_{\xi\in\Z^2}
				\int_{\Omega_{\ge 8}\cap(\xi+(0,1)^2)\times\R} 
				\big(|\nabla w|^2+|q|^2\big) 
				& \le C\|F\|_{L^\infty}^2.
			\end{aligned}
		\end{equation}
		
		\smallskip
		
		\noindent{\bf{Step 3: estimate on $\Omega_{\le 8}$.}} 
		Fix $\xi\in\Z^2$ arbitrarily. For simplicity, we denote the cubes in $\R^3$ centered at $(\xi,0)$ by $Q_R(\xi) = (\xi,0) + (-R,R)^3$. We would like to estimate $|\nabla w|$ and $|q|$ in the cube $Q_8(\xi)$.  Notice that $\Omega_{\le 8} \subset \cup_{\xi \in \Z^2\times \{ 0 \}} Q_8(\xi)$ with finite overlaps.

		Taking a cube $Q_{40}(\xi)$, we decompose $F$ into two parts as $F=F\chi_{Q_{40}(\xi)}+F(1-\chi_{Q_{40}(\xi)})$. Correspondingly, we decompose $(w,q)$ in singular and regular parts, namely $(w,q) =( w_{\rm sing},q_{\rm sing}) + (w_{\rm reg},q_{\rm reg})$, where
		\begin{equation*}
			\left\{
			\begin{aligned}
				w_{\rm sing}(x) &= \int_{\Omega} G(x,y) F(y)\chi_{Q_{40}(\xi)}(y) \dd y,\\
				q_{\rm sing}(x) &= \int_{\Omega} \Pi(x,y)\cdot F(y)\chi_{Q_{40}(\xi)}(y) \dd y,
			\end{aligned}
			 \right.
		\end{equation*}
		and
		\begin{equation*}
			\left\{
			\begin{aligned}
			w_{\rm reg}(x) &= \int_{\Omega} G(x,y) F(y) (1-\chi_{Q_{40}(\xi)}(y)) \dd y, \\
			q_{\rm reg}(x) & = \int_{\Omega} \Pi(x,y)\cdot F(y) (1-\chi_{Q_{40}(\xi)}(y)) \dd y.
		\end{aligned}
		\right.
		\end{equation*}
	
	To estimate the regular part $(w_{\rm reg},q_{\rm reg})$ in $Q_8(\xi)$, we use \eqref{est.DxG.ptwise} and \eqref{est.DG.3-1} in Proposition \ref{prop.DG} to obtain
	\begin{equation*}
	\bigg( \int_{Q_8(\xi)} |\nabla_x G(x,y)|^2 \dd x \bigg)^{1/2} \le \frac{C}{|(\xi,0)-y|^3},
	\end{equation*}
	for any $y\in S\setminus Q_{40}(\xi)$. As a result,
	\begin{equation}\label{est.wreg.Q8}
		\begin{aligned}
		\bigg( \int_{Q_8(\xi)} |\nabla w_{\rm reg}|^2 \bigg)^{1/2} & \le \| F\|_{L^\infty} \int_{S \setminus Q_{40}(\xi)} \bigg( \int_{Q_8(\xi)} |\nabla_x G(x,y)|^2 \dd x \bigg)^{1/2} \dd y \\
		& \le \| F\|_{L^\infty} \int_{S \setminus Q_{40}(\xi) } \frac{C}{|(\xi,0)-y|^3} \dd y \\
		& \le C\| F\|_{L^\infty}.
		\end{aligned}
	\end{equation}
	Similarly, by using \eqref{est.Pi-1} and \eqref{est.Pi-2}, we can derive the estimate of $q_{\rm reg}$,
	\begin{equation}\label{est.qreg.Q8}
		\bigg( \int_{Q_8(\xi)} |q_{\rm reg}|^2 \bigg)^{1/2} \le C\| F\|_{L^\infty}.
	\end{equation}

	Next, we consider the singular part $(w_{\rm sing},q_{\rm sing})$, which actually is a weak solution of
		\begin{equation*}\label{eq.wsing}
			\left\{
			\begin{array}{ll}
				-\Delta w_{\rm sing}+\nabla q_{\rm sing}= F\chi_{Q_{40}(\xi)} &
				\text{in }\Omega \\
				\nabla\cdot w_{\rm sing}=0 &
				\text{in }\Omega \\
				w_{\rm sing} = 0 &
				\text{on }\partial\Omega.
			\end{array}
			\right.
		\end{equation*}
		Note that the energy relation yields
		\begin{equation}\label{sing1.proof.prop.BL1j}
			\| \nabla w_{\rm sing} \|_{L^2(\Omega)} \le C\|F\|_{L^\infty},
		\end{equation}
		where $C$ is independent of $\xi$. This gives the local $L^2$ boundedness of $w_{\rm sing}$ in the channel $\Omega_{\le 8}$. On the other hand, the argument in Step 2, using \eqref{est.Pi-1}, implies that for any $x_3\ge 8$,
		\begin{equation}\label{sing2.proof.prop.BL1j}
			|q_{\rm sing}(x)|  \le \frac{C \|F\|_{L^\infty}}{x_3}.
		\end{equation}
		Let $\Omega_{20}(\xi)$ be the John domain given by Definition \ref{def.John2} satisfying
		\begin{equation*}
			\Omega \cap Q_{20}(\xi) \subset  \Omega_{20}(\xi) \subset \Omega \cap Q_{40}(\xi).
		\end{equation*}
		By the Bogovskii lemma and \eqref{sing1.proof.prop.BL1j},
		\begin{equation}\label{sing3.proof.prop.BL1j}
			\bigg( \dashint_{\Omega_{20}(\xi)} 
			| q_{\rm sing} - \dashint_{\Omega_{20}(\xi)} q_{\rm sing} |^2 \bigg)^{1/2} 
			\le C\|F\|_{L^\infty}.
		\end{equation}
		On the other hand, let $Q_1^*(\xi) = (\xi,10) + (-1,1)^3$. By \eqref{sing2.proof.prop.BL1j},
		\begin{equation}\label{sing4.proof.prop.BL1j}
			\bigg| \dashint_{{Q_1^*(\xi)}} q_{\rm sing} \bigg| \le C\|F\|_{L^\infty}.
		\end{equation}
		Since ${Q_1^*(\xi)} \subset \Omega_{20}(\xi)$, by a familiar argument and \eqref{sing3.proof.prop.BL1j}, we have
		\begin{equation}\label{sing5.proof.prop.BL1j}
			\begin{aligned}
				\bigg| \dashint_{Q_1^*({\xi})} q_{\rm sing} - \dashint_{\Omega_{20}(\xi)} q_{\rm sing} \bigg| 
				& \le \dashint_{Q_1^*({\xi})} |q_{\rm sing} - \dashint_{\Omega_{20}(\xi)} q_{\rm sing}| \\
				& \le C\dashint_{\Omega_{20}(\xi)} |q_{\rm sing} - \dashint_{\Omega_{20}(\xi)} q_{\rm sing}| \le C\|F\|_{L^\infty}.
			\end{aligned}
		\end{equation}
		This, together with \eqref{sing4.proof.prop.BL1j}, implies
		\begin{equation}\label{sing6.proof.prop.BL1j}
			\bigg| \dashint_{\Omega_{20}(\xi)} q_{\rm sing} \bigg| \le C\|F\|_{L^\infty}.
		\end{equation}
		Now, combining \eqref{sing3.proof.prop.BL1j} and \eqref{sing6.proof.prop.BL1j}, we obtain
		\begin{equation}\label{sing7.proof.prop.BL1j}
			\begin{aligned}
				\bigg( \dashint_{\Omega\cap Q_8(\xi)} |q_{\rm sing}|^2 \bigg)^{1/2} 
				& \le \bigg( \dashint_{\Omega\cap Q_8(\xi)} |q_{\rm sing} - \dashint_{\Omega_{20}(\xi)} q_{\rm sing}|^2 \bigg)^{1/2} + \bigg| \dashint_{\Omega_{20}(\xi)} q_{\rm sing} \bigg|  \\
				& \le C\bigg( \dashint_{\Omega_{20}(\xi)} |q_{\rm sing} - \dashint_{\Omega_{20}(\xi)} q_{\rm sing}|^2 \bigg)^{1/2} + \bigg| \dashint_{\Omega_{20}(\xi)} q_{\rm sing} \bigg| \\
				&  \le C\|F\|_{L^\infty}
			\end{aligned}
		\end{equation}
		with $C$ independent of $\xi$.
		
		Now, combining (\ref{est.wreg.Q8}), (\ref{est.qreg.Q8}), (\ref{sing1.proof.prop.BL1j}) and (\ref{sing7.proof.prop.BL1j}), we have
		\begin{equation}\label{est3.proof.prop.BL1j}
			\begin{aligned}
				\sup_{\xi\in\Z^2} \int_{\Omega \cap Q_8(\xi)} 
				\big(|\nabla w|^2+|q|^2\big) 
				& \le C\|F\|_{L^\infty}^2.
			\end{aligned}
		\end{equation}
		Finally, the desired estimate \eqref{est1.proof.prop.BL1j} is a consequence of \eqref{est2.proof.prop.BL1j} and \eqref{est3.proof.prop.BL1j}. This concludes the proof of Theorem \ref{prop.BL1j}.
\end{proof}

	\subsection{Second-order periodic boundary layers}\label{subsec.2nd.BL} 

	Let $P^{(2j)}$ be a no-slip Stokes polynomial of degree $2$. We consider the second-order boundary layer equations
	\begin{equation}\tag{BL$_{{\rm 2nd}}^{(j)}$}\label{BL2j}
		\left\{
		\begin{array}{ll}
			-\Delta v+\nabla q=0, &x\in\Omega \\
			\nabla\cdot v=0, &x\in\Omega \\
			v+P^{(2j)}=0, &x\in\partial\Omega.
		\end{array}
		\right.
	\end{equation}
	Constructing solutions to \eqref{BL2j} for $j\in\{1,3,5,6\}$ with subquadratic growth is much more involved than constructing solutions to \eqref{BL1j} with sublinear growth. Indeed, for $j\in\{1,3,5,6\}$, the boundary data $-P^{(2j)}$ in \eqref{BL2j} grows linearly in the tangential direction. Solutions to \eqref{BL2j} for $j\in\{1,3,5,6\}$ are constructed using the first-order correctors solving \eqref{BL1j}, see below. For this construction we rely on convergence/decay properties of the first-order correctors away from the boundary. Hence, we analyze \eqref{BL2j} under periodicity assumptions. Periodicity ensures exponential convergence/decay away from the boundary.

	Throughout this subsection, we assume $\Omega$ is a periodic bumpy John domain according to Definition \ref{def.PeriodicJohn}. Consider the fundamental periodic domain
	\begin{equation*}
		\Omega_{{\rm p}}=\Omega\cap(-\pi,\pi]^2\times(-1,\infty).
	\end{equation*}
	We regard $\Omega_{{\rm p}}$ as a submanifold of $\mathbb{T}^2\times\R$, where $\mathbb{T}=\R/2\pi\Z$ is the flat torus. By definition, $\Omega_{\rm p}$ is open and connected in $\mathbb{T}^2\times\R$. Moreover, $\Omega_{{\rm p}}\cap \{x_3<2\}$ is diffeomorphic to a bounded John domain in $\R^3$. We thus have the Bogovskii operator on $\Omega_{{\rm p}}\cap \{x_3<2\}$. It is important to notice that there is a one-to-one correspondence between the functions in $\Omega_{\rm p}$ and the $(2\pi \Z)^2$-periodic functions in $\Omega$. We say a function $F$ defined in $\Omega$ is $(2\pi \Z)^2$-periodic if $F(x) = F(x + z)$ for any $x\in \Omega$ and $z\in (2\pi \Z)^2\times \{0\}$. In other words, if $f\in L^2_{\rm loc}(\Omega_{\rm p})$, then there exists a locally $L^2$ function $F$ defined in $\Omega$ so that
	$F(x) = f(\tilde{x})$, where $\tilde{x}$ is the representation in $\Omega_{\rm p}$ so that $x-\tilde{x} \in (2\pi \Z)^2\times \{ 0 \}$. In this sense and for convenience, we do not distinguish between $F$ and $f$.

	Denote by $L^2(\Omega_{\rm p})$ and $\widehat{H}^1_0(\Omega_{\rm p})$ the closure 
	of $C_0^\infty(\Omega_{\rm p})$ under the norms
	\begin{equation*}
		\| f \|_{L^2(\Omega_{\rm p})} := \bigg( \int_{\Omega_{\rm p}} |f|^2 \bigg)^{1/2}, \qquad  \| f \|_{\widehat{H}^1(\Omega_{\rm p})} := \bigg( \int_{\Omega_{\rm p}} |\nabla f|^2 \bigg)^{1/2}.
	\end{equation*}
	Clearly, $\widehat{H}^1_0(\Omega_{\rm p})$ is a Hilbert space with respect to the inner product $\langle\nabla f,\nabla g\rangle_{\Omega_{{\rm p}}}$.
	Here and below, 
	$$
	\langle f,g\rangle_{\Omega_{{\rm p}}}:=\int_{\Omega_{{\rm p}}}  f\cdot g.
	$$
	Let $\widehat{H}^1_{0,\sigma}(\Omega_{{\rm p}})$ be the subspace of ${\widehat{H}^1_0(\Omega_{\rm p})^3}$ that consists of all the divergence-free functions, namely, $\widehat{H}^1_{0,\sigma}(\Omega_{{\rm p}}) = \{ f\in {\widehat{H}^1_{0}(\Omega_{{\rm p}})^3} ~|~ \nabla\cdot f = 0 \}$.

	We now recall the Fourier series representation for the solutions of \eqref{BL1j} on the flat half-space $\{x_3>0\}$. The same formulas are obtained in \cite[Proposition 3]{HP} based on the periodic Poisson kernel.
	Note that \cite{HP} uses the fact that the equations are imposed on a domain whose boundary is given by the graph, but a similar proof is valid if we utilize the zero extension of the functions.
	%
	\begin{proposition}\label{prop.per.BL1j}
		Let $L\in(0,\infty)$ and $\Omega$ be a periodic bumpy John domain with constant $L$ according to Definition \ref{def.PeriodicJohn}. Then the weak solution $(v^{(1j)},q^{(1j)})$ of \eqref{BL1j} given by Theorem \ref{prop.BL1j} satisfies the following. 
		\begin{enumerate}[label=(\roman*)]
			\item $(v^{(1j)},q^{(1j)})$ is expanded in Fourier series in $\{x_3>0\}$ as
			\begin{equation}\label{rep1.prop.per.BL1j}
				\begin{split}
					v^{(1j)}(x) 
					& = \hat{v}^{(1j)}_{(0,0)}
					+ \sum_{k\in\Z^2\setminus\{(0,0)\}} 
					\Big(\hat{v}^{(1j)}_k 
					+ \begin{pmatrix} -ik \\ |k| \end{pmatrix}
					V^{(1j)}(k)
					x_3\Big)
					e^{-|k|x_3} e^{ik\cdot x'}, \\
					q^{(1j)}(x)
					& = 
					\sum_{k\in\Z^2\setminus\{(0,0)\}} 
					2|k| V^{(1j)}(k) e^{-|k|x_3} e^{ik\cdot x'}, 
				\end{split}
			\end{equation}
			where $V^{(1j)}(k)$ is a scalar function of $k$ defined by
			\begin{equation}\label{rep2.prop.per.BL1j}
				\begin{split}
					V^{(1j)}(k)
					= 
					\hat{v}^{(1j)}_{k,3} - i \frac{k}{|k|}\cdot (\hat{v}^{(1j)}_k)',
				\end{split}
			\end{equation}
			and moreover, $\hat{v}^{(1j)}_k$ is the Fourier series coefficient of $v^{(1j)}(x',0)${\rm :}
			\begin{align}\label{e.deffouriercoef}
				\hat{v}^{(1j)}_k 
				= \frac{1}{(2\pi)^2} \int_{(-\pi,\pi)^2} v^{(1j)}(x',0) e^{-ik\cdot x'} \dd x', \quad k\in\Z^2.
			\end{align}
			\item The third component of $\hat{v}^{(1j)}_{(0,0)}$ is zero. Particularly, by setting 
			\begin{align*}
				\hat{v}^{(1j)}_{(0,0)}
				=:\alpha^{(1j)}=(\alpha^{(1j)}_1, \alpha^{(1j)}_2, 0),
			\end{align*}
			we have the exponential convergence
			\begin{equation}\label{conv.prop.per.BL1j}
				\begin{split}
					&|v^{(1j)}(x)-\alpha^{(1j)}|+|\nabla v^{(1j)}(x)|+ |q(x)| \\
					& \qquad \le C\|v^{(1j)}(\cdot,0)\|_{L^2((-\pi,\pi)^2)}\,e^{-x_3/2} \quad \text{for } x_3>1.
				\end{split}
			\end{equation}
			Here $C$ is a universal constant.
		\end{enumerate}
	\end{proposition}
	%

	\subsubsection*{\underline{Construction of $v^{(2j)}$ for $j\in\{1,3,5,6\}$}}
	%
	We construct the second-order boundary layers $v^{(2j)}$ corresponding to $P^{(2j)}=P^{(2j)}(x)$ for $j\in\{1,3,5,6\}$. These boundary layers are solutions to \eqref{BL2j} with subquadratic growth; see Theorem \ref{prop.BL21}. We begin with the case $j=1$, where $P^{(21)}(x)=x_2x_3{\bf e}_3$. We recall that $(v^{(21)},q^{(21)})$ solves
	\begin{equation}\tag{BL$_{{\rm 2nd}}^{(1)}$}\label{BL21}
		\left\{
		\begin{array}{ll}
			-\Delta v+\nabla q=0 &\text{in }\Omega \\
			\nabla\cdot v=0 &\text{in }\Omega \\
			v+x_2x_3{\bf e}_1=0 &\text{on }\partial\Omega.
		\end{array}
		\right.
	\end{equation}
	The difficulty in the analysis of \eqref{BL21} is that the boundary value is not periodic and has linear growth as $x_2 \to \infty$. We aim at eliminating the growth factor $x_2$ and at recovering the periodic structure. The key finding is the connection between the first-order and second-order boundary layers on the boundary, namely
	\begin{equation}
		v^{(21)}-x_2 v^{(11)}=0, \quad x\in\partial\Omega.
	\end{equation}
	This observation is the basis of the Ansatz for $v^{(21)}$. 
	Recall that $v^{(11)}$ converges exponentially fast to the constant $\alpha^{(11)}\in\R^3$, when $x_3\rightarrow\infty$ by the spectral gap near frequency $0$ yielded by the periodicity; see \eqref{conv.prop.per.BL1j}.  
	Hence the non-decaying divergence $\nabla\cdot(x_2 v^{(11)}(x))=v^{(11)}_2(x)$ can be corrected by adding a corrector $- \alpha^{(11)}_2 x_3\eta_+(x_3) {\bf e}_3$. 
	Here $\eta_+(\cdot)$ is a function on $\R$ satisfying
	\begin{equation}\label{def.eta2}
		\begin{aligned}
			&\text{$\eta_+(t)$ is smooth and non-negative,} \\
			&\text{$\eta_+(t)=0$ if $t<\frac12$ and $\eta_+(t)=1$ if $t>1$}.
		\end{aligned}
	\end{equation}
	Below, we also need the cut-off $\eta_-$ defined in \eqref{def.eta1}.
	
	The following statement gives the existence and the structure of second-order boundary layers with subquadratic growth. 
	\begin{theorem}\label{prop.BL21}
		Let $L\in(0,\infty)$ and $\Omega$ be a periodic bumpy John domain with constant $L$ according to Definition \ref{def.PeriodicJohn}. There exists a weak solution $(v^{(21)},q^{(21)})\in H^1_{{\rm loc}}(\overline{\Omega})^3\times L^2_{{\rm loc}}(\overline{\Omega})$ to \eqref{BL21} decomposed as
		\begin{equation}\label{est1.prop.BL21}
			\begin{aligned}
				v^{(21)}(x) &= x_2 v^{(11)}(x) - \alpha^{(11)}_2 x_3\eta_+(x_3) {\bf e}_3 + R^{(21)}(x), \\
				q^{(21)}(x) &= x_2 q^{(11)}(x) + Q^{(21)}(x),
			\end{aligned}
		\end{equation}
		where $(R^{(21)}, Q^{(21)}) \in \widehat{H}_0^1(\Omega_{\rm p})^3 \times L^2(\Omega_{\rm p})$. Moreover, we have
		\begin{equation}\label{est2.prop.BL21}
			\|\nabla R^{(21)}\|_{L^2(\Omega_{{\rm p}})}
			+ \|Q^{(21)}\|_{L^2(\Omega_{{\rm p}})}
			\le C,
		\end{equation}
		where the constant $C$ depends only on $L$.
	\end{theorem}
	%
	\begin{proof}
		We aim at proving the existence of $(R^{(21)},Q^{(21)})$ and estimating it so that $(v^{(21)}, q^{(21)})$ defined by the right-hand sides in  \eqref{est1.prop.BL21} gives a weak solution of \eqref{BL21}. 
		
		\noindent (\textbf{Existence}) 
		By the previous discussion, we begin with a formal examination of $x_2 v^{(11)}(x) - \alpha^{(11)}_2 x_3\eta_+(x_3) {\bf e}_3$. First of all, it is easy to see
		\begin{equation*}
			\begin{aligned}
				&\nabla\cdot \big(x_2 v^{(11)}(x) - \alpha^{(11)}_2 x_3\eta_+(x_3) {\bf e}_3\big) \\
				&\quad = v^{(11)}_2(x)-\alpha^{(11)}_2\eta_+(x_3)-\alpha^{(11)}_2x_3\partial_3\eta_+(x_3)
			\end{aligned}
		\end{equation*}
		for all $x\in\Omega_{\rm p}$. Notice that the expression above simplifies for $x_3>1$: 
		\begin{equation*}
			\nabla\cdot \big(x_2 v^{(11)}(x) - \alpha^{(11)}_2 x_3\eta_+(x_3) {\bf e}_3\big)=v^{(11)}_2(x)-\alpha^{(11)}_2.
		\end{equation*}
		Then, by  Proposition \ref{prop.per.BL1j}, we get
		\begin{equation*}
			\begin{aligned}
				v^{(11)}_2(x)-\alpha^{(11)}_2
				= 
				\sum_{k\in\Z^2\setminus\{(0,0)\}}
				\big(\hat{v}^{(11)}_{k,2}-ik_2 V^{(11)}(k)x_3\big)
				e^{-|k|x_3} e^{ik\cdot x'},\quad x_3>1.
			\end{aligned}
		\end{equation*}
		This means that $x_2 v^{(11)}(x) - \alpha^{(11)}_2 x_3\eta_+(x_3) {\bf e}_3$ is not divergence-free. Thus, our next goal is to construct a function to correct the divergence for $x_3>1$. 
		Define
		\begin{equation}
			\begin{split}
				d(x) 
				&:= 
				\sum_{k\in\Z^2,\, k_1\neq0} 
				\Big(\frac{1}{ik_1}\big(\hat{v}^{(11)}_{k,2}-ik_2 V^{(11)}(k)x_3\big)
				e^{-|k|x_3} e^{ik\cdot x'}\Big)
				{\bf e}_1 \\
				&\quad
				+ \sum_{k_1 = 0,\,  k_2\in\Z\setminus\{0\}} 
				\Big(\frac{1}{ik_2}\big(\hat{v}^{(11)}_{k,2}-ik_2 V^{(11)}(k)x_3\big)
				e^{-|k|x_3} e^{ik\cdot x'}\Big)
				{\bf e}_2,
			\end{split}
		\end{equation}
		which is an element of 
		$H^1(\Omega_{{\rm p},>0})^3$ where $\Omega_{{\rm p},>0}:=\Omega_{{\rm p}}\cap \{x_3>0\}$. Of course, there is no unique way to construct a right-inverse of the divergence such as $d$. 
		We may extend $d(x)$ to the whole domain $\Omega_{\rm p}$ by multiplying it by $\eta_+(x_3)$ and still correct the divergence of $x_2 v^{(11)}(x) - \alpha^{(11)}_2 x_3\eta_+(x_3) {\bf e}_3 - d(x)\eta_+(x_3)$.
		To check the divergence condition, we calculate
		\begin{equation*}
			\begin{aligned}
				D(x)
				&:=
				\nabla\cdot \big(x_2 v^{(11)}(x) - \alpha^{(11)}_2 x_3\eta_+(x_3) {\bf e}_3-d(x)\eta_+(x_3)\big) \\
				&=
				v^{(11)}_2(x)-\alpha^{(11)}_2\eta_+(x_3)-\eta_+(x_3)\nabla\cdot d(x) -(\alpha^{(11)}_2x_3 + d_3(x)) \partial_3\eta_+(x_3).
			\end{aligned}
		\end{equation*}
		Obviously, $D$ is supported in $
		\Omega_{{\rm p},<2}:=\Omega_{{\rm p}}\cap \{x_3<2\}$, in which we can rely on the Bogovskii operator to find a right-inverse of the divergence. 
		Let $A:= \int_{\Omega_{{\rm p},<2}} D$. Let $\chi_+(x_3)$ be a smooth cut-off function such that 
		\begin{equation}\label{def.chi+}
			\chi_+(x_3)=0\quad\mbox{if}\quad  x_3\le 0,\quad \mbox{and}\quad \chi_+(x_3) = (2\pi)^{-2} A\quad \mbox{for}\quad x_3 > 1.
		\end{equation}
		This implies $\partial_3 \chi_+(x_3)$ is supported in $\Omega_{{\rm p},<2}$ and $\int_{\Omega_{{\rm p},<2}} \partial_3 \chi_+ = A$. It follows that
		\begin{equation*}
			\int_{\Omega_{{\rm p},<2}}\big(D(x) - \nabla\cdot\big(\chi_+(x_3){\bf e}_3\big)\big)\, dx
			= \int_{\Omega_{{\rm p},<2}}\big(D(x) - \partial_3\chi_+(x_3)\big)\, dx=0.
		\end{equation*}
		%
		Hence, by Appendix \ref{appendix.Bogovskii},  there is a Bogovskii corrector $\B \in H_0^1(\Omega_{{\rm p},<2})^3$ such that
		\begin{equation*}
			\nabla\cdot\B(x) = D(x) - \partial_3\chi_+(x_3), 
		\end{equation*}
		and $\| \B \|_{H^1(\Omega_{{\rm p},<2})} \le C$, where $C$ depends only on the John constant $L$ of $\Omega$.
		We extend $\B$ by zero to the whole domain $\Omega_{\rm p}$ and denote it again by 
		$\B\in {H^{1}_{0}(\Omega_{{\rm p}})^3}$. 
		Let us combine the above correctors and define
		\begin{equation*}
			\mathcal{C}(x)
			=-d(x)\eta_+(x_3)-\chi_+(x_3){\bf e}_3 - \B(x).
		\end{equation*}
		Note that $\mathcal{C} \in {\widehat{H}_0^1(\Omega_{\rm p})^3}$. In particular, $\| \nabla \mathcal{C}\|_{L^2(\Omega_{{\rm p}})} \le C$, where $C$ depends only on the John constant $L$ of $\Omega$. By \eqref{conv.prop.per.BL1j}, the function $\mathcal C$ converges exponentially fast to $-(2\pi)^{-2} A$ as $x_3 \to \infty$, and its derivatives decay exponentially fast to $0$ as $x_3\rightarrow\infty$.

		By the crucial cancellation
		\begin{equation*}
			x_2 (-\Delta v^{(11)}(x) +  \nabla q^{(11)}(x))=0, 
		\end{equation*}
		as well as the definition of $\mathcal{C}$, we see that the pair 
		\begin{equation*}
			x_2 v^{(11)}(x) - \alpha^{(11)}_2 x_3\eta_+(x_3) {\bf e}_3+\mathcal{C}(x)
			\quad \text{and} \quad
			x_2q^{(11)}(x)
		\end{equation*}
		is a weak solution to \eqref{BL21} with an additional external force
		\begin{equation*}
			f^{(21)}(x) = -2\partial_2v^{(11)}(x)+q^{(11)}(x){\bf e}_2-\Delta\big(-\alpha^{(11)}_2 x_3\eta_+(x_3) {\bf e}_3+\mathcal{C}(x)\big).
		\end{equation*}
		In order to cancel this source term, we consider
		\begin{equation}\label{eq1.proof.prop.BL21}
			\left\{
			\begin{array}{ll}
				-\Delta \mathcal{R}^{(21)} + \nabla \mathcal{Q}^{(21)}=-f^{(21)} &\text{in } \Omega_{\rm p} \\
				\nabla\cdot \mathcal{R}^{(21)}=0 &\text{in } \Omega_{\rm p} \\
				\mathcal{R}^{(21)}=0 &\text{on } \partial\Omega_{\rm p}.
			\end{array}
			\right.
		\end{equation}
		The weak formulation of \eqref{eq1.proof.prop.BL21} is written as
		\begin{equation}\label{eq2.proof.prop.BL21}
			\langle\nabla\mathcal{R}^{(21)},\nabla\varphi\rangle_{\Omega_{{\rm p}}}
			=-\langle f^{(21)}, \varphi\rangle_{\Omega_{{\rm p}}},
			\quad \varphi\in \widehat{H}^{1}_{0,\sigma}(\Omega_{{\rm p}}).
		\end{equation}
		Next, we prove the unique existence of the weak solution of \eqref{eq2.proof.prop.BL21}. By the integration by parts for $\Delta\mathcal{C}$,
		\begin{equation}\label{est1.proof.prop.BL21}
			\begin{aligned}
				\langle f^{(21)}, \varphi\rangle_{\Omega_{{\rm p}}}
				&=-2\langle \partial_2v^{(11)}, \varphi\rangle_{\Omega_{{\rm p}}}
				+ \langle q^{(11)}, \varphi_2\rangle_{\Omega_{{\rm p}}} \\
				&\quad 
				+ \alpha^{(11)}_2 \langle \Delta(x_3\eta_+(x_3) {\bf e}_3), \varphi\big\rangle_{\Omega_{{\rm p}}}
				+ \langle\nabla\mathcal{C}, \nabla\varphi\rangle_{\Omega_{{\rm p}}}.
			\end{aligned}
		\end{equation}
		By the Poincar\'{e} inequality in $\Omega_{{\rm p},<2}$ and the Cauchy-Schwarz inequality in $\Omega_{\rm p}$,
		\begin{align}\label{est2.proof.prop.BL21}
			\begin{split}
				&|\langle f^{(21)}, \varphi\rangle_{\Omega_{{\rm p}}} | \\
				&\le 
				C\Big(\|\nabla v^{(11)}\|_{L^2(\Omega_{{\rm p}})}
				+ \|q^{(11)}\|_{L^2(\Omega_{{\rm p}})}+ \|\Delta(x_3\eta_+(x_3) \|_{L^2(\Omega_{\rm p})} + \|\nabla \mathcal{C} \|_{L^2(\Omega_{\rm p})} \Big)
				\|\nabla \varphi\|_{L^2(\Omega_{{\rm p}})} \\
				&\qquad
				+ \bigg|\int_{1}^{\infty} \int_{(-\pi,\pi)^2} \partial_2v^{(11)}(x) \varphi(x) \,\dd x' \dd x_3\bigg|  \\
				&\qquad
				+ \bigg|\int_{1}^{\infty} \int_{(-\pi,\pi)^2} q^{(11)}(x) \varphi_2(x) \,\dd x' \dd x_3\bigg|.
			\end{split}
		\end{align}
		From Proposition \ref{prop.per.BL1j} again, we have the following representation formulas:
		\begin{equation*}
			\begin{aligned}
				\partial_2v^{(11)}(y) 
				&= \partial_2\bigg(
				\sum_{k\in\Z^2\setminus\{(0,0)\}}
				\Big(\hat{v}^{(1j)}_k 
				+ \begin{pmatrix} -ik \\ |k| \end{pmatrix}
				V^{(1j)}(k)
				x_3\Big)
				e^{-|k|x_3} e^{ik\cdot x'}\bigg), \\
				q^{(11)}(y) 
				&= \partial_1\bigg(
				\sum_{k\in\Z^2,\, k_1\neq0}
				2|k| V^{(1j)}(k) e^{-|k|x_3}
				\frac{e^{ik\cdot x'}}{ik_1}\bigg) \\
				&\quad + \partial_2\bigg(
				\sum_{k_1 = 0,\,  k_2\in\Z\setminus\{0\}} 
				2|k| V^{(1j)}(k) e^{-|k|x_3}
				\frac{e^{ik\cdot x'}}{ik_2}\bigg).
			\end{aligned}
		\end{equation*}
		Thus, by integration by parts in $x_1$ and $x_2$, the last two integrals in \eqref{est2.proof.prop.BL21} are bounded by $C\|\nabla \varphi\|_{L^2(\Omega_{{\rm p}})}$. Consequently, in view of \eqref{est1.prop.BL1j} and \eqref{est2.proof.prop.BL21}, we obtain
		\begin{equation}\label{est3.proof.prop.BL21}
			\begin{aligned}
				|\langle f^{(21)}, \varphi\rangle_{\Omega_{{\rm p}}} |
				\le 
				C\|\nabla \varphi\|_{L^2(\Omega_{{\rm p}})}.
			\end{aligned}
		\end{equation}
		Then, by the Riesz representation theorem, there is an element $\mathcal{R}^{(21)}\in \widehat{H}^{1}_{0,\sigma}(\Omega_{{\rm p}})$ solving \eqref{eq2.proof.prop.BL21} and satisfying $\| \nabla \mathcal{R}^{(21)} \|_{L^2(\Omega_{{\rm p}})} \le C$. The existence of the pressure $\mathcal{Q}^{(21)}\in L^2_{{\rm loc}}(\overline{\Omega}_{\rm p})$ can be proved by using the Bogovskii lemma. Finally, the existence of the remainder $(R^{(21)},Q^{(21)})$ in \eqref{est1.prop.BL21} is proved if we set $R^{(21)}=\mathcal{C}^{(21)}+\mathcal{R}^{(21)}$ and $Q^{(21)}=\mathcal{Q}^{(21)}$. Let us emphasize that $(R^{(21)},Q^{(21)})$ can be ``unwrapped'' (without changing the notation) as a $(2\pi\Z)^2$-periodic function in $\Omega$.
		
		\noindent (\textbf{Estimate}) The estimate of $R^{(21)}$ is clear. Hence we focus on the pressure term $Q^{(21)}=\mathcal{Q}^{(21)}$. Notice that since $\mathcal{Q}^{(21)}$ is locally $L^2$, it suffices to prove
		\begin{equation}\label{est3.prop.BL21}
			\|\mathcal{Q}^{(21)}\|_{L^2(\Omega_{{\rm p},>3})}
			\le C.
		\end{equation}
		We apply the Fourier series expansion in the flat domain $\{x_3>3\}$. We decompose $\mathcal{R}^{(21)}$ and $\mathcal{Q}^{(21)}$ into $\mathcal{R}^{(21)}=w_1+w_2$ and $\mathcal{Q}^{(21)}=r_1+r_2$ (up to a constant), where $(w_1,r_1)$ is a solution of
		\begin{equation*}
			\left\{
			\begin{array}{ll}
				-\Delta w_1 + \nabla r_1=0, &x_3>3 \\
				\nabla\cdot w_1=0, &x_3>3 \\
				w_1(x',3)=\mathcal{R}^{(21)}(x',3),
			\end{array}
			\right.
		\end{equation*}
		while $(w_2,r_2)$ solves
		\begin{equation*}
			\left\{
			\begin{array}{ll}
				-\Delta w_2 + \nabla r_2=-f^{(21)}, &x_3>3 \\
				\nabla\cdot w_2=0, &x_3>3 \\
				w_2(x',3)=0.
			\end{array}
			\right.
		\end{equation*}
		Using the periodicity of $\mathcal{R}^{(21)}(x',3)$ in $x'$, the solution $(w_1,r_1)$ may be written by the Poisson kernel as in Proposition \ref{prop.per.BL1j}, which implies
		\begin{equation}\label{est4.prop.BL21}
			\|r_1\|_{L^2(\Omega_{{\rm p},>3})}
			\le C\|\nabla \mathcal{R}^{(21)}\|_{L^2(\Omega_{{\rm p}})}
			\le C.
		\end{equation}
		On the other hand, observe that the source term $-f^{(21)}$ is represented as
		\begin{equation*}
			-f^{(21)}(x)
			= \sum_{k\in\Z^2\setminus\{(0,0)\}} 
			(\mathcal{F}_1(k) + \mathcal{F}_2(k) x_3) 
			e^{-|k|x_3} e^{ik\cdot x'},\quad x_3>3,
		\end{equation*}
		where 
		\begin{equation*}
			|\mathcal{F}_1(k)| + |\mathcal{F}_2(k)|
			\le C|k|^2 |\hat{v}^{(11)}_k|
		\end{equation*}
		with $\hat{v}^{(1j)}_k$ defined in \eqref{e.deffouriercoef}. 
		Then a simple computation shows that
		\begin{equation*}
			\begin{aligned}
				w_2(x) 
				&= \sum_{k\in\Z^2\setminus\{(0,0)\}}
				(\mathcal{G}_1(k) + \mathcal{G}_2(k) x_3
				+\mathcal{G}_3(k)x_3^2 +\mathcal{G}_4(k) x_3^3) 
				e^{-|k|x_3} e^{ik\cdot x'}, \\
				r_2(x) 
				&= \sum_{k\in\Z^2\setminus\{(0,0)\}}
				(\mathcal{G}_5(k) + \mathcal{G}_6(k) x_3
				+\mathcal{G}_7(k)x_3^2 +\mathcal{G}_8(k) x_3^3) 
				e^{-|k|x_3} e^{ik\cdot x'},
			\end{aligned}
		\end{equation*}
		where
		\begin{equation*}
			|k| \sum_{l=1}^4 |\mathcal{G}_l(k)| 
			+ \sum_{l=5}^8 |\mathcal{G}_l(k)|
			\le C|k|^3 |\hat{v}^{(11)}_k|.
		\end{equation*}
		Now it is easy to see that 
		\begin{equation}\label{est5.prop.BL21}
			\|r_2\|_{L^2(\Omega_{{\rm p},>3})}
			\le C\|v^{(11)}(\cdot,0)\|_{L^2((-\pi,\pi)^2)}
			\le C.
		\end{equation}
		From \eqref{est4.prop.BL21} and \eqref{est5.prop.BL21}, we obtain \eqref{est3.prop.BL21}. This completes the proof of Theorem \ref{prop.BL21}.
	\end{proof}
	%

	By a similar consideration, we can obtain the existence of $(v^{(2j)},q^{(2j)})$ for $j\in \{3,5,6\}$, whose proofs are parallel to Theorem \ref{prop.BL21} and therefore omitted. 
	Recall that $\eta_+$ (resp. $\eta_-$) is defined in \eqref{def.eta2} (resp. \eqref{def.eta1}).
	%
	\begin{theorem}\label{prop.BL2356}
		Let $L\in(0,\infty)$ and $\Omega$ be a periodic bumpy John domain with constant $L$ according to Definition \ref{def.PeriodicJohn}. Let $j\in\{3,5,6\}$. There exists a weak solution $(v^{(2j)},q^{(2j)})\in H^1_{{\rm loc}}(\overline{\Omega})^3\times L^2_{{\rm loc}}(\overline{\Omega})$ to \eqref{BL2j} decomposed as, when $j=3$,
		\begin{equation}\label{est1.prop.BL2356}
			\begin{aligned}
				v^{(23)}(x) &= x_1 v^{(12)}(x) - \alpha^{(12)}_1 x_3\eta_+(x_3) {\bf e}_3 + R^{(23)}(x), \\
				q^{(23)}(x) &= x_1 q^{(12)}(x) + Q^{(23)}(x),
			\end{aligned}
		\end{equation}
		when $j=5$,
		\begin{equation}\label{est2.prop.BL2356}
			\begin{aligned}
				v^{(25)}(x) &= -2x_1 v^{(11)}(x) - x_3^2\eta_-(x_3) {\bf e}_3 + 2\alpha^{(11)}_1 x_3 \eta_+(x_3) {\bf e}_3 + R^{(25)}(x), \\
				q^{(25)}(x) &= -2x_1 q^{(11)}(x) + Q^{(25)}(x),
			\end{aligned}
		\end{equation}
		and when $j=6$,
		\begin{equation}\label{est3.prop.BL2356}
			\begin{aligned}
				v^{(26)}(x) &= -2x_2 v^{(12)}(x) 
				- x_3^2\eta_-(x_3) {\bf e}_3 + 2\alpha^{(12)}_2 x_3 \eta_+(x_3) {\bf e}_3 + R^{(26)}(x), \\
				q^{(26)}(x) &= -2x_2 q^{(12)}(x) + Q^{(26)}(x),
			\end{aligned}
		\end{equation}
		where $(R^{(2j)}, Q^{(2j)}) \in 
		\widehat{H}_0^1(\Omega_{\rm p})^3 \times L^2(\Omega_{\rm p})$. Moreover, we have
		\begin{equation}\label{est4.prop.BL2356}
			\|\nabla R^{(2j)}\|_{L^2(\Omega_{{\rm p}})}
			+ \|Q^{(2j)}\|_{L^2(\Omega_{{\rm p}})}
			\le C,
		\end{equation}
		where the constant $C$ depends only on $L$.
	\end{theorem}

	\subsubsection*{\underline{Construction of $v^{(22)}$ and $v^{(24)}$}}

	The boundary layers corresponding to $P^{(22)}$ and $P^{(24)}$ can be constructed by using the Green function. The fact that $P^{(22)}$ and $P^{(24)}$ only depend on the vertical variable $x_3$ and that there is no growth in the tangential variable $x'$ makes the analysis much easier than for $P^{(2j)}$, $j\in\{1,3,5,6\}$ studied above. 
	The proof of the following proposition is almost identical to the one of Theorem \ref{prop.BL1j}. Notice that here we state Theorem \ref{prop.BL22 and 24} in the periodic case only for convenience. Indeed we use these correctors in combination with $(v^{(2j)},q^{(2j)})$ for $j\in {\{1,3,5,6\}}$ whose existence is stated in 
	Theorems \ref{prop.BL21} and \ref{prop.BL2356} in periodic bumpy John domains. However, the existence of $(v^{(2j)},q^{(2j)})$ for $j\in \{2,4\}$ can be proved in general bumpy John domains according to Definition \ref{def.John2}.
	%
	\begin{theorem}\label{prop.BL22 and 24}
		Let $L\in(0,\infty)$ and $\Omega$ be a periodic bumpy John domain with constant $L$ according to Definition \ref{def.PeriodicJohn}. Let $j\in\{2,4\}$. There exists a unique weak solution $(v^{(2j)},q^{(2j)})\in H^1_{{\rm loc}}(\overline{\Omega})^3\times L^2_{{\rm loc}}(\overline{\Omega})$ to \eqref{BL2j}
		satisfying
		\begin{align}\label{est1.prop.BL22 and 24}
			\|\nabla v^{(2j)}\|_{L^2(\Omega_{{\rm p}})}
			+ \|q^{(2j)}\|_{L^2(\Omega_{{\rm p}})}
			\le C,
		\end{align}
		where the constant $C$ depends only on $L$.
	\end{theorem}
	%

	\subsection{Estimates of boundary layers}\label{subsec.Ests.BL}
	Before closing this section, we summarize the estimates of the boundary layers. The following propositions can be proved in a similar manner as in \cite[Lemma 4]{HP} combined with a direct computation. The details are omitted here.
	
	%

	\begin{proposition}\label{lem.est.scaled.BL1j}
		Let $L\in(0,\infty)$ and $\Omega$ be a bumpy John domain with constant $L$ as in Definition \ref{def.John2}. 
		For $j\in\{1,2\}$, let $(v^{(1j)}, q^{(1j)})$ the weak solution of \eqref{BL1j} provided by Theorem \ref{prop.BL1j}. Then, for $r\in(\ep,1)$,
		\begin{align}\label{est1.lem.est.scaled.BL1j}
			\bigg(\dashint_{B^\ep_{r,+}}
			|(\nabla v^{(1j)})(x/\ep)|^{2} \dd x\bigg)^{1/2}
			+ \bigg(
			\dashint_{B^\ep_{r,+}}
			|q^{(1j)}(x/\ep)|^2 \dd x\bigg)^{1/2}
			\le 
			C \big(\frac{\ep}{r}\big)^{1/2},
		\end{align}
		where $C$ depends only on $L$.
	\end{proposition}

	\begin{proposition}\label{lem.est.scaled.BL2j}
		Let $L\in(0,\infty)$ and $\Omega$ be a periodic bumpy John domain with constant $L$ as in Definition \ref{def.PeriodicJohn}. 
		For $j\in\{1,\cdots,6\}$, let $(v^{(2j)},q^{(2j)})$ the weak solution of \eqref{BL2j} provided by Theorem \ref{prop.BL21} or \ref{prop.BL2356}. Then, for $r\in(\ep,1)$,
		\begin{align}\label{est1.lem.est.scaled.BL2j}
			\frac1{r} \bigg(\dashint_{B^\ep_{r,+}}
			|\ep (\nabla v^{(2j)})(x/\ep)|^{2} \dd x\bigg)^{1/2}
			+ \frac1r\bigg(
			\dashint_{B^\ep_{r,+}}
			|\ep q^{(2j)}(x/\ep)|^2 \dd x\bigg)^{1/2}
			\le
			C \big(\frac{\ep}{r}\big)^{1/2},
		\end{align}
		where $C$ depends only on $L$.
	\end{proposition}

	\section{Higher-order regularity}
	\label{sec.higher}

	%
	%
	\subsection{Large-scale $C^{1,\gamma}$ estimate}\label{subsec.Large-scale.1}
	The goal of this subsection is to prove the large-scale $C^{1,\gamma}$ regularity stated in Theorem \ref{thm.C1g}. We will use the first-order boundary layers and modify the argument of the Lipschitz estimate.
	%
	
	Recall that $\mathscr{P}_1 = {\rm span} \{ P^{(11)}, P^{(12)} \} = \{ (ax_3, bx_3,0)~|~a,b\in \R \}$. Let $\mathscr{P}_2 = \text{span} \{ P^{(2j)}~|~ j=1,2,\cdots, 6 \}$. Let $\mathscr{S}_2 = \text{span} \{(P^{(2j)}, L^{(2j)})~|~ j=1,2,\cdots, 6\}$. Note that any element of $\mathscr{S}_2$ is a weak solution of the Stokes system in $\R^3$. Let $(v^{(1k)}, q^{(1k)})$ with $k=1,2$, and $(v^{(2j)}, q^{(2j)})$ with $j=1,2,\cdots,6$, be the first-order and second-order boundary layers, respectively. Then define
	\begin{align}
		\mathscr{Q}_1(\Omega) & = \text{span} \{ (P^{(1k)},0) +(v^{(1k)}, q^{(1k)})~|~k=1,2  \}, \label{e.defQ1}\\
		\mathscr{Q}_2(\Omega) & = \text{span} \{ (P^{(2j)},{L^{(2j)}}) +(v^{(2j)}, q^{(2j)})~|~ j=1,2,\cdots, 6  \}\label{e.defQ2}.
	\end{align}
	Hence, $\mathscr{Q}_1(\Omega)$ (resp. $\mathscr{Q}_2(\Omega)$) is the vector space that contains all the ``linear'' (resp. ``quadratic'') solutions of the Stokes system in $\Omega$ vanishing on the boundary $\partial\Omega$; see the Liouville-type results at the end of this section stated in Theorem \ref{thm.liouville}.

	\begin{remark}
		Note that the pressure part in estimate \eqref{est.thmC1a} of Theorem \ref{thm.C1g} is different from the Lipschitz estimate in which $\bar{P}$ is $\dashint_{B_{1/2,+}^\ep} p^\ep$. Actually, in \eqref{est.thmC1a}, $\bar{P}_1$ is the average of the corrected pressure over a small ball, i.e.,
		\begin{equation*}
			\bar{P}_1 = \dashint_{B_{O(\ep),+}^\ep} \big( p^\ep - \pi(x/\ep) \big)\dd x
		\end{equation*}
		for some $(w,\pi) \in \mathscr{Q}_1(\Omega)$; see \eqref{def.Pbar}. This is reasonable since we are concerned with the $C^{0,\gamma}$ estimate of the pressure and $\bar{P}_1$ plays a role similar to 
		the zeroth-order term in the Taylor expansion of the pressure, if the boundary is flat. We emphasize that $\bar P_1$ depends on $\ep$. The point here is that $\bar{P}_1$ is independent of $r$.
	\end{remark}

	The critical fact we are going to use is that any $(w,\pi) \in \mathscr{Q}_1(\Omega)$
	is a solution of the Stokes system in $\Omega$ that vanishes on $\partial\Omega$. Hence, by rescaling, $(u^\ep, p^\ep) - (\ep w(x/\ep), \pi(x/\ep))$ is still a weak solution with a no-slip boundary condition.  This observation allows us to capture the regularity beyond the Lipschitz estimate. To this end, we define the first-order excess by 
	\begin{equation}\label{def.H1st}
		\begin{aligned}
			&H_{{\rm 1st}}(u^\ep,p^\ep;\rho) \\
			&\quad=
			\inf_{(w,\pi)\in \mathscr{Q}_1(\Omega)}
			\bigg\{	
			\bigg(\dashint_{B^\ep_{\rho,+}} \big|\nabla u^\ep - \nabla\big(\ep w(x/\ep)\big)\big|^2\dd x \bigg)^{1/2}
			 \\
			&\qquad
			+ \sup_{s,t\in [1/16,1/4]}
			\bigg| \dashint_{B^\ep_{s\rho,+}} \big(p^\ep - \pi(x/\ep) \big)\dd x - \dashint_{B^\ep_{t\rho,+}} \big( p^\ep - \pi(x/\ep) \big) \dd x \bigg| \bigg\}.
		\end{aligned}
	\end{equation}
	Recall that $(w,\pi) \in \mathscr{Q}_1(\Omega)$ means that
	\begin{equation*}
		(w, \pi) = \sum_{k=1}^2 \ell_k (P^{(1k)} + v^{(1k)}, q^{(1k)}),
	\end{equation*}
	for some $\ell_1, \ell_2 \in \R$.
	We will also use the quantity $\Phi(u^\ep,p^\ep;\rho)$ defined in \eqref{def.Phi}.

	\begin{lemma}\label{lem.H1st.excess.decay}
		Let $L\in(0,\infty)$ and $\Omega$ be a bumpy John domain with constant $L$ according to Definition \ref{def.John2}. 
		Let $(u^\ep,p^\ep)$ be as in Theorem \ref{thm.C1g}, namely, a weak solution of \eqref{S.ep} in Subsection \ref{subsec.set-up and approx.} satisfying \eqref{e.defMbis}.
		For all $\ep\in(0,{\frac1{32}}]$, $r\in {[2\ep,\frac1{16}}]$ and $\theta\in(0,{\frac18}]$, 
		\begin{equation}\label{est1.lem.H1st.excess.decay}
			\begin{aligned}
				H_{{\rm 1st}}(u^\ep,p^\ep;\theta r)
				&\le C\big( \theta + \theta^{-3}\big(\frac{\ep}{r}\big)^{1/12} \big) \Phi(u^\ep,p^\ep; 16r) \\
				&\qquad
				+ C\theta^{-3}\bigg(\dashint_{{Q_{10r}}}
				|\mathcal{M}^2_{\ep}[F^\ep]|^3\bigg)^{1/3},
			\end{aligned}
		\end{equation}
		where $C$ depends only on $L$.
	\end{lemma}
	%
	\begin{proof}
		First, we apply Lemma \ref{lem.excess.decay} with $\alpha=1$
		\begin{equation}\label{est.H.C11}
			\begin{aligned}
				H(u^\ep,p^\ep;\theta r)
				&\le C\big(\theta + \theta^{-3}\big(\frac{\ep}{r}\big)^{1/12} \big) \Phi(u^\ep,p^\ep;16r) \\
				&\qquad
				+ C\theta^{-3}\bigg(\dashint_{{Q_{10r}}} |\mathcal{M}^2_{\ep}[F^\ep]|^3\bigg)^{1/3},
			\end{aligned}
		\end{equation}
		where we also used the fact 
		$H(\cdot,\cdot,2r) \le \Phi(\cdot,\cdot,2r) \le C\Phi(\cdot,\cdot,16r)$. 
		Let $P^* = \ell_1^* P^{(11)} + \ell_2^* P^{(12)} \in \mathscr{P}_1$ be the linear solution that minimizes $H(u^\ep,p^\ep;\theta r)$. Then \eqref{est1.lem.h} implies
		\begin{equation}\label{est.lk*}
			\sum_{k=1}^2 |\ell_k^*| \le C\big( H(u^\ep,p^\ep;\theta r)+ \Phi(u^\ep, p^\ep; \theta r) \big) \le C{\theta^{-3/2}} \Phi(u^\ep, p^\ep; r).
		\end{equation}
		By the definition of $H_{\rm 1st}$ and $H$, one has
		\begin{equation}\label{est.H1st-H}
			\begin{aligned}
				& H_{\rm 1st}(u^\ep, p^\ep, \theta r)  \\
				& \le 
				\bigg(\dashint_{B^\ep_{\theta r,+}} \Big|\nabla u^\ep - \nabla\Big(\sum_{k=1}^2 {\ell}^*_k (P^{(1k)} + \ep v^{(1k)}(x/\ep) \Big) \Big|^2 \dd x \bigg)^{1/2} \\
				&\quad + \sup_{s,t\in [1/16,1/4]}
				\bigg| \dashint_{B^\ep_{s\theta r,+}} \Big( p^\ep - \sum_{k=1}^2 {\ell}^*_k q^{(1k)}(x/\ep) \Big)\dd x - \dashint_{B^\ep_{t\theta r,+}} \Big( p^\ep - \sum_{k=1}^2 {\ell}^*_k q^{(1k)}(x/\ep) \Big) \dd x \bigg| \\
				& \le H(u^\ep,p^\ep;\theta r) +  \sum_{k=1}^2 |\ell^*_k| \bigg(\dashint_{B^\ep_{\theta r,+}} |(\nabla v^{(1k)})(x/\ep)|^2\dd x \bigg)^{1/2} \\
				&\qquad 
				+ 2 \sup_{\rho\in [1/16,1/4]}\sum_{k=1}^2 |\ell^*_k|\ \bigg| \dashint_{B_{\rho\theta r,+}^\ep} q^{(1k)}(x/\ep) \dd x - \dashint_{B_{\theta r/2,+}^\ep} q^{(1k)}(x/\ep)\dd x \bigg|.
			\end{aligned}
		\end{equation}
		From Proposition \ref{lem.est.scaled.BL1j}, we have the estimate for the first-order boundary layers
		\begin{equation}\label{est.vq.Br}
			\begin{aligned}
				\sum_{k=1}^2 \bigg\{
				\bigg(\dashint_{B^\ep_{\theta r,+}} 
				|(\nabla v)^{(1k)}(x/\ep)|^2\dd x \bigg)^{1/2} + \bigg(\dashint_{B^\ep_{\theta r,+}} |q^{(1k)}(x/\ep)|^2\dd x\bigg)^{1/2} \bigg\}
				\le C\theta^{-1/2}\big(\frac{\ep}{r}\big)^{1/2}.
			\end{aligned}
		\end{equation}
		Inserting this into \eqref{est.H1st-H} and using \eqref{est.lk*} and \eqref{est.H.C11}, we obtain
		\begin{equation*}
			\begin{aligned}
				& H_{\rm 1st}(u^\ep, p^\ep; \theta r) \\
				& \le H(u^\ep,p^\ep;\theta r) + C{\theta^{-3}}\big(\frac{\ep}{r}\big)^{1/2}\Phi(u^\ep,p^\ep; r) \\
				& \le C\big( \theta + \theta^{-3}\big(\frac{\ep}{r}\big)^{1/12} \big) \Phi(u^\ep,p^\ep;16r) + C\theta^{-3}\bigg(\dashint_{{Q_{10r}}}
				|\mathcal{M}^2_{\ep}[F^\ep]|^3\bigg)^{1/3}.
			\end{aligned}
		\end{equation*}
		This proves the lemma.
	\end{proof}
	%

	%
	\begin{proposition}\label{lem.H1st.iteration}
		Let $L\in(0,\infty)$ and $\Omega$ be a bumpy John domain with constant $L$ according to Definition \ref{def.John2}. Let $(u^\ep,p^\ep)$ be as in Theorem \ref{thm.C1g}. For any $\gamma\in[0,1)$, $\delta\in (0,1)$, $\ep\in(0,\frac12]$ and $r\in [\ep,\frac12]$,
		\begin{equation}\label{est1.lem.H1st.iteration}
			\begin{aligned}
				H_{{\rm 1st}}(u^\ep,p^\ep; r)
				\le C r^\gamma (M+M^{4+2\gamma+\delta}),
			\end{aligned}
		\end{equation}
		where $C$ depends on $L$, $\gamma$ and $\delta$. Here $M$ is the number given in Theorem \ref{thm.C1g}.
	\end{proposition}
	%
	\begin{proof}
		
		Note that it suffices to prove \eqref{est1.lem.H1st.iteration} for $r\in [N_0\ep, 1/N_1]$ with some absolute constant $N_0, N_1 \ge 2$. The cases for $r\in (1/N_1, \frac12]$ or $\ep \ge 1/(N_0 N_1)$ follow directly from the Bogovskii lemma and the Poincar\'{e} inequality. The case $r\in [\ep, N_0\ep]$ follows from the case $r=N_0 \ep$.

		Firstly, using \eqref{est.MFep.large-scale} with $\beta = \gamma+\delta$ (with $\delta\in (0,{\frac{2-\gamma}{2}})$ being arbitrary), we have
		\begin{equation}\label{est.Fep.large-scale Lipschitz}
			\bigg(\dashint_{{Q_{r}}} |\mathcal{M}^2_{\ep}[F^\ep]|^3\bigg)^{1/3} \le C_\delta(M+M^{4+2\gamma +4\delta}) r^{\gamma+\delta}
		\end{equation}
		with $
		C_\delta$ depending on $\delta$.

		Since $\theta\in(0,{\frac18}]$ in Lemma \ref{lem.H1st.excess.decay} is arbitrary, we can choose $\theta$ sufficiently small so that $C\theta \le \frac12 (\frac{\theta}{16})^{\gamma}$ holds in \eqref{est1.lem.H1st.excess.decay}. For such fixed $\theta$, we can find ${\ep_0}\in(0,\frac12)$ depending on $\gamma$ and $\theta$ such that the factor in \eqref{est1.lem.H1st.excess.decay} satisfies
		\begin{equation*}
			C\theta^{-3}\big(\frac{\ep}{r}\big)^{{1/12}} \le \frac{1}{2} \big(\frac{\theta}{16}\big)^\gamma, \quad r\in[\ep/{\ep_0},\frac1{16}]
		\end{equation*}
		in \eqref{est1.lem.H1st.excess.decay}. Then by \eqref{est.Fep.large-scale Lipschitz} and \eqref{est1.lem.H1st.excess.decay}, 
		\begin{equation}\label{est1.proof.lem.H1st.iteration}
			\begin{aligned}
				H_{{\rm 1st}}(u^\ep,p^\ep;\theta r) 
				&\le \big(\frac{\theta}{16} \big)^\gamma \Phi(u^\ep,p^\ep; 16r)
				+C_\delta(M+M^{4+2\gamma+4\delta}) r^{\gamma+\delta}.
			\end{aligned}
		\end{equation}
		Now the key observation is that, for any $(w,\pi) \in \mathscr{Q}_1(\Omega)$, the pair $(U^\ep, \Pi^\ep)$ defined by
		\begin{equation*}
			U^\ep(x) = u^\ep(x) - \ep w(x/\ep), \qquad 
			\Pi^\ep(x) = p^\ep(x) - \pi(x/\ep)
		\end{equation*}
		is still a weak solution of 
		the Stokes system \eqref{S.ep} in Subsection \ref{subsec.set-up and approx.}.
		Therefore, the estimate \eqref{est1.proof.lem.H1st.iteration} still holds if we replace $\Phi(u^\ep,p^\ep;16r)$ by $\Phi(U^\ep,\Pi^\ep;16r)$ for any $(w,\pi) \in \mathscr{Q}_1(\Omega)$. 
		Then taking the infimum over all $(w,\pi) \in \mathscr{Q}_1(\Omega)$, we can further replace $\Phi(U^\ep,\Pi^\ep;16r)$ by $H_{{\rm 1st}}(u^\ep,p^\ep;16r)$. Hence we obtain
		\begin{equation}\label{e.firstexcess}
			\begin{aligned}
				H_{{\rm 1st}}(u^\ep,p^\ep;\theta r) 
				\le \big(\frac{\theta}{16})^\gamma H_{{\rm 1st}}(u^\ep,p^\ep; 16r) 
				+C_\delta(M+M^{4+2\gamma+4\delta}) r^{\gamma+\delta}.
			\end{aligned}
		\end{equation}
		This is the first-order excess decay estimate for the $C^{1,\gamma}$ regularity of $(u^\ep, p^\ep)$. Note that we can eventually replace $4\delta$ by $\delta$ in the right hand side of \eqref{e.firstexcess} as $\delta\in (0,1)$ is arbitrary. Thus, by a simple iteration, we have that for $\ep/\ep_0 \le r \le \frac{\theta}{16}$,
		\begin{equation*}
			H_{{\rm 1st}}(u^\ep,p^\ep; r) 
			\le r^\gamma \big( H_{{\rm 1st}}(u^\ep,p^\ep; r_0) 
			+ C_\delta(M+M^{4+2\gamma+\delta}) \big),
		\end{equation*}
		for some $r_0 \in [\frac{\theta}{16},1]$. Clearly, $H_{{\rm 1st}}(u^\ep,p^\ep; r_0) \le \Phi
		(u^\ep,p^\ep; r_0)$. It remains to show
		\begin{equation*}
			\Phi
			(u^\ep,p^\ep; r_0) 
			\le 
			C(M+M^2).
		\end{equation*}
		Indeed, since $r_0$ is comparable to $1$, the above estimate follows directly from the Poincar\'e inequality and 
		Bogovskii's lemma. The proof is complete.
	\end{proof}
	
	The above theorem directly implies the $C^{1,\gamma}$ estimate for the velocity. To handle the pressure estimate in Theorem \ref{thm.C1g}, we need the following lemma.
	%
	\begin{lemma}\label{lem.1st.h}
		Let $L\in(0,\infty)$ and $\Omega$ be a bumpy John domain with constant $L$ according to Definition \ref{def.John2}. For a given $\rho>0$, let $(\ell_1(\rho),\ell_2(\rho))$ be the pair of real numbers so that
		\begin{equation*}
			(w, \pi) = \sum_{k=1}^2 \ell_k(\rho) (P^{(1k)} + v^{(1k)}, q^{(1k)})
		\end{equation*}
		minimizes $H_{{\rm 1st}}(u^\ep,p^\ep;\rho)$. Then  
		there exists a constant $\ep_1\in (0,1)$ so that for all $\ep\in(0,\ep_1]$
		and $r\in [ \ep/\ep_1,{\frac12}]$, 
		\begin{equation}\label{est1.lem.1st.h}
			\sup_{r_1,r_2\in [r,2r]} \sum_{k=1}^2 |\ell_k(r_1) - \ell_k(r_2)| 
			\le C
			\sup_{t\in [r,2r]}  H_{{\rm 1st}}(u^\ep,p^\ep;t),
			\end{equation}
		where $C$ depends only on $L$.
	\end{lemma}
	%
	\begin{proof}
		By the definition of $H_{\rm 1st}$, the triangle inequality and using that 
		the matrices $\nabla P^{(1k)}$
		are linearly independent over $\R$, if $r\le r_1, r_2 \le 2r$,
		\begin{align*}
			& \sum_{k=1}^2 |\ell_k(r_1) - \ell_k(r_2)| \\
			&\le C\bigg(\dashint_{B^\ep_{r,+}} \Big|\sum_{k=1}^2 (\ell_k(r_1) - \ell_k(r_2)) \nabla P^{(1k)}\Big|^2 \bigg)^{1/2} \\
			&\le C \bigg(\dashint_{B^\ep_{r,+}} \Big|\nabla\Big(\sum_{k=1}^2 (\ell_k(r_1) - \ell_k(r_2)) (P^{(1k)} + \ep v^{(1k)}(x/\ep)\Big)\Big|^2 \dd x \bigg)^{1/2} \\
			&\quad
			+ C \sum_{k=1}^2 |\ell_k(r_1) - \ell_k(r_2)| \bigg(\dashint_{B^\ep_{r,+}} |(\nabla v^{(1k)})(x/\ep)|^2\dd x \bigg)^{1/2} \\
			&\le CH_{{\rm 1st}}(u^\ep,p^\ep;r_1) + CH_{{\rm 1st}}(u^\ep,p^\ep;r_2) + C_0 \big( \frac{\ep}{r} \big)^{1/2} \sum_{k=1}^2 |\ell_k(r_1) - \ell_k(r_2)|,
		\end{align*}
		where in the last inequality, we inserted $u^\ep$ and enlarged the domain from $B_{r,+}^{\ep}$ to $B_{r_i,+}^\ep$ with $i = 1,2$, and applied Proposition \ref{lem.est.scaled.BL1j}.
		Now if $r\ge \ep/\ep_1$ for some small $\ep_1\in(0,1)$ so that $C_0\ep_1^{1/2}
		< \frac12$, then
		\begin{equation*}
			\sum_{k=1}^2 |\ell_k(r_1) - \ell_k(r_2)| \le  C\sum_{i = 1}^2 H_{{\rm 1st}}(u^\ep,p^\ep;r_i).
		\end{equation*}
		This gives the desired estimate.
	\end{proof}
	
	\begin{proofx}{Theorem \ref{thm.C1g}}
		Let $\ep_1\in(0,1)$	be the number in Lemma \ref{lem.1st.h}. Note that it suffices to prove \eqref{est.thmC1a} when 
		$\ep\in(0,\ep_1]$ and $r\in [\ep/\ep_1, \frac1{16}]$ as a familiar argument enables us to remove the smallness condition on $\ep$ and the restriction on $r$.
		The velocity estimate in \eqref{est.thmC1a} follows from 
		the Poincar\'e inequality and
		\eqref{est1.lem.H1st.iteration}. Hence, it suffices to estimate the pressure. Let $(\ell_1(\rho), \ell_2(\rho))$ be as in Lemma \ref{lem.1st.h}. For $r\in [\ep/\ep_1,{\frac1{16}}]$, let $K$ be the integer so that $4^{-K} r \in [\ep/\ep_1, 4\ep/\ep_1)$. By the triangle inequality, the estimate of $q^{(1j)}$
		in Proposition \ref{lem.est.scaled.BL1j},
		\begin{equation}\label{est.Pint.j}
			\begin{aligned}
				&\bigg| \dashint_{B^\ep_{4^{-K}r,+}} \bigg(p^\ep - \sum_{k=1}^2 \ell_k(4^{1-K}r) q^{(1k)}(x/\ep) \bigg) \dd x - \dashint_{B^\ep_{r,+}} \bigg( p^{\ep} - \sum_{k=1}^2 \ell_k (4r) q^{(1k)}(x/\ep) \bigg) \dd x \bigg| \\
				& \le \sum_{i = 1}^{K} \bigg| \dashint_{B^\ep_{4^{i-K-1}r,+}} \bigg(p^\ep - \sum_{k=1}^2 \ell_k (4^{i-K}r) q^{(1k)}(x/\ep) \bigg) \dd x \\
				&\qquad\quad- \dashint_{B^\ep_{4^{i-K} r,+}} \bigg( p^{\ep} - \sum_{k=1}^2 \ell_k (4^{i-K+1}r) q^{(1k)}(x/\ep) \bigg) \dd x \bigg| \\
				& \le \sum_{i=1}^K \Big(H_{\rm 1st} (u^\ep, p^\ep; 4^{i-K+1} r) + \sum_{k=1}^2 |\ell_k (4^{i-K+1} r) - \ell_k (4^{i-K} r)| \big(\frac{\ep}{4^{i-K}r}\big)^{1/2} \Big)\\
				& \le C\sum_{i=1}^K (4^{i-K+1} r)^{{\gamma}} (M + M^{4+2\gamma+\delta}) \\
				& \le Cr^\gamma (M + M^{4+2\gamma+\delta}),
			\end{aligned}
		\end{equation}
		where we have used \eqref{est1.lem.H1st.iteration} and \eqref{est1.lem.1st.h} in the third inequality. Define
		\begin{equation}\label{def.Pbar}
			\bar{P}_1 = \dashint_{B^\ep_{\ep/\ep_1,+}} \bigg(p^\ep - \sum_{k=1}^2 \ell_k (4\ep/\ep_1) q^{(1k)}(x/\ep) \bigg)\dd x.
		\end{equation}
		Then by \eqref{est.Pint.j} and another use of \eqref{est1.lem.H1st.iteration} and \eqref{est1.lem.1st.h}, we have
		\begin{equation*}
			\bigg| \dashint_{B^\ep_{r,+}} \bigg( p^{\ep} - \sum_{k=1}^2 \ell_k (4r) q^{(1k)}(x/\ep) \bigg) \dd x - \bar{P}_1 \bigg| \le Cr^\gamma (M + M^{4+2\gamma+\delta}).
		\end{equation*}
		On the other hand, by Bogovskii's lemma applied to the John domain between $B^\ep_{r,+}$ and $B^\ep_{2r,+}$ given by Definition \ref{def.John2} and \eqref{est.Fep.large-scale Lipschitz}
		with $4\delta$ replaced by $\delta$, 	
		we have
		\begin{equation}\label{est.pi.Pj4r}
			\begin{aligned}
				& \bigg( \dashint_{B_{r,+}^\ep}  \bigg| p^{\ep} - \sum_{k=1}^2 \ell_k(4r) q^{(1k)}(x/\ep) - \dashint_{B^\ep_{r,+}} \bigg( p^{\ep} - \sum_{k=1}^2 \ell_k (4r) q^{(1k)}(x/\ep) \bigg)\dd x \bigg|^2 \dd x \bigg)^{1/2} \\
				& \le C\bigg\{ \bigg(\dashint_{B^\ep_{2r,+}} \Big|\nabla u^\ep - \nabla\Big(\sum_{k=1}^2 \ell_k (4r) (P^{(1k)} + \ep v^{(1k)}(x/\ep)\Big) \Big|^2\dd x \bigg)^{1/2} + \bigg(\dashint_{B^\ep_{{8r},+}} |F^\ep|^2\bigg)^{1/2}\bigg\} \\
				& \le CH_{\rm 1st}(u^\ep,p^\ep; 4r) + C {\bigg(\dashint_{Q_{8r}} |\mathcal{M}^2_{\ep}[F^\ep]|^3\bigg)^{1/3}} \\
				&\le Cr^\gamma (M + M^{4+2\gamma+\delta}).
			\end{aligned}
		\end{equation}
		Combining the above two inequalities, we obtain the desired estimate in \eqref{est.thmC1a} for the pressure. This completes the proof of Theorem \ref{thm.C1g}.
	\end{proofx}

	%
	\subsection{Large-scale $C^{2,\gamma}$ estimate over periodic boundaries}
	\label{sec.C2gamma}
	%
	The goal of this subsection is to prove the large-scale $C^{2,\gamma}$ regularity stated in Theorem \ref{thm.C2g}. In this subsection, we assume $\Omega$ is a periodic bumpy John domain defined in Definition \ref{def.PeriodicJohn}. The argument for $C^{2,\gamma}$ estimate is similar to the $C^{1,\gamma}$ estimate. Throughout, we assume $(w_1 ,\pi_1) \in \mathscr{Q}_1(\Omega)$ and $(w_2, q_2) \in \mathscr{Q}_2(\Omega)$. In other words, for some $\ell_{1k}, \ell_{2j}\in \R$. 
	\begin{equation*}
		\begin{aligned}
			(w_1 ,\pi_1) & = \sum_{k=1}^2 \ell_{1k} (P^{(1k)} + v^{(1k)}, q^{(1k)}),\\
			(w_2 ,\pi_2) & = \sum_{j=1}^6 \ell_{2j} (P^{(2j)} + v^{(2j)}, L^{(2j)}+q^{(2j)}).
		\end{aligned}
	\end{equation*}
	It is important to observe that, by rescaling,
	\begin{equation*}
		\big(\ep w_1(x/\ep) + \ep^2 w_2(x/\ep), {\pi}_1(x/\ep) +\ep {\pi}_2(x/\ep) \big)
	\end{equation*}
	is a solution of the Stokes system in $\Omega^\ep$ with the no-slip boundary condition on $\partial \Omega^\ep$.

	Define the second-order excess as
	\begin{equation*}
		\begin{aligned}
			&H_{\rm 2nd}(u^\ep,p^\ep; \rho) \\
			&= 
			\inf\limits_{\substack{ (w_1,q_1) \in \mathscr{Q}_1(\Omega) \\ (w_2,q_2) \in \mathscr{Q}_2(\Omega)} } \Bigg\{
			\bigg(\dashint_{B^\ep_{\rho,+}} 
			\big|\nabla u^\ep - \nabla\big(\ep w_1(x/\ep) + \ep^2 w_2(x/\ep) \big) \big|^2\dd x \bigg)^{1/2}
			 \\
			&
			+ \sup_{s,t\in [1/16,1/4]}
			\bigg| \dashint_{B^\ep_{s\rho,+}} \big( p^\ep - \pi_1(x/\ep) - \ep \pi_2(x/\ep) \big)\dd x - \dashint_{B^\ep_{t\rho,+}} \big( p^\ep - \pi_1(x/\ep) - \ep \pi_2(x/\ep) \big)\dd x \bigg| \bigg\}.
		\end{aligned}
	\end{equation*}

	\begin{lemma}\label{lem.2nd.est.approx.func.}
		Let $L\in(0,\infty)$ and $\Omega$ be a bumpy periodic John domain with constant $L$ according to Definition \ref{def.PeriodicJohn}. 
		Let $(u^\ep,p^\ep)$ be as in Theorem \ref{thm.C2g}, namely, a weak solution of \eqref{S.ep} in Subsection \ref{subsec.set-up and approx.} satisfying \eqref{e.defMbis}. For all $\ep\in(0,{\frac1{32}}]$, $r\in[2\ep,{\frac1{16}}]$ and $\theta\in(0,{\frac18}]$, 
\begin{equation}
	\begin{aligned}
		H_{{\rm 2nd}}(u^\ep,p^\ep;\theta r)
		&\le C\big( \theta^2 + \theta^{-3}\big(\frac{\ep}{r}\big)^{1/12} \big) \Phi(u^\ep,p^\ep; 16r) \\
		&\qquad
		+ C\theta^{-3}\bigg(\dashint_{{Q_{10r}}}
		|\mathcal{M}^2_{\ep}[F^\ep]|^3\bigg)^{1/3},
	\end{aligned}
\end{equation}
		where $C$ depends only on $L$.
	\end{lemma}
	\begin{proof}
		The proof follows from the strategy developed in Section \ref{sec.large}, in particular from Lemma \ref{lem.est.approx.} to Lemma \ref{lem.excess.decay}. Let $(v_r,q_r)$ be the solution of the approximate problem \eqref{approx.Stokes}.
		We will first use the $C^{2,1}$ estimate of $v_r = (v_{r,1}, v_{r,2}, v_{r,3})$ at the lower boundary $x_3 = -\ep$. Precisely, in view of no-slip Stokes polynomials defined in Section \ref{subsec.Taylor.polys}, the $C^{2,1}$ estimate $v_r$ gives
		\begin{equation}\label{est.C21.vr}
			\begin{aligned}
				&\Big|v_r(x) - \sum_{k=1}^{2} \ell_{1k}^* P^{(1k)}(x+\ep \mathbf{e}_3) - \sum_{j=1}^{6} \ell_{2j}^* P^{(2j)}(x+\ep \mathbf{e}_3)\Big| \\
				& \le 
				C\frac{|x + \ep \mathbf{e}_3|^3}{r^2} \bigg( \dashint_{Q_{r}^\ep} |\nabla v_r|^2 \bigg)^{1/2},
			\end{aligned}
		\end{equation}
		for $x\in Q_{r/2}^\ep$, where we choose $\ell_{1k}^* = \frac{\partial v_{r,k}}{\partial x_3}(-\ep \mathbf{e}_3)$ for $k = 1,2$ and $ \ell_{21}^* = \frac{\partial^2 v_{r,1}}{\partial x_2 \partial x_3}(-\ep \mathbf{e}_3), \ell_{22}^* = {\frac12} \frac{\partial^2 v_{r,1}}{{\partial x_3^2}}(-\ep \mathbf{e}_3), \ell_{23}^* = \frac{\partial^2 v_{r,2}}{\partial x_1 \partial x_3}(-\ep \mathbf{e}_3), \ell_{24}^* = {\frac12} \frac{\partial^2 v_{r,2}}{{\partial x_3^2}}(-\ep \mathbf{e}_3),  \ell_{25}^* = {-\frac12}\frac{\partial^2 v_{r,1}}{\partial x_1 \partial x_3}(-\ep \mathbf{e}_3), \ell_{26}^* = {-\frac12}\frac{\partial^2 v_{r,2}}{\partial x_2 \partial x_3}(-\ep \mathbf{e}_3)$. 
		
		Moreover,
		\begin{equation}\label{est.l1k.l2j}
			\sum_{k=1}^2 |\ell_{1k}^*| + r \sum_{j = 1}^6 |\ell_{2j}^*| \le 
			C \bigg( \dashint_{Q_r^\ep} |\nabla v_r|^2 \bigg)^{1/2}.
		\end{equation}
		%
		Observe that
		\begin{equation*}
			\begin{aligned}
				v_r^*(x) &= v_r(x) - \sum_{k=1}^2 \ell_{1k}^* P^{(1k)}(x+\ep \mathbf{e}_3) - \sum_{j=1}^6 \ell_{2j}^* P^{(2j)}(x+\ep \mathbf{e}_3), \\
				q_r^*(x) & = q_r(x) - \sum_{j=1}^6 \ell_{2j}^* L^{(2j)}(x+\ep \mathbf{e}_3),
			\end{aligned}
		\end{equation*}
		is a solution of the Stokes system in $Q_r^\ep$ with a no-slip condition on $x_3 = -\ep$.
		Therefore, for any $\theta \in (0,\frac18]$ and $r>\ep$, 
		it follows from \eqref{est.C21.vr} and the Caccioppoli inequality in rectangular region $Q_{2\theta r}^\ep$ that
		\begin{equation}\label{est.vr.thetar.C21'}
			\begin{aligned}
				\bigg(\dashint_{Q^\ep_{\theta r}}
				|\nabla v_r^*|
				\bigg)^{1/2}
				&\le\frac{C}{\theta r} 
				\bigg(\dashint_{Q^\ep_{2\theta r}} 
				|v_r^*|^2 \bigg)^{1/2} \\
				& \le C\Big(\theta^{2} + \theta^{-1} \big(\frac{\ep}{r}\big)\Big)
				\bigg( \dashint_{{Q_{r}^\ep}} |\nabla v_r|^2 \bigg)^{1/2}.
			\end{aligned}
		\end{equation}
		Then \eqref{est.l1k.l2j} implies
		\begin{equation}\label{est.vr.thetar.C21}
			\begin{aligned}
		&\bigg(\dashint_{Q^\ep_{\theta r}} \Big|\nabla v_r - \nabla\Big(\sum_{k=1}^2 \ell_{1k}^* P^{(1k)} + \sum_{j=1}^6 \ell_{2j}^* P^{(2j)}\Big) \Big|^2 \bigg)^{1/2} \\
		&\le 		
		C\Big(\theta^{2} + \theta^{-1}\big(\frac{\ep}{r}\big)\Big)
		\bigg( \dashint_{Q_{r}^\ep} |\nabla v_r|^2 \bigg)^{1/2}.
		\end{aligned}
		\end{equation}
		Next, to see the oscillation estimate for the pressure, applying Bogovskii's lemma to $q_r^*$ and the Caccioppoli inequality to $v_r^*$ (combined with \eqref{est.vr.thetar.C21'})
		in Lipschitz domains, we have
		\begin{equation}\label{est.C21.qr}
			\begin{aligned}
				\sup_{s,t\in [1/16,1/4]} \bigg|  \dashint_{Q^\ep_{s\theta r}} q_r^*
				- \dashint_{Q^\ep_{t\theta r}} q_r^* \bigg| 
				\le 
				C\Big(\theta^{2} + \theta^{-1} \big(\frac{\ep}{r}\big)\Big)
				\bigg( \dashint_{Q_{r}^\ep} |\nabla v_r|^2 \bigg)^{1/2}.
			\end{aligned}
		\end{equation}
		Notice that $L^{(2j)}$ are linear functions. Thus, an application of \eqref{est.l1k.l2j} and the triangle inequality to \eqref{est.C21.qr} leads to
		\begin{equation}\label{est.qr.C21at0}
			\begin{aligned}
				&\sup_{s,t\in [1/16,1/4]} \bigg|  \dashint_{Q^\ep_{s\theta r}} \bigg( q_r - \sum_{j=1}^6 \ell_{2j}^* L^{(2j)} \bigg) - \dashint_{Q^\ep_{t\theta r}} \bigg( q_r - \sum_{j=1}^6 \ell_{2j}^* L^{(2j)} \bigg) \bigg| \\
				& \le C\Big(\theta^{2} + \theta^{-1}\big(\frac{\ep}{r}\big)\Big)
				\bigg( \dashint_{Q_{r}^\ep} |\nabla v_r|^2 \bigg)^{1/2}.
			\end{aligned}
		\end{equation}
		This, combined with \eqref{est.vr.thetar.C21},  gives
		\begin{equation*}
			\begin{aligned}
				& \bigg(\dashint_{Q^\ep_{\theta r}} \Big|\nabla v_r - \nabla\Big(\sum_{k=1}^2 \ell_{1k}^* P^{(1k)} + \sum_{j=1}^6 \ell_{2j}^* P^{(2j)}\Big) \Big|^2 \bigg)^{1/2} \\
				& \quad + \sup_{s,t\in [1/16,1/4]} \bigg|  \dashint_{Q^\ep_{s\theta r}} \bigg( q_r - \sum_{j=1}^6 \ell_{2j}^* L^{(2j)} \bigg) - \dashint_{Q^\ep_{t\theta r}} \bigg( q_r - \sum_{j=1}^6 \ell_{2j}^* L^{(2j)} \bigg) \bigg| \\
				& \le C\Big(\theta^{2} + \theta^{-1}\big(\frac{\ep}{r}\big)\Big)
				\bigg( \dashint_{Q_{r}^\ep} |\nabla v_r|^2 \bigg)^{1/2}.
			\end{aligned}
		\end{equation*}
		This is the key second-order excess estimate we need for $(v_r, q_r)$ in $Q_r^\ep$. To proceed, we follow the similar argument developed in Section \ref{sec.large}. Precisely, using an analogue of Lemma \ref{lem.est.comparability} and taking the approximation estimate in Lemma \ref{lem.est.approx.}, we can replace $(v_r,q_r)$ by $(u^\ep, p^\ep)$ with new errors in $u^\ep$ and $F^\ep$. Combined with the energy estimate for \eqref{approx.Stokes},
		we now have
		\begin{equation}\label{est.ue.C21}
			\begin{aligned}
				& \bigg(\dashint_{B^\ep_{r,+}} \Big|\nabla u^\ep - \nabla\Big(\sum_{k=1}^2 \ell_{1k}^* P^{(1k)} + \sum_{j=1}^6 \ell_{2j}^* P^{(2j)}\Big) \Big|^2 \bigg)^{1/2} \\
				& \quad + \sup_{s,t\in [1/16,1/4]} \bigg|  \dashint_{{B^\ep_{s\theta r,+}}} \bigg( p^\ep - \sum_{j=1}^6 \ell_{2j}^* L^{(2j)} \bigg) - \dashint_{B^\ep_{t\theta r,+}} \bigg( p^\ep - \sum_{j=1}^6 \ell_{2j}^* L^{(2j)} \bigg) \bigg| \\
				& \le C\Big(\theta^{2} + \theta^{-3}\big(\frac{\ep}{r})^{1/12} \Big) \bigg( \dashint_{{B^\ep_{5r,+}}}
				 |\nabla u^\ep|^2 \bigg)^{1/2} + C\theta^{-3}\bigg(\dashint_{{Q_{4r}}} |\mathcal{M}^2_{\ep} [F^\ep]|^3\bigg)^{1/3}.
			\end{aligned}
		\end{equation}
		Next, we insert the boundary layers into the above inequality. By \eqref{est.vq.Br} and
		\begin{equation*}
	\begin{aligned}
			&\sum_{j = 1}^6 \bigg\{ \frac{1}{r} \bigg( \dashint_{B^\ep_{\theta r,+}} |
			\ep (\nabla v^{(2j)})(x/\ep)|^2\dd x  \bigg)^{1/2} + \frac{1}{r}\bigg( \dashint_{B^\ep_{\theta r,+}} |\ep q^{(2j)}(x/\ep)|^2\dd x  \bigg)^{1/2} \bigg\} \\
			&\le C\theta^{-1/2}\big(\frac{\ep}{r}\big)^{1/2},
		\end{aligned}
\end{equation*}
		which follows from Proposition \ref{lem.est.scaled.BL2j},
		we obtain from  
		\eqref{est.ue.C21} and \eqref{est.l1k.l2j}
			along with the energy estimate for \eqref{approx.Stokes}
		that
		\begin{equation}\label{est.ue.C21+}
			\begin{aligned}
				&
				\bigg(\dashint_{B^\ep_{\theta r,+}} \Big|\nabla u^\ep 
				- \nabla \Big(\sum_{k=1}^2 \ell_{1k}^* (P^{(1k)} + \ep v^{(1k)}(x/\ep)) \\
				& \qquad\qquad\qquad\qquad 
				+ \sum_{j=1}^6  \ell_{2j}^* (P^{(2j)} + \ep^2 v^{(2j)}(x/\ep))\Big) \Big|^2\dd x \bigg)^{1/2} \\
				& \quad + \sup_{s,t\in [1/16,1/4]} \bigg|  \dashint_{{B^\ep_{s\theta r,+}}} \bigg( p^\ep 
				- \sum_{k=1}^2 {\ell}^*_{1k} q^{(1k)}(x/\ep)
				- \sum_{j=1}^6 \ell_{2j}^* (L^{(2j)} + \ep q^{(2j)}(x/\ep)) \bigg)\dd x \\
				& \qquad\qquad\qquad\qquad - \dashint_{B^\ep_{t\theta r,+}} \bigg( p^\ep - \sum_{k=1}^2 {\ell}^*_{1k} q^{(1k)}(x/\ep)
				- \sum_{j=1}^6 \ell_{2j}^* 
				(L^{(2j)} + \ep q^{(2j)}(x/\ep))
				\bigg)\dd x \bigg| \\
				& \le C\Big(\theta^{2} + \theta^{-3}\big(\frac{\ep}{r})^{1/12} \Big) \bigg( \dashint_{{B^\ep_{5r,+}}}
|\nabla u^\ep|^2 \bigg)^{1/2} + C\theta^{-3}\bigg(\dashint_{{Q_{4r}}} |\mathcal{M}^2_{\ep} [F^\ep]|^3\bigg)^{1/3}.
		\end{aligned}
		\end{equation}
		In view of the definition of $H_{\rm 2nd}$, we arrive at
		\begin{equation*}
			\begin{aligned}
				& H_{\rm 2nd}(u^\ep,p^\ep; \theta r) \\ & \le 
				 C\Big(\theta^{2} + \theta^{-3}\big(\frac{\ep}{r})^{1/12} \Big) \bigg( \dashint_{{B^\ep_{5r,+}}}
|\nabla u^\ep|^2 \bigg)^{1/2} + C\theta^{-3}\bigg(\dashint_{{Q_{4r}}} |\mathcal{M}^2_{\ep} [F^\ep]|^3\bigg)^{1/3},
			\end{aligned}
		\end{equation*}
		which implies the desired estimate.
		\end{proof}

	\begin{proposition}\label{lem.H2nd.iteration}
		Let $L\in(0,\infty)$ and $\Omega$ be a bumpy periodic John domain with constant $L$ according to Definition \ref{def.PeriodicJohn}. Let $(u^\ep,p^\ep)$ be as in Theorem \ref{thm.C2g}. 
		For any $\gamma\in[0,1)$, $\delta\in(0,1)$, $\ep\in(0,\frac12]$ and $r\in[\ep,\frac12]$,
		\begin{equation}\label{est1.lem.H2st.iteration}
			\begin{aligned}
				H_{{\rm 2nd}}(u^\ep,p^\ep; r)
				\le C r^{1+\gamma} (M+M^{6+2\gamma+\delta}),
			\end{aligned}
		\end{equation}
		where $C$ depends on $L$, $\gamma$ and $\delta$. Here $M$ is the number in Theorem \ref{thm.C2g}.
	\end{proposition}
	\begin{proof}
		For any $\gamma\in [0,1)$, we choose 
		an arbitrary $\delta>0$ small enough so that $\delta<\frac{1-\gamma}{2}$. Then applying \eqref{est.MFep.large-scale} with $\beta = 1+\gamma+\delta$, we have
		\begin{equation*}
			\bigg(\dashint_{{Q_{r}}} |\mathcal{M}^2_{\ep} [F^\ep]|^3\bigg)^{1/3} \le 
			C (M + M^{6+2\gamma +4\delta}) r^{1+\gamma+\delta}.
		\end{equation*}
		Now, the rest of the proof is parallel to Proposition \ref{lem.H1st.iteration}. We omit the details.
	\end{proof}
	
	The following lemma is parallel to Lemma \ref{lem.1st.h}.
		
	\begin{lemma}\label{lem.2nd.h}
		Let $L\in(0,\infty)$ and $\Omega$ be a bumpy John domain with constant $L$ according to Definition \ref{def.John2}. For a given $\rho>0$, let
		$\ell_{1k}(\rho)$ and $\ell_{2j}(\rho)$ be the real numbers 
		so that 
		\begin{equation*}
			\begin{aligned}
				(w_1 ,\pi_1) & = \sum_{k=1}^2 \ell_{1k}(\rho) (P^{(1k)} + v^{(1k)}, q^{(1k)}),\\
				(w_2 ,\pi_2) & = \sum_{j=1}^6 \ell_{2j}(\rho) (P^{(2j)} + v^{(2j)}, L^{(2j)}+q^{(2j)})
			\end{aligned}
		\end{equation*}
		minimize $H_{{\rm 2nd}}(u^\ep,p^\ep;\rho)$. Then there exists a constant $\ep_2 \in (0,1)$ so that for all $\ep\in(0,\ep_2]$ and $r\in[\ep/\ep_2,\frac12]$,
		\begin{equation}\label{est1.lem.2st.h}
			\begin{aligned}
				&\sup_{r_1,r_2\in [r,2r]}\sum_{k = 1}^2 |\ell_{1k}(r_1) - \ell_{1k}(r_2)| \\
				&\qquad +\sup_{r_1,r_2\in [r,2r]}\sum_{j = 1}^6  r|\ell_{2j}(r_1) - \ell_{2j}(r_2)| 
				\le C
				{\sup_{t\in [r,2r]} H_{{\rm 2nd}}(u^\ep,p^\ep;t),}
			\end{aligned}
		\end{equation}
		where $C$ depends only on $L$.
	\end{lemma}
	\begin{proof}
		First, observe that for any $a_{k}, b_{j}\in \R$, 
		\begin{equation}\label{est.akbk}
			\sum_{k = 1}^2  |a_{k}| +\sum_{j = 1}^6  |b_{j}| \le C\bigg( \dashint_{B_{1}(0) \cap \{x_3>0\} } \Big| \sum_{k = 1}^2 a_{k} P^{(1k)}  + \sum_{j=1}^6 b_{j} P^{(2j)} \Big|^2\bigg)^{1/2}.
		\end{equation}
		This inequality is true because $P^{(1k)}$ and $P^{(2j)}$ are all linearly independent polynomials. Recall that $P^{(1k)}$ are homogeneous linear functions and $P^{(2j)}$ are homogeneous quadratic functions. This means $P^{(1k)}(rx) = rP^{(1k)}(x)$ and $P^{(2j)}(rx) = r^2P^{(2j)}(x)$. Fix $r_1, r_2 \in [r,2r]$. Applying \eqref{est.akbk} with $a_k = \ell_{1k}(r_1) - \ell(r_2)$ and $b_j = r(\ell_{2j}(r_1) - \ell_{2j}(r_2))$, we have
		\begin{equation*}
			\begin{aligned}
				&\sum_{k = 1}^2  |\ell_{1k}(r_1) - \ell_{1k}(r_2)| +\sum_{j = 1}^6  r|\ell_{2j}(r_1) - \ell_{2j}(r_2)| \\
				&\le C\bigg( \dashint_{B_{1}(0) \cap \{x_3>0\}} \Big| \sum_{k = 1}^2 ( \ell_{1k}(r_1) - \ell_{1k}(r_2)) P^{(1k)}  + \sum_{j=1}^6 r(\ell_{2j}(r_1) - \ell_{2j}(r_2) ) P^{(2j)} \Big|^2\bigg)^{1/2} \\
				& \le \frac{C}{r}\bigg( \dashint_{B_{r}(0) \cap \{x_3>0\}} \Big| \sum_{k = 1}^2 ( \ell_{1k}(r_1) - \ell_{1k}(r_2)) P^{(1k)}  + \sum_{j=1}^6 (\ell_{2j}(r_1) - \ell_{2j}(r_2) ) P^{(2j)} \Big|^2\bigg)^{1/2} \\
				& 
				\le C\bigg( \dashint_{B_{r}(0) \cap \{x_3>0\}} \Big| \sum_{k = 1}^2 ( \ell_{1k}(r_1) - \ell_{1k}(r_2)) \nabla P^{(1k)}  + \sum_{j=1}^6 (\ell_{2j}(r_1) - \ell_{2j}(r_2) ) \nabla P^{(2j)} \Big|^2\bigg)^{1/2},
			\end{aligned}
		\end{equation*}
		where the Poincar\'e inequality has been applied in the last line.
		Now, inserting $u^\ep, v^{(1k)}(x/\ep)$ and $v^{(2j)}(x/\ep)$ into the right-hand side, and using the triangle inequality, we obtain
		\begin{equation*}
	\begin{aligned}
		&\sum_{k = 1}^2  |\ell_{1k}(r_1) - \ell_{1k}(r_2)| +\sum_{j = 1}^6  r|\ell_{2j}(r_1) - \ell_{2j}(r_2)| \\
		&\le C\bigg( \dashint_{B_{r,+}^\ep} \Big| \nabla\Big(\sum_{k = 1}^2 ( \ell_{1k}(r_1) - \ell_{1k}(r_2)) (P^{(1k)} + \ep v^{(1k)}(x/\ep))  \\
		& \qquad \qquad \qquad \qquad
		+ {\sum_{j=1}^6} (\ell_{2j}(r_1) - \ell_{2j}(r_2)) (P^{(2j)} + \ep^2 v^{(2j)}(x/\ep)) \Big)\Big|^2\dd x\bigg)^{1/2} \\
		& \qquad \qquad 
		+ C\sum_{k = 1}^2 |\ell_{1k}(r_1) - \ell_{1k}(r_2)| \bigg( \dashint_{B_{r,+}^\ep} |(\nabla v^{(1j)}) (x/\ep)|^2\dd x \bigg)^{1/2} \\
		& \qquad \qquad 
		+ \frac{C}{r}\sum_{j = 1}^6 r|\ell_{2j}(r_1) - \ell_{2j}(r_2)| \bigg( \dashint_{B_{r,+}^\ep} 
		|\ep(\nabla v^{(2j)})(x/\ep)|^2\dd x \bigg)^{1/2}\\
		& \le CH_{\rm 2nd}(u^\ep,p^\ep; r_1) +CH_{\rm 2nd}(u^\ep,p^\ep; r_2) \\
		& \qquad \qquad + C_1\big( \frac{\ep}{r} \big)^{1/2} \sum_{j=1}^2 |\ell_{1k}(r_1) - \ell_{1k}(r_2)| 
		+ C_2\big( \frac{\ep}{r} \big)^{1/2} \sum_{k=1}^6 r|\ell_{2j}(r_1) - \ell_{2j}(r_2)|,
	\end{aligned} 
\end{equation*}
		where Proposition \ref{lem.est.scaled.BL2j} is applied in the last inequality.
		Thus, if $r > \ep/\ep_2$ for some sufficiently small constant $\ep_2 \in (0,1)$ so that $C_1 (\ep/r)^{1/2} < \frac12$ and $C_2 (\ep/r)^{1/2} < \frac12$, then
		\begin{equation*}
			\begin{aligned}
				\sum_{k = 1}^2 |\ell_{1k}(r_1) - \ell_{1k}(r_2)| +\sum_{j = 1}^6  r|\ell_{2j}(r_1) - \ell_{2j}(r_2)| 
				& \le  C \sum_{i = 1}^2 H_{\rm 2nd}(w^\ep,\pi^\ep; r_i). 
			\end{aligned}
		\end{equation*}
		This leads to the assertion.
	\end{proof}
	
	\begin{proofx}{Theorem \ref{thm.C2g}}
		The estimate for the velocity 
		is contained in \eqref{est1.lem.H2st.iteration}. The estimate for pressure can be derived similarly as Theorem \ref{thm.C1g}. The details are left to the reader.
	\end{proofx}
	
	\subsection{Liouville-type results}
	\label{sec.liouville}

	As an application of the construction of boundary layers and uniform regularity, a Liouville-type theorem for Stokes systems can be shown by the large-scale Lipschitz, $C^{1,\gamma}$ and $C^{2,\gamma}$ estimates. We point out that our large-scale regularity results hold also for the linear Stokes system, although with linear dependence on $M$ in the right-hand sides of \eqref{est2.theo.lip.nonlinear}, \eqref{est.thmC1a} and \eqref{est.C2gamma}. The proofs are simpler, 
	using that the source term $F^\ep = 0$. To describe the Liouville-type theorem, consider the Stokes system in the entire $\Omega$
	\begin{equation}\label{eq.forLiouville}
		\left\{
		\begin{array}{ll}
			-\Delta u+\nabla p=0, &x\in\Omega \\
			\nabla\cdot u=0, &x\in\Omega \\
			u=0, &x\in\partial\Omega, 
		\end{array}
		\right.
	\end{equation}
	where $\Omega$ is a bumpy John domain according to Definition \ref{def.John2}. Let $B_R = B_R(0)$. We state the Liouville-type theorem as follows. Its proof follows from a routine rescaling of the large-scale regularity estimates. Notice that this result complements Corollary \ref{cor.liouville} already stated above.
	\begin{theorem}\label{thm.liouville}
		Let $\Omega$ be a bumpy John domain according to Definition \ref{def.John2}. Let $(u,p)$ be a weak solution of \eqref{eq.forLiouville}. 
		\begin{enumerate}[label=(\roman*)]
			\item If for some $\sigma\in (0,1)$
			\begin{equation*}
				\liminf_{R\to \infty}  \frac{1}{R^{1+\sigma}} \bigg( \dashint_{B_{R,+}} |u|^2 \bigg)^{1/2} = 0,
			\end{equation*}
			then $(u,p) \in \mathscr{Q}_1(\Omega)$ (up to a constant for $p$).
			\item In addition, assume $\Omega$ is periodic bumpy John domain according to Definition \ref{def.PeriodicJohn}. If for some $\sigma\in (0,1)$,
			\begin{equation*}
				\liminf_{R\to \infty}  \frac{1}{R^{2+\sigma}} \bigg( \dashint_{B_{R,+}} |u|^2 \bigg)^{1/2} = 0,
			\end{equation*}
			then $(u,p) \in \mathscr{Q}_1(\Omega) + \mathscr{Q}_2(\Omega)$ (up to a constant for $p$).
		\end{enumerate}
	\end{theorem}
	
	\appendix

	\section{Bogovskii's lemma and some applications}\label{appendix.Bogovskii}
	%
	For a bounded open set $D\subset\R^3$ and $p\in(1,\infty)$, let
	\begin{align*}
		L^q_0(D)=\bigg\{f\in L^q(D)~\bigg|~\dashint_D f=0\bigg\}.
	\end{align*}

	\begin{theorem}[{\cite[Theorem 4.1]{ADM06}}]\label{theo.acosta}
		Let $\Omega\subset\R^3$ be a bounded John domain according to Definition \ref{def.John} with constant $L$. There exists an operator $\B: L^q_0(\Omega)\to W^{1,q}_0(\Omega)^3$ satisfying 
		\begin{align*}
			\nabla\cdot \B[f]=f \quad {\rm in} \ \, \Omega
		\end{align*}
		and
		\begin{align}\label{e.estbogtheom}
			\|\B[f]\|_{W^{1,q}(\Omega)} \le C\|f\|_{L^q(\Omega)},	
		\end{align}
		with $C$ depending on $L$.
	\end{theorem}

	\begin{lemma}\label{lem.pressure.recovered}
		Let $\Omega$ be a bounded John domain according to Definition \ref{def.John}. 
		Set 
		\begin{align*}
			H^1_{0,\sigma}(\Omega)
			:=\{u\in {H^1_{0}(\Omega)^3}~|~\nabla\cdot u=0 \ \, {\rm in} \  \Omega\}.
		\end{align*}
		Let $f\in L^2(\Omega)^3$ and $F\in L^2(\Omega)^{3\times3}$. If $u\in H^1(\Omega)^3$ is a weak solution of the Stokes equations in the sense
		\begin{align*}
			\int_\Omega \nabla u \cdot \nabla \varphi 
			=\int_\Omega f\cdot\varphi - \int_\Omega F \cdot \nabla \varphi, \quad 
			\varphi\in {H^1_{0,\sigma}(\Omega)},
		\end{align*}
		then there exists a function $p\in L^2(\Omega)$ unique up to a constant for which we have
		\begin{align*}
			\int_\Omega \nabla u \cdot \nabla \phi 
			- 	\int_\Omega p (\nabla\cdot\phi) 
			=\int_\Omega f\cdot\phi - \int_\Omega F \cdot \nabla \phi, \quad \phi\in H^{1}_0(\Omega)^3.
		\end{align*}
		Namely, the pair $(u,p)$ is a weak solution of the Stokes equations. Moreover,
		\begin{equation}\label{est.Bogovskii}
			\| p - {\dashint_{\Omega}} p \|_{L^2(\Omega)} \le C\Big( \| u \|_{L^2(\Omega)} + { \rm diam}(\Omega)\| f \|_{L^2(\Omega)} + \| F \|_{L^2(\Omega)} \Big),
		\end{equation}
	where ${\rm diam}(\Omega)$ denotes the diameter of $\Omega$.
	\end{lemma}

	A direct application of Bogovskii's operator is the Caccioppoli inequality for the Stokes equations. Let $Q_{r,+} = Q_r \cap \{ x_3 >0 \}$. Suppose $(u,p)$ is a weak solution of
	\begin{equation}\label{eq.Stokes.Qr}
		\left\{
		\begin{array}{ll}
			-\Delta u+\nabla p
			= \nabla\cdot F&\mbox{in}\ Q_{2r,+}\\
			\nabla\cdot u=0&\mbox{in}\ Q_{2r,+}\\
			u =0&\mbox{on}\ Q_{2r} \cap \{ x_3 =0 \}.
		\end{array}
		\right.
	\end{equation}
	The following is the Caccioppoli inequality over flat boundaries whose proof is classical \cite[Theorem 1.1]{GM82} (the interior Caccioppoli inequality is similar).
	\begin{lemma}\label{lem.StrongCaccioppoli}
		Let $F\in L^2(Q_{2r,+})^{3\times 3}$
		and let $(u,p)\in H^1(Q_{2r,+})^3 \times L^2(Q_{2r,+})$ be a weak solution to \eqref{eq.Stokes.Qr}.
		Then,
		\begin{equation}\label{est.Stokes.Qr}
			\| \nabla u \|_{L^2(Q_{r,+})} \le C\bigg( \frac{1}{r} \| u \|_{L^2(Q_{2r,+})} + \| F\|_{L^2(Q_{2r,+})} \bigg),
		\end{equation}
		where the constant $C$ is independent of $r$.
	\end{lemma}

	%
	Now, consider the Stokes equations over John boundaries
	\begin{equation}\label{appendix.S.ep}
		\left\{
		\begin{array}{ll}
			-\Delta u^\ep+\nabla p^\ep
			= \nabla\cdot F^\ep&\mbox{in}\ B^\ep_{4r,+}\\
			\nabla\cdot u^\ep=0&\mbox{in}\ B^\ep_{4r,+}\\
			u^\ep=0&\mbox{on}\ \Gamma^\ep_{4r}.
		\end{array}
		\right.
	\end{equation}
	Unfortunately, the Caccioppoli inequality in the form of (\ref{est.Stokes.Qr}) cannot be derived for the weak solution of \eqref{appendix.S.ep} by the usual iteration argument (see e.g., \cite[Lemma 0.5]{GM82} or \cite[Chapter V, Lemma 3.1]{G83}) due to the assumption that the John domain condition (after rescaling) holds only for scales $r\ge \ep$. Actually, we only have a weaker Caccioppoli inequality valid for $r\ge \ep$, which is sufficient for us to show a (large-scale) Meyers estimate.
	%

	%
\begin{lemma}[A weak Caccioppoli inequality]\label{appendix.lem.S.Caccioppoli.ineq.}
	Let $L\in(0,\infty)$ and $\Omega$ be a bumpy John domain with constant $L$ according to Definition \ref{def.John2}. Let $\ep\in(0,\frac12]$ and $F^\ep\in L^2(B^\ep_{4r,+})^{3\times3}$, and let $(u^\ep, p^\ep)\in H^1(B^\ep_{4r,+})^3 \times L^2(B_{4r,+}^\ep)$ be a weak solution to \eqref{appendix.S.ep} with $r \ge \ep$. 
	Then for any $\theta \in (0,1)$,
	\begin{align}\label{est.appendix.lem.S.Caccioppoli.ineq.}
		\begin{split}
			\|\nabla u^\ep\|_{L^2(B^\ep_{{r},+})}
			\le \theta \|\nabla u^\ep\|_{L^2(B^\ep_{{4r},+})}
		+
			\frac{C}{\theta r} \|u^\ep\|_{L^2(B^\ep_{4r,+})} 
			+ C\|F^\ep\|_{L^2(B^\ep_{4r,+})}\,,
		\end{split}
	\end{align}
	where the constant $C$ depends only on $L$. 
	In particular $C$ is independent of 
	$\theta,\ep$ and $r$. Moreover, if $r\ge 4\ep$, then by the standard interior Caccioppoli inequality and a covering argument as in the proof of Lemma \ref{lem.selfimprove}, $B^\ep_{4r,+}$ may be replaced by $B^\ep_{2r,+}$ on the right-hand side of \eqref{est.appendix.lem.S.Caccioppoli.ineq.}.
\end{lemma}

\begin{proof}
	Let $\phi_r$ be a smooth cut-off function so that $\phi_r(x) = 1$ for ${x} \in B_r$, $\phi(x) = 0$ for $x\notin B_{2r}$ and $|\nabla \phi| \le C/r$. Integrating the first equation of \eqref{appendix.S.ep} against $u^\ep \phi^2$, we have
	\begin{equation}\label{eq.LemmaA4}
		\int_{B^\ep_{2r,+}} \nabla u^\ep \cdot \nabla u^\ep \phi^2 = -\int_{B^\ep_{2r,+}} \phi\nabla u^\ep \cdot (\nabla \phi \otimes u^\ep)  + \int_{B^\ep_{2r,+}} \nabla p^\ep \cdot u^\ep\phi^2 - \int_{B^\ep_{2r,+}} F^\cdot \nabla (u^\ep \phi^2).
	\end{equation}
	The first and third terms on the right-hand side are routine.
	To deal with the pressure, by Definition \ref{def.John2} of the bumpy John domain $\Omega$, we use the Bogovskii operator in a John domain $\Omega_{2r}^\ep$ satisfying $B_{2r,+}^\ep\subset\Omega_{2r}^\ep\subset B_{4r,+}^\ep$ and \eqref{est.Bogovskii} to obtain
	\begin{equation*}
		\bigg( \int_{\Omega_{2r}^\ep} |p^\ep - {\dashint_{\Omega_{2r}^\ep}} p^\ep|^2 \bigg) \le C\Big( \| \nabla u^\ep \|_{L^2(\Omega_{2r}^\ep)} + \| F^\ep \|_{L^2(\Omega_{2r}^\ep)} \Big).
	\end{equation*}
	Let $L = \dashint_{\Omega_{2r}^\ep} p^\ep$. Then, using the above estimate and $\nabla\cdot u^\ep = 0$,
	\begin{equation*}
		\begin{aligned}
		\bigg| \int_{B^\ep_{2r,+}} \nabla p^\ep \cdot u^\ep\phi^2 \bigg| &= \bigg| \int_{B^\ep_{2r,+}} \nabla (p^\ep - L) \cdot u^\ep\phi^2 \bigg| \\
		&= \bigg| \int_{B^\ep_{2r,+}} (p^\ep - L) u^\ep\cdot 2\phi \nabla\phi \bigg| \\
		& \le \frac{C}{r}\Big( \| \nabla u^\ep \|_{L^2(\Omega_{2r}^\ep)} + \| F^\ep \|_{L^2(\Omega_{2r}^\ep)} \Big) \| u^\ep \|_{L^2(B_{2r,+}^\ep)} \\
		& \le \theta^2 \| \nabla u^\ep \|_{L^2(B_{4r,+}^\ep)}^2 + \frac{C}{\theta^2 r^2} \| u^\ep \|_{L^2(B_{4r,+}^\ep)}^2 + C\|F^\ep\|_{L^2(B^\ep_{4r,+})},
	\end{aligned}
	\end{equation*}
	for any $\theta \in (0,1)$. In view of \eqref{eq.LemmaA4}, this gives the desired estimate by a standard argument.
\end{proof}
	

	\section{Large-scale estimates for the Green function}
	\label{app.green}

	This appendix is devoted to the study of the Green function for the Stokes equations in a bumpy John half-space according to Definition \ref{def.John2}. 
	The large-scale estimates proved in Section \ref{sec.large} will be applied. The basic scheme is to derive estimates for the velocity part of the Green function directly from the interior and large-scale boundary Lipschitz estimates. For this we follow the strategy pioneered in \cite{AL87,AL91}. Then, we deduce the estimates for the pressure part of the Green function from Bogovskii's lemma and the estimates for the velocity part.

	
	We use $B_R(x) = Q_R(x)$ to denote the cube centered at $x$ with side length $2R$. If the center is not important in the context, it is abbreviated as $B_R$. Throughout this appendix, $\Omega_{\le N}$, $\Omega_{\ge N}$, $\Omega_{< N}$, and $\Omega_{> N}$ defined around \eqref{e.defOmegaN} will be used. Moreover, let $\hat{x}$ denote the projection of $x\in\R^3$ on $\partial\R^3_+$.

	\subsection{Construction of the Green function}
	
	Let $D$ be an open set in $\R^3$. Denote by $Y^{1,2}(D)$ the space of functions 
	\begin{equation}
		\{u\in L^6(D)~|~\nabla u\in L^2(D)^3\}
	\end{equation}
	equipped with the norm $\| u \|_{Y^{1,2}(D)}=\| u \|_{L^6(D)} + \| \nabla u\|_{L^2(D)}$. Let $Y^{1,2}_0(D)$ be the closure of $C_0^\infty(D)$ under $\| \cdot \|_{Y^{1,2}(D)}$. The closed subspace of $Y^{1,2}_0(D)^3$
	\begin{equation}
		\{u\in Y^{1,2}_0(D)^3~|~\nabla\cdot u=0 \ \, {\rm in} \  D\}
	\end{equation}
	is denoted by $Y^{1,2}_{0,\sigma}(D)$. Note that, when the Lebesgue measure of $D$ is finite, we have $Y^{1,2}_0(D)=H^{1,2}_0(D)$ by the Sobolev inequality $\| u\|_{L^6(D)} \le C\| \nabla u\|_{L^2(D)}$ if $u\in C_{0}^\infty(D)$. Moreover, we see that $Y^{1,2}_{0,\sigma}(D)$ as well as $Y^{1,2}_{0}(D)^3$ is a Hilbert space with an inner product $\langle u,v \rangle=\int_D \nabla u \cdot \nabla v$.
	
	Let $\Omega$ be a bumpy John domain with constant $L\in(0,\infty)$ according to the Definition \ref{def.John2}. Based on similar proofs in \cite{HK07} and \cite{CL17} and using the large-scale Lipschitz estimate of Theorem \ref{theo.lip.nonlinear} proved in Section \ref{sec.large}, we can construct the Green function $(G,\Pi)=(G(x,y),\Pi(x,y))$, which satisfies the following properties:

	\begin{enumerate}[label=(\roman*)]	

		\item For any $q\in[1,\frac32)$, $G(\cdot,y) \in W^{1,q}_{0,{\rm loc}}(\overline{\Omega})^{3\times 3}$
		and $G(\cdot,y) \in Y^{1,2}(\Omega\setminus B_r(y) )^{3\times 3}$ for each $y\in \Omega$ and $r>0$. Moreover, $\Pi(\cdot,y)\in L^2_{{\rm loc}}(\overline{\Omega\setminus B_r(y)})^{3\times1}$ for each $y\in \Omega$ and $r>0$.
		\item $(G(\cdot,y),\Pi(\cdot,y))$ satisfies, for each $y\in \Omega$, 
		\begin{equation}\label{eq.def.GPi}
			\int_{\Omega} \nabla G(\cdot,y)\cdot \nabla \phi 
			- \int_{\Omega} \Pi(\cdot,y)(\nabla\cdot\phi) =\phi(y), \quad \phi\in C_{0}^\infty(\Omega)^3.
		\end{equation}
		\item For all $f\in C^{\infty}_{0}(\Omega)^3$, if
		the function $(u,p) \in Y^{1,2}_{0,\sigma}(\Omega)^3 \times L^2_{\rm loc}(\overline{\Omega})$, with $p(x)\to 0$ as $x_3\to \infty$, satisfies the Stokes equations in the sense 
		\begin{equation}\label{est.var.def}
			\int_{\Omega} \nabla u \cdot \nabla \phi 
			- \int_{\Omega} p (\nabla\cdot \phi)
			=\int_{\Omega} f \cdot \phi, \quad \phi\in C_{0}^\infty(\Omega)^3,
		\end{equation}
		then
		\begin{equation}\label{eq.GreenFormula}
			u(x)=\int_{\Omega} G(x,y) f(y) \dd y, \qquad p(x) = \int_{\Omega} \Pi(x,y) \cdot f(y) \dd y.
		\end{equation}
	\end{enumerate}

	We describe how to obtain $(G,\Pi)$ meeting Properties (i)-(iii) above. The existence and basic estimates of the velocity component $G(x,y)$ follow from a similar argument as \cite[Theorem 4.1]{HK07} by working in the Hilbert space $Y^{1,2}_{0,\sigma}(\Omega)$. In fact, there is $G(x,y)$ such that $u(x)$ defined in \eqref{eq.GreenFormula} belongs to $u\in Y^{1,2}_{0,\sigma}(\Omega)$ and is the unique solution of the Stokes equations in the sense 
	\begin{equation}\label{est.var.def'}
		\int_{\Omega} \nabla u \cdot \nabla \varphi 
		=\int_{\Omega} f \cdot \varphi, \quad \text{for any } \varphi\in Y^{1,2}_{0,\sigma}(\Omega).
	\end{equation}
	Then, by using Lemma \ref{lem.pressure.recovered} on each bounded John subdomain, one sees that there is a pressure $p\in L^2_{{\rm loc}}(\overline{\Omega})$ for which we have \eqref{est.var.def}, uniquely determined under the condition $p(x)\to 0$ as $x_3\to \infty$.

	When constructing the pressure component $\Pi(x,y)$ in \eqref{eq.GreenFormula}, we need a careful analysis since the domain is unbounded unlike in \cite{CL17}. Here the oscillation estimate of $p$ will play a crucial role. For an open set $E$, define the oscillation of $p$ in $E$ by
	\begin{equation}\label{def.osc}
		\osc_E p = \sup_{x,y\in E} |p(x) - p(y)|.
	\end{equation}
	The following lemma shows a fundamental oscillation estimate for the pressure.
	
	\begin{lemma}\label{lem.osc.interior}
		Let $L\in (0,\infty)$ and $\Omega$ be a bumpy John domain with constant $L$ according to Definition \ref{def.John2}. Then we have the following statements.
		\begin{enumerate}[label=(\roman*)]
			\item Let $\overline{B_R} \subset \Omega$. If $-\Delta u + \nabla p = 0$ and $\nabla\cdot u = 0$ in $B_R$, then
			\begin{equation}\label{est1.lem.osc.interior}
				\osc_{B_{R/2}} p \le C\bigg( \dashint_{B_{R}} |\nabla u|^2 \bigg)^{1/2},
			\end{equation}
			where $C$ is a universal constant.
			\item Let $z\in\partial\R^3_+$ and let $R>2$. If $-\Delta u + \nabla p = 0$ and $\nabla\cdot u = 0$ in $\Omega \cap B_R(z)$ and $u=0$ on $\partial\Omega \cap B_R(z)$, then
			\begin{equation}\label{est2.lem.osc.interior}
				\osc_{\Omega_{>2}\cap B_{R/2}(z)} p 
				\le C\bigg( \dashint_{\Omega\cap B_R(z)} |\nabla u|^2 \bigg)^{1/2},
			\end{equation}
			where $C$ depends on $L$ and is independent of $z$ and $R$.
		\end{enumerate}
	\end{lemma}
	\begin{proof}
		The interior case (i) is classical and the proof is omitted. Let us prove the boundary case (ii). Since only the case where $R$ is sufficiently large is nontrivial, we assume that $R>32$. For any $x\in\Omega_{>2}\cap B_{R/2}(z)$,	the mean value property of harmonic functions 
		yields
		\begin{equation}\label{eq.p.mvp}
			p(x) = \dashint_{\mathcal{B}_r(x)} p, \qquad\mbox{if}\quad \mathcal{B}_r(x):=\{y|\ |y-x|<r\} \subset\Omega.
		\end{equation}
		Here we assume $r = \frac{x_3}{2} \le \frac{R}{16}$, hence $r>1$. Note that if $x_3 > \frac{R}{8}$, the oscillation can be handled by the interior estimate \eqref{est1.lem.osc.interior}.
		
		Recall that $\hat{x}$ is the projection of $x$ on $\partial\R^3_+$.
		Let $\Omega_{4r}(\hat{x})$ be a John domain given by Definition \ref{def.John2} so that $\Omega \cap B_{4r}(\hat{x}) \subset \Omega_{4r}(\hat{x}) \subset \Omega\cap B_{8r}(\hat{x})$. 
		Clearly $\mathcal{B}_r(x) \subset \Omega_{4r}(\hat{x})$, 
		$B_{R/4}(\hat{x}) \subset B_{3R/4}(z)$ and $B_{R/2}(\hat{x}) \subset B_{R}(z)$.
		By \eqref{eq.p.mvp} and the Bogovskii lemma,
		\begin{equation}\label{est.p-osc}
			\begin{aligned}
				\bigg| p(x) - \dashint_{\Omega_{4r}(\hat{x})} p\bigg| & \le \dashint_{\mathcal{B}_r(x)} |p - \dashint_{\Omega_{4r}(\hat{x})} p| \le C\dashint_{{\Omega}_{4r}(\hat{x})} |p - \dashint_{\Omega_{4r}(\hat{x})} p| \\
				& \le C\bigg( \dashint_{B_{8r}(\hat{x}) \cap \Omega} |\nabla u|^2 \bigg)^{1/2} \le C\bigg( \dashint_{B_{R/2}(\hat{x})\cap \Omega} |\nabla u|^2 \bigg)^{1/2} \\
				& \le C\bigg( \dashint_{B_{R}(z)\cap \Omega} |\nabla u|^2 \bigg)^{1/2},
			\end{aligned}
		\end{equation}
		where we also used the Lipschitz estimate of $u$ in the fourth inequality. Similarly, we have
		\begin{equation*}
			\bigg| \dashint_{B_{4r}(\hat{x})\cap \Omega} p - \dashint_{\Omega_{4r}(\hat{x})} p\bigg|  \le C\bigg( \dashint_{B_{R}(z) \cap \Omega} |\nabla u|^2 \bigg)^{1/2}.
		\end{equation*}
		On the other hand, by the pressure estimate for the Stokes system (an analogue of Theorem \ref{theo.lip.nonlinear} with linear dependence on $M$),
		\begin{equation*}
			\bigg| \dashint_{B_{4r}(\hat{x}) \cap \Omega } p - \dashint_{B_{R/4}(\hat{x}) \cap \Omega} p\bigg| \le C\bigg( \dashint_{B_{R}(z) \cap \Omega} |\nabla u|^2 \bigg)^{1/2}.
		\end{equation*}
		Similar to \eqref{est.p-osc}, because $B_{R/4}(\hat{x}) \subset B_{3R/4}(z)$, we obtain
		\begin{equation*}
			\bigg| \dashint_{B_{R/4}(\hat{x}) \cap \Omega} p -\dashint_{B_{3R/4}(z) \cap \Omega} p \bigg| \le C\bigg( \dashint_{B_{2R}(z) \cap \Omega} |\nabla u|^2 \bigg)^{1/2}.
		\end{equation*}
		Finally, combining the above estimates, we arrive at
		\begin{equation}\label{e.estoscp}
			|p(x) - \dashint_{B_{3R/4}(z)\cap \Omega} p| \le C\bigg( \dashint_{B_{2R}(z) \cap \Omega } |\nabla u|^2 \bigg)^{1/2}.
		\end{equation}
		Since $x\in \Omega_{>2}\cap B_{R/2}(z)$ is arbitrary, this implies the estimate \eqref{est2.lem.osc.interior} with $2R$ in the right-hand side instead of $R$. Then a covering argument using \eqref{est1.lem.osc.interior} yields the desired estimate \eqref{est2.lem.osc.interior}. The proof is complete.
	\end{proof}

	\begin{remark}\label{rmk.oscp.f}
		The interior oscillation estimate holds also for $-\Delta u + \nabla p = f$ and $\nabla \cdot u = 0$ in $B_R$ provided $f\in L^q(B_R)^3$ for some $q>3$. Precisely, by the classical Schauder theory,
		\begin{equation*}
			\osc_{B_{R/2}} p \le C\bigg( \dashint_{B_{R}} |\nabla u|^2 \bigg)^{1/2} + CR \bigg( \dashint_{B_{R}} |f|^q \bigg)^{1/q},
		\end{equation*}
		where $C$ is a universal constant.
	\end{remark}

	\medskip

		Now, we are ready to construct $\Pi(x,y)$ and prove Properties (i)-(iii) of $(G,\Pi)$. For a given $f\in C_0^\infty ( B_R(0)\cap\Omega)^3$ with $R>32$, we consider the Stokes equations \eqref{est.var.def} with $u$ given by \eqref{eq.GreenFormula} and the associated pressure $p\in L^2_{{\rm loc}}(\overline{\Omega})$.
	For $x\in \Omega_{>2}$ such that $|x| \geq 4R$, we set $r=\frac{|x|}{2} \ge 2R$.
	Since $f$ is supported in $B_R(0)\cap\Omega$, we have $-\Delta u + \nabla p = 0$ in $B_{r}(x)\cap\Omega$. Moreover, by an energy estimate using \eqref{est.var.def'}, we have
	\begin{equation}\label{e.enestnablau}
		\bigg( \dashint_{B_r(x)\cap\Omega} |\nabla u|^2 \bigg)^{1/2}
		\le \frac{CR}{r^{3/2}} \| {f}\|_{L^2(B_R(0)\cap\Omega)}.
	\end{equation}
	Therefore, Lemma \ref{lem.osc.interior} and a covering argument imply the oscillation estimate of $p$, namely,
	\begin{equation*}
		\osc_{\Omega_{>2} \cap B_{2r}(0)\setminus \overline{B_r(0)} } p 
		\le \frac{C R}{r^{3/2}} \| {f}\|_{L^2(B_R(0)\cap\Omega)}.
	\end{equation*}
	This further implies
	\begin{equation}\label{est.oscp}
		\begin{aligned}
			\osc_{\Omega_{>2} \setminus \overline{B_r(0)} } p 
			& \le \sum_{k = 1}^\infty \osc_{\Omega_{>2} \cap B_{2^k r}(0)\setminus \overline{B_{2^{k-1} r}(0)} } p \\
			& \le \sum_{k = 1}^\infty \frac{C R}{{(2^{k-1} r)}^{3/2}} 
			\| {f}\|_{L^2(B_R(0)\cap\Omega)}\\
			& \le \frac{C R}{r^{3/2}} \| {f}\|_{L^2(B_R(0)\cap\Omega)}.
		\end{aligned}
	\end{equation}
	This shows that $p(x)$ converges to a constant as $x\to \infty$. By the assumption that $p(x) \to 0$ as $x_3 \to \infty$, we know the limiting constant is zero. Hence, in view of \eqref{est.oscp}, we derive
	\begin{equation}\label{est.px.ptws1}
		|p(x)| \le \frac{C R}{|x|^{3/2}} \| {f}\|_{L^2(B_R(0)\cap\Omega)},
	\end{equation}
	for all $x\in \Omega_{>2}$ satisfying $|x|\geq 4R$. 
	Moreover, by arguing in a similar manner as in Step 3 in the proof of Theorem \ref{prop.BL1j} and using \eqref{est.px.ptws1} instead of \eqref{sing2.proof.prop.BL1j}, we find that for sufficiently large $R'\ge R$,
	\begin{align}\label{est.pressure.recovered}
		\|p\|_{L^2(B_{R'}(0)\cap\Omega)}
		\le C(R') \|f\|_{L^2(B_R(0)\cap\Omega)}
	\end{align}
	with a constant $C(R')$ depending on $R'$.

	On the other hand, for $x$ with either $x\in\Omega_{\le 2}$ or $|x| \ge 4R$, we can connect $x$ to another point $\tilde{x}\in\Omega_{>2}$ with $|\tilde{x}|\ge 4R$ by a chain of a finite number of cubes $\{ B_{r_i}(z_i)~|~i =1,2,\cdots, N \}$ such that $B_{2r_i}(z_i) \subset \Omega$. Using 
	Remark \ref{rmk.oscp.f} on each $B_{r_i}(z_i)$, as well as \eqref{est.px.ptws1} applied to $\tilde{x}$, we see that for any $x,R$
	\begin{equation}\label{est.px.ptws2}
		|p(x)| \le C_q(x,R) \| {f}\|_{L^q(B_R(0)\cap \Omega )} 
	\end{equation}
	provided $q>3$, where $C_q(x,R)$ is a constant depending only on $q,x$, and $R$. 
		
	From \eqref{est.px.ptws1} and \eqref{est.px.ptws2}, for each fixed $x\in\Omega$, the map $f \mapsto p(x)$ is a bounded linear functional on $L^q(B_R(0)\cap \Omega)^3$. By the Riesz representation theorem, there is a unique function $\Pi(x,\cdot) \in L^{q'}(B_R(0)\cap \Omega)^3$ with $q'\in [1,\frac32)$, so that
	\begin{equation*}
		p(x) = \int_{B_R(0) \cap \Omega} \Pi(x,y) \cdot f(y) \dd y.
	\end{equation*}
	Note that the above $\Pi(x,\cdot)$ is only defined in $B_R(0)\cap\Omega$ for a fixed $x$. As $x$ and $R$ vary, we can obtain a family of such functions, which can be glued together by the uniqueness of $p$. Thus we have constructed a function $\Pi(x,y)$ defined in the entire $\Omega\times\Omega$ satisfying $\Pi(x,\cdot) \in L^{q'}_{\rm loc}(\overline{\Omega})^{3}
	$. To investigate the local integrability of $\Pi(\cdot,\cdot)$, let us fix $R>1$ and define a functional $S(f,g)$ for smooth $f,g$ supported in $B_R(0)\cap\Omega$ by 
	\begin{align*}
		S(f,g)
		&=\int_{B_R(0) \cap \Omega} p(x) g(x) \dd x\\
		&=\int_{B_R(0) \cap \Omega} \int_{B_R(0) \cap \Omega}\big(\Pi(x,y)\cdot f(y)\big) g(x) \dd y \dd x.
	\end{align*}
	From \eqref{est.pressure.recovered}, by taking a sufficiently large $R'\ge R$,
	we see that
	\begin{align*}
		|S(f,g)|
		&\le \|p\|_{L^2(B_R(0)\cap\Omega)}\|g\|_{L^2(B_R(0)\cap\Omega)}\\
		&\le 
		{C(R')}
		\|f\|_{L^2(B_R(0)\cap\Omega)} \|g\|_{L^2(B_R(0)\cap\Omega)}.
	\end{align*}
	Hence $S$ is a bounded functional on $L^2(B_{R}(0)\cap\Omega)^3\times L^2(B_{R}(0)\cap\Omega)$, which implies that
	\begin{equation}\label{est.Pi0.product}
		\int_{B_R(0) \cap \Omega} \Pi(x,\cdot) g(x) \dd x \quad \text{is in} \quad L^2(B_{R}(0)\cap\Omega)^3.
	\end{equation}

	Now we can prove that $(G,\Pi)$ satisfies Properties (i)-(iii). Property (iii) is obvious from the arguments so far. Property (ii) follows from Property (iii) combined with the Lebesgue differentiation theorem. Here we use the fact that, for all $\phi\in C^\infty_0(\Omega)^3$, the function of $y$
	\begin{equation*}
		\int_{\Omega} \Pi(x,y)(\nabla\cdot\phi)(x) \dd x
	\end{equation*}
	belongs to $L^2_{{\rm loc}}(\overline{\Omega})^3$ because of \eqref{est.Pi0.product}. The integrability of $\Pi(\cdot,y)$ in Property (i) follows from the weak form \eqref{eq.def.GPi}.
	Consequently, we have constructed the Green function $(G,\Pi)$ meeting Properties (i)-(iii).

	\medskip
	
	We should point out that in the above argument for existence, the estimate, for example of $\Pi(x, \cdot)$, is very rough, especially when $x$ is close to the boundary $\partial\Omega$. This is because the large-scale regularity of $\Pi(x,\cdot)$ is not taken into consideration. In the following, we obtain some more careful estimates of $(G,\Pi)$ by studying the equation \eqref{eq.def.GPi}.

	\subsection{Large-scale estimates of the velocity component}
	For convenience, let $G(x,y)$ and $\Pi(x,y)$ be zero-extended for both $x$ and $y$. Recall the symmetry $G(x,y) = G^t(y,x)$, where $G^t$ is the transpose of $G$. 
	Thus by definition, $G(x,y) = 0$ if either $x\in \partial \Omega$ or $y\in \partial \Omega$ and $x\neq y$. 
	Denote by $\delta(x)$ the distance from $x$ to $\partial\Omega$. 
	
	Notice that $\nabla_xG$ denotes the derivative of $G$ with respect to the first variable, i.e. 
	$$(\nabla_xG)(x,y)=(\nabla G(\cdot,y))(x),\qquad\mbox{for all}\quad (x,y)\in\Omega.$$
	Similarly, $\nabla_yG$ denotes the derivative of $G$ with respect to the second variable. The following estimates for the derivatives of $G$ are crucial.

	\begin{proposition}\label{prop.DG}
		Let $L\in(0,\infty)$ and $\Omega$ be a bumpy John domain with constant $L$ according to Definition \ref{def.John2}. The velocity component $G(x,y)$ satisfies{\rm :}
		\begin{enumerate}[label=(\roman*)]
			\item For $x_3>2$ and $y_3>2$,
			\begin{equation}\label{est.DxG.ptwise}
				|\nabla_x G(x,y)| \le C \min \bigg\{ \frac{1}{|x-y|^2}, \frac{\delta(y)}{|x-y|^3} \bigg\},
			\end{equation}
			and
			\begin{equation}\label{est.DyG.ptwise}
				|\nabla_y G(x,y)| \le C \min \bigg\{ \frac{1}{|x-y|^2}, \frac{\delta(x)}{|x-y|^3} \bigg\}.
			\end{equation}
			\item For $x_3>2$ and $y_3 < 2$ with $|x-y|>{32}$,
			\begin{equation}\label{est.DG.2-1}
				\bigg( \int_{B_{1}(y)} |\nabla_y G(x,z)|^2 \dd z \bigg)^{1/2} \le C \min \bigg\{ \frac{1}{|x-y|^2}, \frac{\delta(x)}{|x-y|^3} \bigg\},
			\end{equation} 
			and
			\begin{equation}\label{est.DG.2-2}
				\bigg( \int_{B_{1}(y)} |\nabla_x G(x,z)|^2 \dd z \bigg)^{1/2} \le  \frac{C}{|x-y|^3}.
			\end{equation}
			\item For $x_3<2$ and $y_3 >2$ with $|x-y|>{32}$,
			\begin{equation}\label{est.DG.3-1}
				\bigg( \int_{B_{1}(x)} |\nabla_x G(z,y)|^2 \dd z \bigg)^{1/2} \le C \min \bigg\{ \frac{1}{|x-y|^2}, \frac{\delta(y)}{|x-y|^3} \bigg\},
			\end{equation}
			and
			\begin{equation}\label{est.DG.3-2}
				\bigg( \int_{B_{1}(x)} |\nabla_y G(z,y)|^2 \dd z \bigg)^{1/2} \le \frac{C}{|x-y|^3}.
			\end{equation}
		\end{enumerate}
		Here $C$ depends on $L$.
	\end{proposition}
	
		Notice that $G$ and $\Pi$ are zero-extended outside $\Omega$. Therefore, the integrals above make sense even in the case when $B_1(x)$ or $B_1(y)$ intersect $\Omega^{{\rm c}}$. For the estimates concerned with the oscillation of the pressure, on the contrary, we make precise when the balls intersect the boundary; see for instance Lemma \ref{lem.osc.interior}.
	\begin{proof}[Proof of Proposition \ref{prop.DG}]
		Note that (ii) and (iii) are symmetric. While (i) is the interior estimate whose proof is similar to (ii) and (iii).
		Hence, we will only prove (ii). Since we are working on cubes, it is more convenient to define $R = |x-y|_\infty:= \max_{1\le i\le 3} |x_i - y_i|$, which is comparable to the usual distance $|x-y|$. Recall that $(G(x,\cdot), \Pi(x,\cdot))$ is a weak solution of Stoke system in $\Omega\setminus \{ x\}$. To show \eqref{est.DG.2-1}, we begin with the interior and boundary Lipschitz estimates for $G(x,\cdot)$,
		\begin{equation}\label{est.DGB12BR}
			\begin{aligned}
				\bigg( \dashint_{B_1(y)}|\nabla_y G(x,z)|^2 {\dd z} \bigg)^{1/2} & \le C\bigg( \dashint_{B_3(\hat{y})} |\nabla_y G(x,z)|^2 {\dd z} \bigg)^{1/2} \\
				& \le C \bigg( \dashint_{B_{R/2}(\hat{y})} |\nabla_y G(x,z)|^2 {\dd z} \bigg)^{1/2}.
			\end{aligned}
		\end{equation}
	where $\hat{y}$ is the projection of $y$ on $\{ y_3 = 0 \}$.
		
		 
	To proceed, let $F\in L^2(B_{R/2}(\hat{y}) \cap \Omega)^{3\times 3}$ (zero-extended to the whole of $\Omega$). Let $(u,p)$ be the weak solution of
	\begin{equation}\label{eq.uF}
		\left\{
		\begin{array}{ll}
			-\Delta u+\nabla p= \nabla\cdot F & \text{in } \Omega \\
			\nabla\cdot u=0 &\text{in }\Omega \\
			u = 0 &\text{on }\partial\Omega.
		\end{array}
		\right.
	\end{equation}
	Recall from \eqref{eq.GreenFormula} that
	\begin{equation}\label{eq.Green.DF}
	u(x) = \int_{\Omega} G(x,y) \nabla\cdot F(y) {\dd y} = -\int_{\Omega} \nabla_y G(x,y) F(y) {\dd y}.
	\end{equation}
	The energy estimate implies
	\begin{equation}\label{est.energy.L2}
		\int_{\Omega} |\nabla u|^2 \le C\int_{\Omega} |F|^2.
	\end{equation}

	Next, we estimate $|u(x)|$ and $|\nabla u(x)|$. Let $r = x_3$, which is comparable to $\delta(x)$ since $x_3 > 2$. We consider two cases: $r< R/10$ or $r> R/10$.  If $r>R/10$,
	since $F$ is supported in $B_{R/2}(\hat{y}) \cap \Omega$ which does not intersect with $B_{R/10}(x)$, we can apply the interior Lipschitz estimate to $u$ and \eqref{est.energy.L2}
	\begin{equation}\label{est.GDu1}
		|\nabla u(x)| \le C\bigg( \dashint_{B_{R/10}(x)} |\nabla u|^2 \bigg)^{1/2} \le CR^{-3/2} \bigg( \int_{\Omega} |F|^2 \bigg)^{1/2}.
	\end{equation}
	On the other hand, we apply the interior estimate, Sobolev embedding and \eqref{est.energy.L2} to obtain
\begin{equation}
	\begin{aligned}
	|u(x)| & \le C\bigg( \dashint_{B_{R/10}(x)} |u|^6 \bigg)^{1/6} \le CR^{-1/2} \bigg( \int_{\Omega} |\nabla u|^2 \bigg)^{1/2} \\
	&  \le C R^{-1/2} \bigg( \int_{\Omega} |F|^2 \bigg)^{1/2}.
	\end{aligned}
\end{equation}

If $r< R/10$, by the interior and boundary Lipschitz estimate
\begin{equation}
	\begin{aligned}
	|\nabla u(x)| & \le C\bigg( \dashint_{B_{r}(x)} |\nabla u|^2 \bigg)^{1/2} \le C\bigg( \dashint_{B_{2r}(\hat{x})} |\nabla u|^2 \bigg)^{1/2} \\
	& \le C\bigg( \dashint_{B_{R/5}(\hat{x})} |\nabla u|^2  \bigg)^{1/2} \le CR^{-3/2} \bigg( \int_{\Omega} |F|^2 \bigg)^{1/2}.
\end{aligned}
\end{equation}
Moreover, using the Poincar\'{e} inequality and the boundary Lipschitz estimate, we have
\begin{equation}\label{est.Gu4}
	\begin{aligned}
	|u(x)| & \le C\bigg( \dashint_{B_{r}(x)} |u|^2 \bigg)^{1/2} \le C\bigg( \dashint_{B_{2r}(\hat{x})} |u|^2 \bigg)^{1/2} \\
	& \le Cr \bigg( \dashint_{B_{2r}(\hat{x})} |\nabla u|^2 \bigg)^{1/2} \le Cr\bigg( \dashint_{B_{R/5}(\hat{x})} |\nabla u|^2 \bigg)^{1/2} \\
	& \le Cr R^{-3/2} \bigg( \int_{\Omega} |F|^2 \bigg)^{1/2}.
\end{aligned}
\end{equation}

From the estimates \eqref{est.GDu1} - \eqref{est.Gu4}, \eqref{eq.Green.DF} and duality, we see that
\begin{equation}\label{est.DDGBR}
	\bigg( \dashint_{B_{R/2}(\hat{y}) \cap \Omega} |\nabla_x \nabla_y G(x,z)|^2 {\dd z} \bigg)^{1/2} \le \frac{C}{R^3},
\end{equation}
and
\begin{equation}\label{est.DGBR}
	\bigg( \dashint_{B_{R/2}(\hat{y}) \cap \Omega} |\nabla_y G(x,z)|^2 {\dd z} \bigg)^{1/2} \le \frac{Cr }{R^3}.
\end{equation}
Note that \eqref{est.DGB12BR} and \eqref{est.DGBR} combined lead to \eqref{est.DG.2-1}. To see \eqref{est.DG.2-2}, notice that $(\nabla_x G(x,y), \nabla_x \Pi(x, y))$ is a weak solution in $y \in \Omega\setminus \{x\}$. Thus, we may apply \eqref{est.DDGBR},  Poincar\'{e} inequality and boundary Lipschitz estimate to obtain
\begin{equation*}
	\begin{aligned}
	\bigg( \dashint_{B_{1}(y)} |\nabla_x G(x,z)|^2 \dd z \bigg)^{1/2} &\le C \bigg( \dashint_{{B_{3}(\hat{y})}} |\nabla_x G(x,z)|^2 \dd z \bigg)^{1/2}  \\
	& \le C \bigg( \dashint_{{B_{3}(\hat{y})}} |\nabla_y \nabla_x G(x,z)|^2 \dd z \bigg)^{1/2} \\
	& \le C \bigg( \dashint_{B_{R/2}(\hat{y})} |\nabla_y \nabla_x G(x,z)|^2 \dd z \bigg)^{1/2} \le 
	\frac{C}{R^3}.
\end{aligned}
\end{equation*}
The proof of (ii) thus is complete.
	\end{proof}
	
		Analogously, we can also show the estimates for $G$ itself. The proof is left to the reader.
		\begin{proposition}\label{prop.Green.G}
		Let $L\in(0,\infty)$ and $\Omega$ be a bumpy John domain with constant $L$ according to Definition \ref{def.John2}. The velocity component $G(x,y)$ satisfies:
		\begin{enumerate}[label=(\roman*)] 
			\item For $x_3>2$ and ${y_3}>2$, 
			\begin{equation}\label{est.G.ptws}
				|G(x,y)| \le C \min \bigg\{ \frac{1}{|x-y|}, \frac{\delta(x)}{|x-y|^2}, \frac{\delta(y)}{|x-y|^2}, \frac{\delta(x){\delta(y)}}{|x-y|^3} \bigg\}.
			\end{equation} 
			\item For $x_3 > 2$ and $|x-y|>{32}$,
			\begin{equation}\label{est.G.largescale}
				\bigg( \int_{B_{1}(y)} |G(x,z)|^2 \dd z \bigg)^{1/2} \le C \min \bigg\{ \frac{1}{|x-y|}, \frac{\delta(x)}{|x-y|^2}, \frac{\delta(y)+1}{|x-y|^2}, \frac{\delta(x)(\delta(y)+1)}{|x-y|^3} \bigg\}.
			\end{equation}
			\item For $y_3 > 2$ and $|x-y|>{32}$,
			\begin{equation}\label{est.G.largescale2}
				\bigg( \int_{B_{1}(x)} |G(z,y)|^2 \dd z \bigg)^{1/2} \le C \min \bigg\{ \frac{1}{|x-y|}, \frac{\delta(y)}{|x-y|^2}, \frac{\delta(x)+1}{|x-y|^2}, \frac{\delta(y)(\delta(x)+1)}{|x-y|^3} \bigg\}.
			\end{equation}
		\end{enumerate}
		Here $C$ depends on $L$.
	\end{proposition}

	\subsection{Large-scale estimates of the pressure component} The estimates of $\Pi$ are stated as follows.

	\begin{proposition}\label{prop.pressureest}
		Let $L\in(0,\infty)$ and $\Omega$ be a bumpy John domain with constant $L$ according to Definition \ref{def.John2}. The pressure component $\Pi(x,y)$ satisfies{\rm :}
		\begin{enumerate}[label=(\roman*)]
			\item For $x_3>2$ and $y_3>2$,
			\begin{equation}\label{est.Pi-1}
				|\Pi(x,y)| \le C \min \bigg\{ \frac{1}{|x-y|^2}, \frac{\delta(y)}{|x-y|^3} \bigg\}.
			\end{equation}
			\item For $x_3<2$ and $ y_3>2$ with $|x-y|>{32}$,
			\begin{equation}\label{est.Pi-2}
				\bigg( \int_{B_{1}(x)} |\Pi(z,y)|^2 \dd z \bigg)^{1/2} \le C \min \bigg\{ \frac{1}{|x-y|^2}, \frac{\delta(y)}{|x-y|^3} \bigg\}.
			\end{equation}		
			\item For $x_3>2$ and $y_3<2$ with $|x-y|>{32}$,
			\begin{equation}\label{est.Pi-3}
				\bigg( \int_{B_{1}(y)} |\Pi(x,z)|^2 \dd z \bigg)^{1/2} \le \frac{C}{|x-y|^3}.
			\end{equation}
		\end{enumerate}
		Here $C$ depends on $L$.
	\end{proposition}
	
	\begin{proof}
		We will carry out a delicate oscillation estimate of the pressure originating from \cite{GZ19}. We first consider the estimate (i), i.e., $x_3>2$ and $y_3>2$. Consider a point $w\in \Omega$ with $w\neq y$. Let $t = |w-y|_\infty$. 
		We claim
		\begin{equation}\label{est.oscPi.Bt}
			\osc_{B_{t/4}(w)\cap \Omega_{>2}} \Pi(\cdot,y) 
			\le C\min\{ 
			\frac1{t^2},\frac{\delta(y)}{t^3}
			\}
		\end{equation}
		with $C$ independent of $t,w$, and $y$. The operator $\osc$ is defined in \eqref{def.osc}.

		We prove the above claim by considering different situations.
		If $w\in B_{y_3/2}(y)$, then $t<\frac{y_3}{2}$ and $B_{t}(w) \subset \Omega$. By the interior pressure estimate \eqref{est1.lem.osc.interior} in Lemma \ref{lem.osc.interior} and \eqref{est.DxG.ptwise},
		\begin{equation*}
			\begin{aligned}
				\osc_{B_{t/4}(w)} \Pi(\cdot,y) & \le C\bigg( \dashint_{B_{t/2}(w)} |\nabla_x G(z,y)|^2 \dd z \bigg)^{1/2} 
				\le C\min\{ 
				\frac1{t^2},\frac{\delta(y)}{t^3}
				\}.
			\end{aligned}
		\end{equation*}
		Next, if $w\notin B_{y_3/2}(y)$, we consider two subcases: (a) $|w_3| < \frac{t}{4}$; (b) $|w_3| \ge \frac{t}{4}$. Without loss of generality, we assume $t>\frac{y_3}{2}>20$.
		
		For the case (a), let $\hat{w}$ be the projection of $w$ on $\partial\R^3_+$. 
		Using the interior and boundary pressure estimates in John domains from Lemma \ref{lem.osc.interior} combined with a covering argument,
		\begin{equation}\label{Question1}
			\osc_{B_{t/4}(w)\cap \Omega_{>2}} \Pi(\cdot,y) \le C\bigg( \dashint_{B_{t/2}(\hat{w})} |\nabla_x G(z,y)|^2 \dd z \bigg)^{1/2} 
			\le C\min\{ 
			\frac1{t^2},\frac{\delta(y)}{t^3}
			\},
		\end{equation}
		where we have also used \eqref{est.DxG.ptwise} and \eqref{est.DG.3-1} in the second inequality.

		Now, for the case (b), $B_{t/4}(w)$ may be decomposed as a union of a finite number of cubes $B_{t/16}(w_i)$ with $i = 1,2,\cdots, K_0$ where $K_0$ is an absolute constant, so that $B_{t/8}(w_i)$ is contained in $\Omega_{>2}$. Thus,
		\begin{equation*}
			\begin{aligned}
				\osc_{B_{t/4}(w)} \Pi(\cdot,y) &\le \sum_{i = 1}^{K_0} \osc_{B_{t/16}(w_i)} \Pi(\cdot,y) \\
				&\le C\sum_{i = 1}^{K_0} \bigg( \dashint_{B_{t/8}(w_i)} |\nabla_x G(z,y)|^2 \dd z \bigg)^{1/2}\\
				&
				\le C\min\{ 
				\frac1{t^2},\frac{\delta(y)}{t^3}
				\},
			\end{aligned}
		\end{equation*}
		where we have used \eqref{est.DxG.ptwise}.
		Thus, the claim \eqref{est.oscPi.Bt} is proved.

		Now, by a covering argument, it is not difficult to see from \eqref{est.oscPi.Bt} that, for any $r>0$,
		\begin{equation*}
			\osc_{\Omega_{>2}\cap B_{2r}(y)\setminus\overline{B_r(y)}}
			\Pi(\cdot,y)
			=\osc_{ (B_{2r}(y)\cap \Omega_{>2})\setminus 
			\overline{(B_r(y)\cap \Omega_{>2})}} \Pi(\cdot,y) 
			\le C\min\{ 	
			\frac1{r^2},\frac{\delta(y)}{r^3}
			\}.
		\end{equation*}
		Consequently,
		\begin{equation}\label{est.Piosc>Br}
			\osc_{\Omega_{>2} \setminus \overline{B_r(y)} } \Pi(\cdot,y) \le \sum_{k = 1}^\infty 
			\osc_{\Omega_{>2}\cap B_{2^kr}(y)\setminus\overline{B_{2^{k-1}r}(y)}}
			\le C\min\{ 
			\frac1{r^2},\frac{\delta(y)}{r^3}
			\}.
		\end{equation}
		This means that for each $y$ with $y_3 > 2$, 
		there exists a function $\widehat{\Pi}(y)$ such that
		\begin{equation*}
			\lim_{|x|\to \infty,\, x_3>2} \Pi(x,y) = \widehat{\Pi}(y).
		\end{equation*}
		This convergence is uniform on any compact set in $\{y_3>2\}$. 
		We show that $\widehat{\Pi}(y) \equiv 0$. In fact, if $f\in C_0^\infty(\Omega)^3$, 
		the pressure of the Stokes equations with the source $f$ 
		is given by
		\begin{equation*}
			p(x) = \int_{\Omega} \Pi(x,y)\cdot f(y) \dd y.
		\end{equation*}
		By the definition of the Green function, $p(x) \to 0$ holds as $|x_3| \to \infty$. It follows that
		\begin{equation*}
			\int_{\Omega} \widehat{\Pi}(y) \cdot f(y) \dd y = 0.
		\end{equation*}
		This holds for any 
		$f\in C_{0}^\infty(\Omega_0)^3$ where $\Omega_0$ is a bounded open set whose closure is contained in $\{y_3>2\}$. Thus we have $\widehat{\Pi}(y) \equiv 0$.
		Therefore, \eqref{est.Piosc>Br} implies \eqref{est.Pi-1} 
		since $r$ is arbitrary.

		Next, we prove (ii). Let $x_3<2, y_3>2$, and $r:=|x-y|_\infty$. 
		Without loss of generality, it suffices to assume $r>32$.
		For such $x = (x_1,x_2,x_3)$, we pick $\tilde{x} = (x_1,x_2,3)$. Because $-1<x_3<2$ and ${|x - \tilde{x}|_\infty} < 4$, then $r-4\le {|\tilde{x} - y|_\infty} \le r + 4$ and hence by (i),
		\begin{equation}\label{est.PiAt.tx}
			|\Pi(\tilde{x},y)| \le C \min \bigg\{ \frac{1}{(r-4)^2}, \frac{\delta(y)}{(r-4)^3} \bigg\} \le C \min \bigg\{ \frac{1}{r^2}, \frac{\delta(y)}{r^3} \bigg\}.
		\end{equation}
		Next, we consider
		\begin{equation*}
			|\Pi(\tilde{x},y) - \dashint_{\Omega_{3}(\hat{x})} \Pi(\cdot,y)|,
		\end{equation*}
		where $\hat{x} = (x_1,x_2,0)$ is the projection
		and $\Omega_{3}(\hat{x})$ is the John domain between $\Omega\cap B_3(\hat{x})$ and $\Omega\cap B_6(\hat{x})$ given by Definition \ref{def.John2}. Following the argument 
		in the proof 
		of Lemma \ref{lem.osc.interior}, we can show
		\begin{equation*}
			|\Pi(\tilde{x},y) - \dashint_{\Omega_{3}(\hat{x})} \Pi(\cdot,y)| \le C\bigg( \dashint_{\Omega_{10}(\hat{x})} |\nabla_x G(z,y)|^2\dd z \bigg)^{1/2} \le C \min \bigg\{ \frac{1}{r^2}, \frac{\delta(y)}{r^3} \bigg\},
		\end{equation*}
		where we have used \eqref{est.DG.3-1} 
		as well as \eqref{est.DxG.ptwise} combined with a covering argument 
		and the fact 
		$\text{dist}(\Omega_{10}(\hat{x}), y) \approx r$ 
		in the last inequality. On the other hand, observe that $\Omega\cap B_1(x) \subset \Omega_{3}(\hat{x})$. Hence, by the Bogovskii lemma in $\Omega_{3}(\hat{x})$ and \eqref{est.DG.3-1} with a covering argument,
		\begin{equation}\label{est.PiB1-O3}
			\begin{aligned}
				& \bigg( \dashint_{B_1(x)} |\Pi(z,y) 
				- \dashint_{\Omega_{3}(\hat{x})} \Pi(\cdot,y)|^2 \dd z \bigg)^{1/2} \\
				& \le C\bigg( \dashint_{\Omega_{3}(\hat{x})} |\Pi(z,y) - \dashint_{\Omega_{3}(\hat{x})} \Pi(\cdot,y)|^2 \dd z \bigg)^{1/2} \\
				& \le C \bigg( \dashint_{\Omega_{3}(\hat{x})} |\nabla_x G(z,y)|^2\dd z \bigg)^{1/2}\\
				& \le C \min \bigg\{ \frac{1}{r^2}, \frac{\delta(y)}{r^3} \bigg\}.
			\end{aligned}
		\end{equation}
		Combining the estimates above, we obtain 
		\begin{equation*}
			\bigg( \dashint_{B_1(x)} |\Pi(z,y)|^2 \dd z \bigg)^{1/2} \le C \min \bigg\{ \frac{1}{r^2}, \frac{\delta(y)}{r^3} \bigg\}.
		\end{equation*}
		This proves \eqref{est.Pi-2}.

		Next, we use a duality method to prove (iii). Let $f\in C_0^\infty( B_1(y)\cap\Omega)^3$, zero-extended to $\Omega$, 
		and consider
		\begin{equation}\label{eq.Question2}
			\left\{
			\begin{array}{ll}
				-\Delta u+\nabla p= f\chi_{B_{1}(y)} & \text{in } \Omega \\
				\nabla\cdot u=0 &\text{in }\Omega \\
				u = 0 &\text{on }\partial\Omega.
			\end{array}
			\right.
		\end{equation}
		By definition, the solution $(u,p)$ with finite energy can be represented by \eqref{eq.GreenFormula}.
		Since we already know the estimate of $\nabla_x G$ (namely, \eqref{est.DG.2-2}),
		we have
		\begin{equation*}
			|\nabla u(x)| \le \frac{C}{|x-y|^3} \|f\|_{L^2(B_1(y))}
		\end{equation*}
		for $|x-y|_\infty>4$ with $x_3>2$. 
		By a familiar oscillation argument, we obtain
		\begin{equation}\label{est.Question3}
			|p(x)| = \bigg| \int_{\Omega} \Pi(x,y)\cdot f(y) \dd y \bigg| \le \frac{C}{|x-y|^3} \|f\|_{L^2(B_1(y))}.
		\end{equation}
		This implies \eqref{est.Pi-3}.
		\end{proof}

	\section{Proof of the iteration lemma}\label{app.it}
	\begin{proof}[Proof of Lemma \ref{lem.iteration}]
		The proof is a variation of the one in \cite{Z}. For fixed $r\in(\ep,{\frac1{16}})$, the assumption \eqref{lem.iteration.assump.f} implies 
		\begin{align*}
			\int_{r}^{{1/8}} \frac{h(t)}{t} \dd t
			&\le \int_{r}^{{1/8}} \frac{h(2t)}{t} \dd t + C_0\int_{r}^{{1/8}} \frac{H(2t)}{t} \dd t \\
			&\le \int_{2r}^{{1/4}} \frac{h(t)}{t} \dd t + C_0\int_{2r}^{{1/4}} \frac{H(t)}{t} \dd t,
		\end{align*}
		which, combined with \eqref{lem.iteration.assump.b}, \eqref{lem.iteration.assump.d} and \eqref{lem.iteration.assump.c}, gives
		\begin{align*}
			\int_{r}^{2r} \frac{h(t)}{t} \dd t
			&\le \int_{{1/8}}^{{1/4}} \frac{h(t)}{t} \dd t + C_0\int_{2r}^{{1/4}} \frac{H(t)}{t} \dd t \\
			&\le C{(\Phi(\frac12) + B_0 )}
			 + C_0\int_{r}^{{1/2}} \frac{H(t)}{t} \dd t.
		\end{align*}
		Then from \eqref{lem.iteration.assump.f} we have 
		\begin{align*}
			\int_{r}^{2r} \frac{h(t)}{t} \dd t
			&\ge \int_{r}^{2r} \frac{h(r)-C_0H(2t)}{t} \dd t \\
			&\ge \frac{h(r)}{{4}} - C_0\int_{r}^{{1/2}} \frac{H(t)}{t} \dd t.
		\end{align*}
		Therefore for $r\in(\ep,{\frac1{16}})$, we find
		\begin{align}\label{est1.proof.lem.iteration}
			h(r) \le C{(\Phi(\frac12) + B_0 )}
			+ C\int_{r}^{{1/2}} \frac{H(t)}{t} \dd t.
		\end{align}
		Let $\delta\in(0,\min\{\frac{\theta}4,{\frac1{(16)^2}}\})$ be a small number to be determined later and let us set $\ep_*=\delta^2$. We temporarily assume that $\ep\in(0,\theta\ep_*)$ in the following proof. From \eqref{lem.iteration.assump.a} we have
		\begin{align*}
			\int_{\ep/\delta}^{\delta} \frac{H(\theta t)}{t} \dd t
			&\le \frac12 \int_{\ep/\delta}^{\delta} \frac{H({2t})}{t} \dd t
			+ C_0
			\bigg(\int_{\ep/\delta}^{\delta} \big(\frac{\ep}{t}\big)^{\alpha} \frac{\Phi({16}t)}{t} \dd t
			+ B_0\int_{\ep/\delta}^{\delta}t^{\beta-1}\dd t\bigg) \\
			&\le \frac12 \int_{\ep/\delta}^{{1/2}} \frac{H(t)}{t} \dd t
			+ C_0
			\bigg(\int_{\ep/\delta}^{\delta} \big(\frac{\ep}{t}\big)^{\alpha} \frac{\Phi({16}t)}{t} \dd t
			+ \beta^{-1}B_0\bigg).
		\end{align*}
		From \eqref{lem.iteration.assump.e} and the estimate \eqref{est1.proof.lem.iteration} for $h(r)$, we have
		\begin{align*}
			\int_{\ep/\delta}^{\delta} \big(\frac{\ep}{t}\big)^{\alpha} \frac{\Phi({16}t)}{t} \dd t
			&\le
			C_0\int_{\ep/\delta}^{\delta} \big(\frac{\ep}{t}\big)^{\alpha} \frac{H({16}t)+h({16}t)}{t} \dd t \\
			&\le
			C_0\delta^{\alpha}\int_{{16}\ep/\delta}^{{16}\delta} \frac{H(t)}{t} \dd t \\
			&\quad
			+ C\bigg(\int_{\ep/\delta}^{\delta} \big(\frac{\ep}{t}\big)^{\alpha} \frac{\dd t}{t}\bigg)
			\bigg({(\Phi(\frac12) + B_0 )} + \int_{{16}\ep/\delta}^{{1/2}} \frac{H(t)}{t} \dd t\bigg) \\
			&\le
			(C_0+C_1\alpha^{-1})\delta^\alpha \int_{\ep/\delta}^{{1/2}} \frac{H(t)}{t} \dd t 
			+ C_1\alpha^{-1}\delta^{\alpha} {(\Phi(\frac12) + B_0 )}.
		\end{align*}
		Now let us choose $\delta$ sufficiently small depending on $\alpha$, $C_0$ and $C_1$ so that
		\begin{align*}
			\frac12+C_0(C_0+C_1\alpha^{-1})\delta^{\alpha}
			\le \frac34.
		\end{align*}
		Then we obtain
		\begin{align*}
			\int_{\theta\ep/\delta}^{\theta\delta} \frac{H(t)}{t} \dd t
			&\le \frac34 \int_{\ep/\delta}^{{1/2}} \frac{H(t)}{t} \dd t
			+ C{(\Phi(\frac12) + B_0 )},
			\end{align*}
		and consequently, from $\ep/\delta<\theta\delta$,
		\begin{align*}
			\int_{\theta\ep/\delta}^{\theta\delta} \frac{H(t)}{t} \dd t
			&\le 3\int_{\theta\delta}^{{1/2}} \frac{H(t)}{t} \dd t
			+ C{(\Phi(\frac12) + B_0 )}.
		\end{align*}
		Therefore from \eqref{lem.iteration.assump.b} and \eqref{lem.iteration.assump.c} we have
		\begin{equation}\label{est2.proof.lem.iteration}
			\begin{aligned}
				\int_{\theta\ep/\delta}^{{1/2}} \frac{H(t)}{t} \dd t
				&\le 4\int_{\theta\delta}^{{1/2}} \frac{H(t)}{t} \dd t 
				+ C{(\Phi(\frac12) + B_0 )} \\
				&\le C{(\Phi(\frac12) + B_0 )},
			\end{aligned}
		\end{equation}
		where we have used
		\begin{align}\label{est3.proof.lem.iteration}
			\sup_{\theta\delta\le r \le1/2} \Phi(r) 
			\le C{(\Phi(\frac12) + B_0 )}
			\end{align}
		with some constant $C$ independent of $\ep$, which is proved by applying \eqref{lem.iteration.assump.c} finitely many times. Hence, from $4\ep<\theta\ep/\delta$, the estimates \eqref{est1.proof.lem.iteration} and \eqref{est2.proof.lem.iteration} lead to, for $r\in (\theta\ep/\delta, {\frac1{16}})$,
		\begin{equation}\label{est4.proof.lem.iteration}
			\begin{aligned}
				h(r) 
				&\le C{(\Phi(\frac12) + B_0 )}
				+ C\int_{\theta\ep/\delta}^{{1/2}} \frac{H(t)}{t} \dd t \\
				&\le C{(\Phi(\frac12) + B_0 )}.
				\end{aligned}
		\end{equation}
		For $r\in(\theta\ep/\delta,{\frac1{32}})$, from \eqref{lem.iteration.assump.e}, \eqref{est2.proof.lem.iteration} and \eqref{est4.proof.lem.iteration}, we see that
		\begin{align*}
			\int_{r}^{2r}\frac{\Phi(t)}{t}\dd t
			&\le 
			C_0\int_{r}^{2r}\frac{H(t)}{t}\dd t + C_0\int_{r}^{2r}\frac{h(t)}{t}\dd t \\
			&\le 
			C{(\Phi(\frac12) + B_0 )}.
			\end{align*}
		From this, using the following inequality valid for all fixed $r\in({2\ep},1/2)$
		\begin{align*}
			\Phi(r) &\le C\big(\Phi(t) + B_0t^\beta\big), \quad t\in[r,2r],
		\end{align*}
		which is a consequence of \eqref{lem.iteration.assump.c}, we find
		\begin{align*}
			\sup_{\theta\ep/\delta\le r\le1/{32}} \Phi(r)
			\le 
			C{(\Phi(\frac12) + B_0 )}.
		\end{align*}
		Using repeatedly \eqref{lem.iteration.assump.c} finitely many times, we have
		\begin{align}\label{est5.proof.lem.iteration}
			\sup_{\ep\le r\le1/{32}}\Phi(r) 
			\le C{(\Phi(\frac12) + B_0 )}
		\end{align}
		with a constant $C$ independent of $\ep$. On the other hand, \eqref{lem.iteration.assump.b} and \eqref{est5.proof.lem.iteration} imply
		\begin{equation}\label{est6.proof.lem.iteration}
			\begin{aligned}
				\int_{\ep}^{\theta\ep/\delta}\frac{H(t)}{t}\dd t
				&\le C_0\int_{\ep}^{\theta\ep/\delta}\frac{\Phi(t)}{t}\dd t \\
				&\le C{(\Phi(\frac12) + B_0 )}.
			\end{aligned}
		\end{equation}
		Combining \eqref{est2.proof.lem.iteration}, \eqref{est3.proof.lem.iteration}, \eqref{est5.proof.lem.iteration} and \eqref{est6.proof.lem.iteration}, we obtain the assertion \eqref{est1.lem.iteration}, provided $\ep \in (0,\theta \ep_*)$. Finally, if $\ep \in (\theta\ep_*, {\frac1{48}})$, \eqref{est1.lem.iteration} is trivial by applying \eqref{lem.iteration.assump.b} and \eqref{lem.iteration.assump.c} finitely many times.
	\end{proof}

	\small 
	\bibliographystyle{abbrv}
	\bibliography{quant-NS}

\begin{thebibliography}{10}

\bibitem{APV98}
Y.~Achdou, O.~Pironneau, and F.~Valentin.
\newblock Effective boundary conditions for laminar flows over periodic rough
  boundaries.
\newblock {\em Journal of Computational Physics}, 147(1):187--218, 1998.

\bibitem{ADM06}
G.~Acosta, R.~G. Dur\'{a}n, and M.~A. Muschietti.
\newblock Solutions of the divergence operator on {J}ohn domains.
\newblock {\em Adv. Math.}, 206(2):373--401, 2006.

\bibitem{AS97}
Y.~Amirat and J.~Simon.
\newblock Riblets and drag minimization.
\newblock {\em Contemp. Math.}, 209:9--17, 1997.

\bibitem{anderson2011}
W.~Anderson and C.~Meneveau.
\newblock Dynamic roughness model for large-eddy simulation of turbulent flow
  over multiscale, fractal-like rough surfaces.
\newblock {\em Journal of Fluid Mechanics}, 679:288--314, 2011.

\bibitem{AKM16}
S.~Armstrong, T.~Kuusi, and J.-C. Mourrat.
\newblock Mesoscopic higher regularity and subadditivity in elliptic
  homogenization.
\newblock {\em Communications in Mathematical Physics}, 347(2):315--361, 2016.

\bibitem{AS16}
S.~Armstrong and C.~Smart.
\newblock Quantitative stochastic homogenization of convex integral
  functionals.
\newblock {\em Annales Scientifiques de l'Ecole Normale Superieure},
  49(2):423--481, 2016.

\bibitem{AM16}
S.~N. Armstrong and J.-C. Mourrat.
\newblock Lipschitz regularity for elliptic equations with random coefficients.
\newblock {\em Archive for Rational Mechanics and Analysis}, 219(1):255--348,
  2016.

\bibitem{AShen16}
S.~N. Armstrong and Z.~Shen.
\newblock Lipschitz estimates in almost-periodic homogenization.
\newblock {\em Communications on pure and applied mathematics},
  69(10):1882--1923, 2016.

\bibitem{AL87}
M.~Avellaneda and F.-H. Lin.
\newblock Compactness methods in the theory of homogenization.
\newblock {\em Comm. Pure Appl. Math.}, 40(6):803--847, 1987.

\bibitem{AL91}
M.~Avellaneda and F.-H. Lin.
\newblock {$L^p$} bounds on singular integrals in homogenization.
\newblock {\em Comm. Pure Appl. Math.}, 44(8-9):897--910, 1991.

\bibitem{BLTV}
G.~R. Barrenechea, P.~Le~Tallec, and F.~Valentin.
\newblock New wall laws for the unsteady incompressible navier-stokes equations
  on rough domains.
\newblock {\em ESAIM: Mathematical Modelling and Numerical Analysis},
  36(2):177--203, 2002.

\bibitem{BGV08}
A.~Basson and D.~G\'{e}rard-Varet.
\newblock Wall laws for fluid flows at a boundary with random roughness.
\newblock {\em Comm. Pure Appl. Math.}, 61(7):941--987, 2008.

\bibitem{bechert1989}
D.~W. Bechert and M.~Bartenwerfer.
\newblock The viscous flow on surfaces with longitudinal ribs.
\newblock {\em Journal of fluid mechanics}, 206:105--129, 1989.

\bibitem{BF02}
A.~Bensoussan and J.~Frehse.
\newblock {\em Regularity results for nonlinear elliptic systems and
  applications}, volume 151 of {\em Applied Mathematical Sciences}.
\newblock Springer-Verlag, Berlin, 2002.

\bibitem{bocquet2007}
L.~Bocquet and J.-L. Barrat.
\newblock Flow boundary conditions from nano-to micro-scales.
\newblock {\em Soft matter}, 3(6):685--693, 2007.

\bibitem{BB12}
M.~Bonnivard and D.~Bucur.
\newblock The uniform rugosity effect.
\newblock {\em Journal of Mathematical Fluid Mechanics}, 14(2):201--215, 2012.

\bibitem{BSG18}
M.~Bonnivard and F.~J. Su{\'a}rez-Grau.
\newblock Homogenization of a large eddy simulation model for turbulent fluid
  motion near a rough wall.
\newblock {\em Journal of Mathematical Fluid Mechanics}, 20(4):1771--1813,
  2018.

\bibitem{BM10}
D.~Bresch and V.~Milisic.
\newblock High order multi-scale wall-laws, part i: the periodic case.
\newblock {\em Quarterly of applied mathematics}, 68(2):229--253, 2010.

\bibitem{BFN10}
D.~Bucur, E.~Feireisl, and {\v{S}}.~Ne{\v{c}}asov{\'a}.
\newblock Boundary behavior of viscous fluids: Influence of wall roughness and
  friction-driven boundary conditions.
\newblock {\em Archive for rational mechanics and analysis}, 197(1):117--138,
  2010.

\bibitem{BFNW08}
D.~Bucur, E.~Feireisl, {\v{S}}.~Ne{\v{c}}asov{\'a}, and J.~Wolf.
\newblock On the asymptotic limit of the {N}avier--{S}tokes system on domains
  with rough boundaries.
\newblock {\em Journal of Differential Equations}, 244(11):2890--2908, 2008.

\bibitem{CP98}
L.~A. Caffarelli and I.~Peral.
\newblock On {$W^{1,p}$} estimates for elliptic equations in divergence form.
\newblock {\em Comm. Pure Appl. Math.}, 51(1):1--21, 1998.

\bibitem{cardillo13}
J.~Cardillo, Y.~Chen, G.~Araya, J.~Newman, K.~Jansen, and L.~Castillo.
\newblock {DNS} of a turbulent boundary layer with surface roughness.
\newblock {\em Journal of Fluid Mechanics}, 729:603--637, 2013.

\bibitem{CDFS03}
J.~Casado-D{\i}az, E.~Fern{\'a}ndez-Cara, and J.~Simon.
\newblock Why viscous fluids adhere to rugose walls:: A mathematical
  explanation.
\newblock {\em Journal of Differential Equations}, 189(2):526--537, 2003.

\bibitem{castro2007}
I.~P. Castro.
\newblock Rough-wall boundary layers: mean flow universality.
\newblock {\em Journal of Fluid Mechanics}, 585:469--485, 2007.

\bibitem{CL17}
J.~Choi and K.-A. Lee.
\newblock The {G}reen function for the {S}tokes system with measurable
  coefficients.
\newblock {\em Commun. Pure Appl. Anal.}, 16(6):1989--2022, 2017.

\bibitem{CM12}
L.~Chupin and S.~Martin.
\newblock Rigorous derivation of the thin film approximation with
  roughness-induced correctors.
\newblock {\em SIAM Journal on Mathematical Analysis}, 44(4):3041--3070, 2012.

\bibitem{CLS12}
W.~Craig, D.~Lannes, and C.~Sulem.
\newblock Water waves over a rough bottom in the shallow water regime.
\newblock {\em Annales de l'Institut Henri Poincar\'e C, Analyse non
  lin\'eaire}, 29(2):233--259, 2012.

\bibitem{DGV11}
A.-L. Dalibard and D.~G{\'e}rard-Varet.
\newblock Effective boundary condition at a rough surface starting from a slip
  condition.
\newblock {\em Journal of Differential Equations}, 251(12):3450--3487, 2011.

\bibitem{DGV17}
A.-L. Dalibard and D.~G{\'e}rard-Varet.
\newblock Nonlinear boundary layers for rotating fluids.
\newblock {\em Analysis \& PDE}, 10(1):1--42, 2017.

\bibitem{DP14}
A.-L. Dalibard and C.~Prange.
\newblock Well-posedness of the stokes--coriolis system in the half-space over
  a rough surface.
\newblock {\em Analysis \& PDE}, 7(6):1253--1315, 2014.

\bibitem{davis2003}
R.~H. Davis, Y.~Zhao, K.~P. Galvin, and H.~J. Wilson.
\newblock Solid--solid contacts due to surface roughness and their effects on
  suspension behaviour.
\newblock {\em Philosophical Transactions of the Royal Society of London.
  Series A: Mathematical, Physical and Engineering Sciences},
  361(1806):871--894, 2003.

\bibitem{dean2010}
B.~Dean and B.~Bhushan.
\newblock Shark-skin surfaces for fluid-drag reduction in turbulent flow: a
  review.
\newblock {\em Philosophical Transactions of the Royal Society A: Mathematical,
  Physical and Engineering Sciences}, 368(1929):4775--4806, 2010.

\bibitem{DDM15}
G.~Deolmi, W.~Dahmen, and S.~M\"uller.
\newblock Effective boundary conditions for compressible flows over rough
  boundaries.
\newblock {\em Mathematical Models and Methods in Applied Sciences},
  25(07):1257--1297, 2015.

\bibitem{FKV88}
E.~B. Fabes, C.~E. Kenig, and G.~C. Verchota.
\newblock The {D}irichlet problem for the {S}tokes system on {L}ipschitz
  domains.
\newblock {\em Duke Math. J.}, 57(3):769--793, 1988.

\bibitem{flack2007}
K.~A. Flack, M.~P. Schultz, and J.~S. Connelly.
\newblock Examination of a critical roughness height for outer layer
  similarity.
\newblock {\em Physics of Fluids}, 19(9):095104, 2007.

\bibitem{GMJ11}
R.~Garc\'ia-Mayoral and J.~Jim\'enez.
\newblock Drag reduction by riblets.
\newblock {\em Philosophical Transactions of the Royal Society A: Mathematical,
  Physical and Engineering Sciences}, 369(1940):1412--1427, 2011.

\bibitem{GV03}
D.~G{\'e}rard-Varet.
\newblock Highly rotating fluids in rough domains.
\newblock {\em Journal de Math{\'e}matiques Pures et Appliqu{\'e}es},
  82(11):1453--1498, 2003.

\bibitem{GV09}
D.~G{\'e}rard-Varet.
\newblock The {N}avier wall law at a boundary with random roughness.
\newblock {\em Communications in mathematical physics}, 286(1):81--110, 2009.

\bibitem{GVD06}
D.~G{\'e}rard-Varet and E.~Dormy.
\newblock Ekman layers near wavy boundaries.
\newblock {\em Journal of Fluid Mechanics}, 565:115, 2006.

\bibitem{GVH12}
D.~G\'erard-Varet and M.~Hillairet.
\newblock Computation of the drag force on a sphere close to a wall.
\newblock {\em ESAIM: Mathematical Modelling and Numerical Analysis -
  Mod\'elisation Math\'ematique et Analyse Num\'erique}, 46(5):1201--1224,
  2012.

\bibitem{GVLNR18}
D.~G{\'e}rard-Varet, C.~Lacave, T.~T. Nguyen, and F.~Rousset.
\newblock The vanishing viscosity limit for 2d navier--stokes in a rough
  domain.
\newblock {\em Journal de Math{\'e}matiques Pures et Appliqu{\'e}es},
  119:45--84, 2018.

\bibitem{GVM10}
D.~G\'{e}rard-Varet and N.~Masmoudi.
\newblock Relevance of the slip condition for fluid flows near an irregular
  boundary.
\newblock {\em Comm. Math. Phys.}, 295(1):99--137, 2010.

\bibitem{G83}
M.~Giaquinta.
\newblock {\em Multiple integrals in the calculus of variations and nonlinear
  elliptic systems}, volume 105 of {\em Annals of Mathematics Studies}.
\newblock Princeton University Press, Princeton, NJ, 1983.

\bibitem{GMbook}
M.~Giaquinta and L.~Martinazzi.
\newblock {\em An introduction to the regularity theory for elliptic systems,
  harmonic maps and minimal graphs}, volume~11 of {\em Appunti. Scuola Normale
  Superiore di Pisa (Nuova Serie) [Lecture Notes. Scuola Normale Superiore di
  Pisa (New Series)]}.
\newblock Edizioni della Normale, Pisa, second edition, 2012.

\bibitem{GM82}
M.~Giaquinta and G.~Modica.
\newblock Nonlinear systems of the type of the stationary {N}avier-{S}tokes
  system.
\newblock {\em J. Reine Angew. Math.}, 330:173--214, 1982.

\bibitem{Giga83}
Y.~Giga.
\newblock Time and spatial analyticity of solutions of the {N}avier-{S}tokes
  equations.
\newblock {\em Comm. Partial Differential Equations}, 8(8):929--948, 1983.

\bibitem{GNO14}
A.~Gloria, S.~Neukamm, and F.~Otto.
\newblock A regularity theory for random elliptic operators.
\newblock {\em arXiv preprint arXiv:1409.2678}, 2014.

\bibitem{GNO15}
A.~Gloria, S.~Neukamm, and F.~Otto.
\newblock Quantification of ergodicity in stochastic homogenization: optimal
  bounds via spectral gap on glauber dynamics.
\newblock {\em Inventiones mathematicae}, 199(2):455--515, 2015.

\bibitem{GX17}
S.~Gu and Q.~Xu.
\newblock Optimal boundary estimates for {S}tokes systems in homogenization
  theory.
\newblock {\em SIAM J. Math. Anal.}, 49(5):3831--3853, 2017.

\bibitem{GZ19}
S.~Gu and J.~Zhuge.
\newblock Periodic homogenization of {G}reen's functions for {S}tokes systems.
\newblock {\em Calc. Var. Partial Differential Equations}, 58(3):Paper No. 114,
  46, 2019.

\bibitem{GZ}
S.~{Gu} and J.~{Zhuge}.
\newblock {Large-scale Regularity of Nearly Incompressible Elasticity in
  Stochastic Homogenization}.
\newblock {\em arXiv e-prints}, page arXiv:2004.14568, Apr. 2020.

\bibitem{Higaki16}
M.~Higaki.
\newblock {N}avier wall law for nonstationary viscous incompressible flows.
\newblock {\em Journal of Differential Equations}, 260(10):7358--7396, 2016.

\bibitem{HP}
M.~Higaki and C.~Prange.
\newblock Regularity for the stationary {N}avier-{S}tokes equations over bumpy
  boundaries and a local wall law.
\newblock {\em Calc. Var. Partial Differential Equations}, 59(4):Paper No. 131,
  46, 2020.

\bibitem{HK07}
S.~Hofmann and S.~Kim.
\newblock The {G}reen function estimates for strongly elliptic systems of
  second order.
\newblock {\em Manuscripta Math.}, 124(2):139--172, 2007.

\bibitem{izard2014}
E.~Izard, T.~Bonometti, and L.~Lacaze.
\newblock Modelling the dynamics of a sphere approaching and bouncing on a wall
  in a viscous fluid.
\newblock {\em Journal of Fluid Mechanics}, 747:pp--422, 2014.

\bibitem{JM01}
W.~J{\"a}ger and A.~Mikeli{\'c}.
\newblock On the roughness-induced effective boundary conditions for an
  incompressible viscous flow.
\newblock {\em Journal of Differential Equations}, 170(1):96--122, 2001.

\bibitem{JM03}
W.~J{\"a}ger and A.~Mikeli{\'c}.
\newblock Couette flows over a rough boundary and drag reduction.
\newblock {\em Communications in Mathematical Physics}, 232(3):429--455, 2003.

\bibitem{jimenez2004}
J.~Jim\'enez.
\newblock Turbulent flows over rough walls.
\newblock {\em Annual Review of Fluid Mechanics}, 36(1):173--196, 2004.

\bibitem{John61}
F.~John.
\newblock Rotation and strain.
\newblock {\em Comm. Pure Appl. Math.}, 14:391--413, 1961.

\bibitem{joseph2001}
G.~G. Joseph, R.~Zenit, M.~L. Hunt, and A.~M. Rosenwinkel.
\newblock Particle-wall collisions in a viscous fluid.
\newblock {\em Journal of Fluid Mechanics}, 433:329--346, 2001.

\bibitem{joughin19}
I.~Joughin, B.~E. Smith, and C.~G. Schoof.
\newblock Regularized coulomb friction laws for ice sheet sliding: Application
  to pine island glacier, antarctica.
\newblock {\em Geophysical Research Letters}, 46(9):4764--4771, 2019.

\bibitem{KLS13}
C.~E. Kenig, F.~Lin, and Z.~Shen.
\newblock Homogenization of elliptic systems with {N}eumann boundary
  conditions.
\newblock {\em J. Amer. Math. Soc.}, 26(4):901--937, 2013.

\bibitem{KLS14}
C.~E. Kenig, F.~Lin, and Z.~Shen.
\newblock Periodic homogenization of {G}reen and {N}eumann functions.
\newblock {\em Comm. Pure Appl. Math.}, 67(8):1219--1262, 2014.

\bibitem{KP15}
C.~E. Kenig and C.~Prange.
\newblock Uniform lipschitz estimates in bumpy half-spaces.
\newblock {\em Archive for Rational Mechanics and Analysis}, 216(3):703--765,
  2015.

\bibitem{KP18}
C.~E. Kenig and C.~Prange.
\newblock Improved regularity in bumpy lipschitz domains.
\newblock {\em Journal de Math{\'e}matiques Pures et Appliqu{\'e}es},
  113:1--36, 2018.

\bibitem{lauga2007}
E.~Lauga, M.~Brenner, and H.~Stone.
\newblock Microfluidics: The no-slip boundary condition.
\newblock In {\em Springer Handbooks}, pages 1219--1240. Springer, 2007.

\bibitem{lee2005}
S.-J. Lee and Y.-G. Jang.
\newblock Control of flow around a naca 0012 airfoil with a micro-riblet film.
\newblock {\em Journal of Fluids and Structures}, 20(5):659--672, 2005.

\bibitem{MS79}
O.~Martio and J.~Sarvas.
\newblock Injectivity theorems in plane and space.
\newblock {\em Ann. Acad. Sci. Fenn. Ser. A I Math.}, 4(2):383--401, 1979.

\bibitem{Mas67}
K.~Masuda.
\newblock On the analyticity and the unique continuation theorem for solutions
  of the {N}avier-{S}tokes equation.
\newblock {\em Proc. Japan Acad.}, 43:827--832, 1967.

\bibitem{minchew20}
B.~Minchew and I.~Joughin.
\newblock Toward a universal glacier slip law.
\newblock {\em Science}, 368(6486):29--30, 2020.

\bibitem{narteau2001}
C.~Narteau, J.~Le~Mou{\"e}l, J.~Poirier, E.~Sep{\'u}lveda, and M.~Shnirman.
\newblock On a small-scale roughness of the core--mantle boundary.
\newblock {\em Earth and Planetary Science Letters}, 191(1-2):49--60, 2001.

\bibitem{pu2016}
X.~Pu, G.~Li, and H.~Huang.
\newblock Preparation, anti-biofouling and drag-reduction properties of a
  biomimetic shark skin surface.
\newblock {\em Biology Open}, 5(4):389--396, 2016.

\bibitem{schlichting2016}
H.~Schlichting and K.~Gersten.
\newblock {\em Boundary-layer theory}.
\newblock Springer, 2016.

\bibitem{schultz2007}
M.~P. Schultz and K.~A. Flack.
\newblock The rough-wall turbulent boundary layer from the hydraulically smooth
  to the fully rough regime.
\newblock {\em Journal of Fluid Mechanics}, 580:381--405, 2007.

\bibitem{Shen17}
Z.~Shen.
\newblock Boundary estimates in elliptic homogenization.
\newblock {\em Analysis \& PDE}, 10(3):653--694, 2017.

\bibitem{Sh}
Z.~Shen.
\newblock {\em Periodic homogenization of elliptic systems}, volume 269 of {\em
  Operator Theory: Advances and Applications}.
\newblock Birkh\"{a}user/Springer, Cham, 2018.
\newblock Advances in Partial Differential Equations (Basel).

\bibitem{S20}
Z.~Shen.
\newblock Weighted ${L}^2$ estimates for elliptic homogenization in lipschitz
  domains, 2020.

\bibitem{smart1989}
J.~R. Smart and D.~T. Leighton~Jr.
\newblock Measurement of the hydrodynamic surface roughness of noncolloidal
  spheres.
\newblock {\em Physics of Fluids A: Fluid Dynamics}, 1(1):52--60, 1989.

\bibitem{squire16}
D.~T. Squire, C.~Morrill-Winter, N.~Hutchins, M.~P. Schultz, J.~C. Klewicki,
  and I.~Marusic.
\newblock Comparison of turbulent boundary layers over smooth and rough
  surfaces up to high reynolds numbers.
\newblock {\em Journal of Fluid Mechanics}, 795:210--240, 2016.

\bibitem{townsend1980}
A.~Townsend.
\newblock {\em The structure of turbulent shear flow}.
\newblock Cambridge university press, 1980.

\bibitem{varnik2007}
F.~Varnik, D.~Dorner, and D.~Raabe.
\newblock Roughness-induced flow instability: a lattice boltzmann study.
\newblock {\em Journal of Fluid Mechanics}, 573:191, 2007.

\bibitem{waheed2016}
S.~Waheed, J.~M. Cabot, N.~P. Macdonald, T.~Lewis, R.~M. Guijt, B.~Paull, and
  M.~C. Breadmore.
\newblock 3d printed microfluidic devices: enablers and barriers.
\newblock {\em Lab on a Chip}, 16(11):1993--2013, 2016.

\bibitem{WZ14}
W.~Wei and Z.~Zhang.
\newblock {$L^p$} resolvent estimates for constant coefficient elliptic systems
  on {L}ipschitz domains.
\newblock {\em J. Funct. Anal.}, 267(9):3262--3293, 2014.

\bibitem{Z}
J.~Zhuge.
\newblock Regularity theory of elliptic systems in $\varepsilon$-scale flat
  domains.
\newblock {\em Advances in Mathematics}, 379:107566, 2021.

\end{thebibliography}
	
\end{document}